\documentclass[com,crcready]{iosart2x}

\usepackage{url}
\usepackage{hyperref}

\usepackage{amsmath}
\usepackage{amssymb}

\usepackage{amsthm}

\theoremstyle{plain}
\newtheorem{theorem}{Theorem}
\newtheorem{lemma}[theorem]{Lemma}
\newtheorem{corollary}[theorem]{Corollary}
\newtheorem{proposition}[theorem]{Proposition}
\newtheorem{definition}[theorem]{Definition}

\newtheorem{postulate}{Postulate}
\newtheorem{thesis}{Thesis}

\theoremstyle{definition}
\newtheorem{example}[theorem]{Example}

\renewcommand{\labelenumi}{(\roman{enumi})}

\newcommand{\abs}[1]{\left\lvert#1\right\rvert}

\DeclareMathOperator{\Dom}{dom}

\newcommand{\rest}[2]{#1\!\!\restriction_{#2}}
\newcommand{\reste}[2]{#1\restriction_{#2}}
\newcommand{\osg}[1]{\left[#1\right]^{\prec}}

\newcommand{\N}{\mathbb{N}}%
\newcommand{\Q}{\mathbb{Q}}%
\newcommand{\R}{\mathbb{R}}%
\newcommand{\X}{\{0,1\}^*}%
\newcommand{\XI}{\{0,1\}^\infty}%

\newcommand{\Bm}[2]{\lambda_{#1}\left(#2\right)}

\newcommand{\ket}[1]{| #1 \rangle}
\newcommand{\bra}[1]{\langle #1 |}

\newcommand{\PS}{\mathbb{P}}%
\newcommand{\charaps}[2]{C\!\left(#1,#2\right)}
\newcommand{\charapse}[2]{C\left(#1,#2\right)}
\newcommand{\chara}[2]{\mathrm{C}_{#1}\left(#2\right)}
\newcommand{\cond}[2]{\mathrm{Filtered}_{#1}\left(#2\right)}

\newcommand{\Projl}[1]{\mathrm{Proj}_1\!\left(#1\right)}

\newcommand{\jump}[1]{\ensuremath{[\![#1]\!]} }

\pubyear{0000}
\volume{0}
\firstpage{1}
\lastpage{1}

\begin{document}

\begin{frontmatter}

\title{An operational characterization of the notion of probability by algorithmic randomness and its applications\thanks{%
This work
is a substantial extension of two preliminary papers of the author,
entitled:
``An operational characterization of the notion of probability by algorithmic randomness,''
which appeared in the Proceedings of the 37th Symposium on Information Theory and its Applications (SITA2014),
5.4.1, pp.~389--394, December 9-12, 2014, Unazuki, Toyama, Japan,
and:
``An operational characterization of the notion of probability by algorithmic randomness and
its application to cryptography,''
which appeared in the Proceedings of the 32nd Symposium on Cryptography and Information Security (SCIS2015),
2D4-3, January 20-23, 2015, Kokura, Japan.
The both
preliminary
papers are available on the author's website at \url{http://www2.odn.ne.jp/tadaki/}%
}}
\runtitle{An operational characterization of the notion of probability by algorithmic randomness and its applications}

\author{\inits{K.}\fnms{Kohtaro} \snm{Tadaki}\ead[label=e1]{tadaki@isc.chubu.ac.jp}}
\address{Department of Computer Science,
\orgname{Chubu University},
1200 Matsumoto-cho, Kasugai-shi, Aichi 487-8501, \cny{Japan}\printead[presep={\\}]{e1}}

\begin{abstract}
The notion of probability plays an important role in almost all areas of science and technology.
In modern mathematics, however, probability theory means nothing other than measure theory,
and the operational characterization of the notion of probability is not established yet.
In this paper, based on the toolkit of algorithmic randomness we present
an operational characterization of the notion of probability, called an \emph{ensemble}.
Algorithmic randomness, also known as algorithmic information theory,
is a field of mathematics which enables us to consider the randomness of an individual infinite sequence.
We use the notion of Martin-L\"of randomness with respect to Bernoulli measure
to present the operational characterization.
As the first step of the research of this line,
in this paper we
consider
the case of finite probability space, i.e.,
the case where the sample space of the underlying probability space is finite, for simplicity.
We give a natural operational characterization of the notion of conditional probability
in terms of ensemble,
and give equivalent characterizations of the notion of independence between two events based on it.
Furthermore,
we give equivalent characterizations of the notion of independence
of
an arbitrary number of events/random variables in terms of ensembles.
Moreover,
we show that the independence
of
events/random variables is equivalent
to the independence in the sense of van Lambalgen's Theorem,
in the case where the underlying finite probability space is computable.
In the paper we make applications of our framework to information theory, cryptography, and
the simulation of a biased coin using fair coins,
in order to demonstrate the wide applicability of our framework to the general areas of science and technology.
\end{abstract}


\end{frontmatter}

\section{Introduction}

The notion of probability plays an important role in almost all areas of science and technology.
In modern mathematics, however, probability theory means nothing other than \emph{measure theory},
and the operational characterization of the notion of probability is not established yet.
In this paper, based on the toolkit of \emph{algorithmic randomness}
we present an operational characterization of the notion of probability.
Algorithmic randomness
is a field of mathematics which enables us to consider the randomness of an individual infinite sequence.
We use the notion of
\emph{Martin-L\"of randomness with respect to Bernoulli measure}
to present the operational characterization.

To clarify our motivation
and standpoint,
and the meaning of the operational characterization,
let us consider
a familiar example of a probabilistic phenomenon.
We
here
consider the repeated throwings of a fair die.
In this probabilistic phenomenon,
as throwings progressed,
a specific infinite sequence such as
\begin{equation*}
  3,5,6,3,4,2,2,3,6,1,5,3,5,4,1,\dotsc\dotsc\dotsc
\end{equation*}
is being generated,
where each number is the outcome of the corresponding throwing of the die.
Then the following
naive
question
may
arise naturally. 

\begin{quote}
\textbf{Question}:
What property should this infinite sequence
satisfy as a
probabilistic
phenomenon?
\end{quote}

In this paper we try to answer this question.
We characterize the notion of probability
as
an infinite sequence of outcomes in a probabilistic phenomenon, which has a \emph{specific mathematical property}.
We call such an infinite sequence of outcomes
the \emph{operational characterization of the notion of probability}.
As the specific mathematical property, in this paper we adopt
the notion of \emph{Martin-L\"of randomness with respect to Bernoulli measure},
a notion in algorithmic randomness.

We put forward
this proposal
as a thesis (see Thesis~\ref{thesis} in Section~\ref{OPT}).
We
check
the validity of
the thesis
in detail,
based on
our intuitive understanding of the notion of probability.
In particular, we show that the thesis is properly consistent with the notion of probability in quantum mechanics.
We do this by demonstrating the fact that an elementary event with probability one occurs certainly in quantum mechanics,
on a variety of grounds.
We then
characterize \emph{equivalently}
the basic notions in probability theory in terms of the operational characterization.
Namely, we equivalently characterize the notion of the \emph{independence} of random variables/events
in terms of the operational characterization,
and represent the notion of \emph{conditional probability}
in terms of the operational characterization in a natural way.
The existence of these equivalent characterizations confirms further the validity of the thesis.

\subsection{Historical background}

In the past century,
there was a comprehensive attempt to provide
an operational characterization of the notion of probability.
Namely, von Mises
developed
a mathematical theory of repetitive events which was aimed at reformulating
the theory of probability and statistics
based on an operational characterization of the notion of probability \cite{vM57,vM64}.
In a series of his comprehensive works which began in 1919,
von Mises developed this theory and, in particular,
introduced the notion of \emph{collective} as a mathematical idealization of a long sequence of
outcomes of experiments or observations repeated under a set of invariable conditions,
such as the repeated tossings of a coin or of a pair of dice.

The collective plays a role as an operational characterization of the notion of probability,
and is an infinite sequence of sample points in the sample space of a probability space.
The notion of  collective is defined
by means of
the notion of \emph{place selection},
which is a function from the set of finite strings
of the sample points
to the set~$\{\mathrm{YES}, \mathrm{NO}\}$.
Given an infinite sequence of the sample points,
the place selection is used for selecting positions of the infinite sequence in a certain manner
to form a subsequence of it.
Then, for aiming to introduce a randomness property to the collective,
von Mises defined a collective
to be an infinite sequence of the sample points such that
the law of large numbers with a fixed limit value holds for all its subsequences selected by place selections.
In this definition of the collective, if all possible place selections are allowed,
the definition becomes empty.
Thus, the class of place selections used in the definition is assumed to
a
countable set. 
Actually, Wald~\cite{Wa36,Wa37}
showed that,
for every countable class of place selection,
there exist infinitely many infinite sequences
which are
classified as
a collective
according to the definition above. 
However, the arbitrariness of the choice of a countable class of place selections remains left, and
von Mises would seem to propose to choose this
countable
class,
depending on a particular problem in probability theory to be solved.
In order to remove
this arbitrariness,
in 1940 Church~\cite{Ch40}
proposed that a place selection
be
a computable function.
Actually, all of such functions form a countable class.
In 1939, however, Ville~\cite{Vi39} revealed
one of the inescapable defects
of the notion of collective.
Namely, he showed that,
no matter how
a
countable collection of place selections is chosen,
there exists an infinite sequence of the sample points
which must be classified
into
a collective according to the definition above,
and which
cannot be regarded as random according to our intuition.
Apart from Ville's result,
the collective has,
in the first place,
an \emph{intrinsic defect} that
\emph{it cannot exclude the possibility that an event with probability zero may occur}.
For the
historical
development of the theory of collectives from the point of view of the definition of randomness,
see Downey and Hirschfeldt~\cite{DH10}.

In 1966, Martin-L\"of \cite{M66} introduced the definition of random sequences,
which is called \emph{Martin-L\"of randomness} nowadays, and
plays a central role in the recent development of algorithmic randomness.
At the same time,
he introduced the notion of \emph{Martin-L\"of randomness with respect to Bernoulli measure} \cite{M66}.
He then pointed out that this notion overcomes the defect of the collective in the sense of von Mises, and
this can be regarded precisely as the collective which von Mises wanted to define.
However, he did not develop probability theory
based on Martin-L\"of random
sequences
with respect to Bernoulli measure.

Algorithmic randomness is a field of mathematics which studies the definitions of random sequences and
their property
(see
Nies~\cite{N09} and Downey and Hirschfeldt~\cite{DH10}
for the recent developments of
the field).
However, the recent research on algorithmic randomness would seem only interested in
the notions of randomness themselves and their
interrelation,
and not seem to have
made an attempt
to develop probability theory based on Martin-L\"of randomness with respect to Bernoulli measure in an operational manner so far.

\subsection{Contribution of the paper}

The subject of this paper is to make such an attempt.
Namely, in this paper we present an operational characterization of the notion of probability
based on Martin-L\"of randomness with respect to Bernoulli measure.
We call it an \emph{ensemble}, instead of collective for distinction.
The name ``ensemble'' comes from physics, in particular, from quantum mechanics and statistical mechanics.
In Thesis~\ref{thesis} below
we propose to identify it with
an infinite sequence of outcomes
resulting from
the infinitely repeated trials in a probabilistic phenomenon.
We show that
the
ensemble has enough properties to regard it as an operational characterization of
the notion of probability
from the point of view of our intuitive understanding of the notion of probability.

Actually, we give a natural operational characterization of the notion of conditional probability
in terms of ensemble,
and give equivalent characterizations of the notion of independence between two events based on it.
Furthermore,
we give equivalent characterizations of the notion of independence of an arbitrary number of events/random variables in terms of ensembles.
Moreover,
we show that the independence
of
events/random variables is equivalent
to the independence in the sense of van Lambalgen's Theorem \cite{vL87},
in the case where the underlying
probability space is \emph{computable}.

As the first step of the research of this line,
in this paper
we
consider
only
the case of \emph{finite probability space}, i.e.,
the case where the sample space of the underlying probability space is finite, for simplicity.
The investigation of the case of general probability spaces is
reported in the sequels to the paper,
such as Tadaki~\cite{T19arXiv}
which
especially investigates the case of general discrete probability spaces
whose sample space is countably infinite.

We emphasize that
\emph{the Bernoulli measure which we consider in this paper
is quite arbitrary and
therefore
is not required to be computable at all}
(except for the results related to van Lambalgen's Theorem
given in Section~\ref{FENICFPS},
an effectivization of the law of large numbers for an ensemble
given in Section~\ref{Section:EffectiveLLN},
and the applications of our framework given in
Sections~\ref{subsec:AC} and~\ref{AFBC},
where the computability of finite probability spaces
is
thought to be a natural requirement for the applications),
whereas
the measures considered in algorithmic randomness so far are usually computable.
Hence,
the
central
results in this paper hold for
an \emph{arbitrary}
finite probability space.

For instance,
in order to confirm the validity of Thesis~\ref{thesis},
we show that the law of large numbers holds for an \emph{arbitrary} ensemble
in Theorem~\ref{LLN}
in Section~\ref{Subsec:The law of large numbers}.
In that theorem,
the underlying finite probability space is \emph{quite arbitrary}, and therefore
is not required to be computable at all, in particular.
In this regard, however,
it would be
interesting
to investigate
what happens when we venture to impose computability restrictions on
the underlying finite probability space.
In Section~\ref{Section:EffectiveLLN}
we show that the computability of the underlying finite probability space
leads to
the
\emph{effectivization} of the law of large numbers.

Finally,
we make applications of our framework to
the general areas of science and technology.
In this paper,
as examples of the fields for the applications,
we choose information theory, cryptography, and the simulation of a biased coin using fair coins.
Thereby, we demonstrate
how properly our framework works in practical problems
in the general areas of science and technology.

Modern probability theory originated from the \emph{axiomatic approach} to probability theory,
introduced by Kolmogorov~\cite{K50} in 1933,
where the probability theory is precisely \emph{measure theory}.
One of the important roles of modern probability theory is, of course, in its applications to
the general areas of science and technology.
As we have already pointed out,
however,
an operational characterization of the notion of probability is still missing in modern probability theory.
Thus, when we apply the results of modern probability theory,
we have no choice but to
make such applications \emph{thoroughly based on our intuition without formal means}.

The aim of this paper is to try to fill in this gap between modern probability theory and its applications.
We present the operational characterization of the notion of probability as
a \emph{rigorous interface} between theory and practice,
without appealing to our intuition for filling in the gap.
No matter what,
in this
work
we \emph{keep} modern probability theory \emph{in its original form} without any modifications,
and propose the operational characterization of the notion of probability
as an \emph{additional mathematical structure} to it,
which
provides modern probability theory
with more comprehensive and rigorous 
opportunities for applications.

\subsection{Organization of the paper}

The paper is organized as follows.
We begin in Section~\ref{preliminaries} with some preliminaries to measure theory,
computability theory, and algorithmic randomness.
In Section~\ref{FPS}, we introduce the notion of finite probability space
on which the operational characterization of the notion of probability is presented.
Based on the notion of  finite probability space we then introduce
the notion of Martin-L\"of randomness with respect to Bernoulli measure in Section~\ref{MLRP}.

In Section~\ref{OPT} we
introduce the notion of ensemble,
and put forward a thesis
which states to identify the ensemble
as an operational characterization of the notion of probability.
We then
check
the validity of the thesis.
In Section~\ref{CPITW} we start to construct our framework
by characterizing operationally the notions of conditional probability and
the independence between two events,
in terms of ensembles.
We then characterize operationally the notion of the independence
of an arbitrary number of events/random variables
in terms of ensembles in Section~\ref{IANERV}.
In Section~\ref{FENICFPS} we show that
the independence notions, introduced in the preceding sections, are further equivalent to
the notion of the independence in the sense of van Lambalgen's Theorem,
in the case where the underlying finite probability space is computable,
by generalizing van Lambalgen's Theorem
over our framework.
Thus we show that the three independence notions, considered in this paper,
are all equivalent in this case.
In Section~\ref{Section:EffectiveLLN} we show that we
can accomplish
the
effectivization of the law of large numbers for an arbitrary ensemble
if
we venture to impose
computability restrictions upon
the underlying finite probability space.

In Section~\ref{Applications}
we make applications of our framework to information theory, cryptography,
and the simulation of a biased coin using fair coins.
We then mention an application of our framework to quantum mechanics,
which has already been developed by
a series of our works,
such as Tadaki~\cite{T18arXiv,T20Kokyuroku}.
We conclude this paper in Section~\ref{Concluding} with a mention of the next step of the research,
i.e., an extension of our framework over
general discrete probability spaces whose sample space is
\emph{countably infinite},
which has already been developed by Tadaki~\cite{T19arXiv}.

\section{Preliminaries}
\label{preliminaries}

\subsection{Basic notation and definitions}
\label{basic notation}

We start with some notation about numbers and strings which will be used in this paper.
$\#S$ is the cardinality of $S$ for any set $S$.
$\N=\left\{0,1,2,3,\dotsc\right\}$ is the set of \emph{natural numbers},
and $\N^+$ is the set of \emph{positive integers}.
$\Q$ is the set of \emph{rationals}, and $\R$ is the set of \emph{reals}.
For any $a\in\R$, as usual, $\lceil a\rceil$ denotes the smallest integer greater than or equal to $a$.

An \emph{alphabet} is
a non-empty
finite set.
Let $\Omega$ be an arbitrary alphabet throughout the rest of this section.
A \emph{finite string over $\Omega$} is a finite sequence of elements from the alphabet $\Omega$.
We use $\Omega^*$ to denote the set of all finite strings over $\Omega$,
which contains the \emph{empty string} denoted by $\lambda$.
We use $\Omega^+$ to denote the set $\Omega^*\setminus\{\lambda\}$.
For any $\sigma\in\Omega^*$, $\abs{\sigma}$ is the \emph{length} of $\sigma$.
Therefore $\abs{\lambda}=0$.
For any $\sigma\in\Omega^+$ and $k\in\N^+$ with $k\le\abs{\sigma}$,
we use $\sigma(k)$ to denote the $k$th element in $\sigma$.
Therefore, we have $\sigma=\sigma(1)\sigma(2)\dots\sigma(\abs{\sigma})$
for every $\sigma\in\Omega^+$.
For any $n\in\N$,
we use $\Omega^n$ and $\Omega^{\le n}$ to denote the sets
$\{\,x\mid x\in\Omega^*\;\&\;\abs{x}=n\}$ and $\{\,x\mid x\in\Omega^*\;\&\;\abs{x}\le n\}$,
respectively.
A subset $S$ of $\Omega^*$ is called
\emph{prefix-free}
if no string in $S$ is a prefix of another string in $S$.
A \emph{minimal} string in a subset $S$ of $\Omega^*$ is
a finite string in $S$ whose no proper prefix is in $S$.
For any $\sigma,\tau\in\Omega^*$, we say that $\sigma$ is \emph{incompatible with} $\tau$ if
$\sigma$ is not a prefix of $\tau$ and
moreover
$\tau$ is not a prefix of $\sigma$.

An \emph{infinite sequence over $\Omega$} is an infinite sequence of elements from the alphabet $\Omega$,
where the sequence is infinite to the right but finite to the left.
We use $\Omega^\infty$ to denote the set of all infinite sequences over $\Omega$.

Let $\alpha\in\Omega^\infty$.
For any $n\in\N$
we denote by $\rest{\alpha}{n}\in\Omega^*$ the first $n$ elements
in the infinite sequence $\alpha$,
and
for any $n\in\N^+$ we denote
by $\alpha(n)$ the $n$th element in $\alpha$.
Thus, for example, $\rest{\alpha}{4}=\alpha(1)\alpha(2)\alpha(3)\alpha(4)$, and $\rest{\alpha}{0}=\lambda$.

For any $S\subset\Omega^*$, the set
$\{\alpha\in\Omega^\infty\mid\exists\,n\in\N\;\rest{\alpha}{n}\in S\}$
is denoted by $\osg{S}$.
Note that (i) $\osg{S}\subset\osg{T}$ for every $S\subset T\subset\Omega^*$, and
(ii) for every set $S\subset\Omega^*$ there exists a prefix-free set $P\subset\Omega^*$ such that
$\osg{S}=\osg{P}$.
For any $\sigma\in\Omega^*$, we denote by $\osg{\sigma}$ the set $\osg{\{\sigma\}}$, i.e.,
the set of all infinite sequences over $\Omega$ extending $\sigma$.
Therefore $\osg{\lambda}=\Omega^\infty$.

For any function $f$, the \emph{domain of definition} of $f$ is denoted by $\Dom f$.

\subsection{Measure theory}
\label{MR}

We briefly review measure theory according to Nies~\cite[Section 1.9]{N09}.
See also Billingsley~\cite{B95}
for measure theory in general.

A real-valued function $\mu$ defined on the class of all subsets of $\Omega^\infty$ is called
an \emph{outer measure on $\Omega^\infty$} if the following conditions hold.
\begin{enumerate}
\item $\mu\left(\emptyset\right)=0$,
\item $\mu\left(\mathcal{C}\right)\le\mu\left(\mathcal{D}\right)$
  for every subsets $\mathcal{C}$ and $\mathcal{D}$ of $\Omega^\infty$
  with $\mathcal{C}\subset\mathcal{D}$, and
\item $\mu\left(\bigcup_{i}\mathcal{C}_i\right)\le\sum_{i}\mu\left(\mathcal{C}_i\right)$
  for every sequence $\{\mathcal{C}_i\}_{i\in\N}$ of subsets of $\Omega^\infty$.
\end{enumerate}
A \emph{probability measure representation over $\Omega$} is
a function $r\colon\Omega^*\to[0,1]$ such that
\begin{enumerate}
  \item $r(\lambda)=1$ and
  \item for every $\sigma\in\Omega^*$ it holds that
    \begin{equation}\label{pmr}
       r(\sigma)=\sum_{a\in\Omega}r(\sigma a).
    \end{equation}
\end{enumerate}
A probability measure representation $r$ over $\Omega$ \emph{induces}
an outer measure $\mu_r$ on $\Omega^\infty$ in the following manner:
A subset $\mathcal{R}$ of $\Omega^\infty$ is \emph{open} if
$\mathcal{R}=\osg{S}$ for some $S\subset\Omega^*$.
Let
$r$
be an arbitrary probability measure representation over $\Omega$. 
For each open subset $\mathcal{A}$ of $\Omega^\infty$, we define $\mu_r(\mathcal{A})$ by
\begin{equation*}
  \mu_r(\mathcal{A}):=\sum_{\sigma\in E}r(\sigma),
\end{equation*}
where $E$ is a prefix-free subset of $\Omega^*$ with $\osg{E}=\mathcal{A}$.
Due to the equality \eqref{pmr}
the sum is independent of the choice of  the prefix-free set $E$,
and therefore the value $\mu_r(\mathcal{A})$ is well-defined.
Then, for any subset $\mathcal{C}$ of $\Omega^\infty$, we define $\mu_r(\mathcal{C})$ by
\begin{equation*}
  \mu_r(\mathcal{C}):=
  \inf\{\mu_r(\mathcal{A})\mid
  \mathcal{C}\subset\mathcal{A}\text{ \& $\mathcal{A}$ is an open subset of $\Omega^\infty$}\}.
\end{equation*}
We can then show that $\mu_r$ is an
outer measure
on $\Omega^\infty$ such that
$\mu_r(\Omega^\infty)=1$.

A class $\mathcal{F}$ of subsets of $\Omega^\infty$ is called
a \emph{$\sigma$-field on $\Omega^\infty$}
if  $\mathcal{F}$ includes $\Omega^\infty$, is closed under complements,
and is closed under the formation of countable unions.
The \emph{Borel class} $\mathcal{B}_{\Omega}$ is the $\sigma$-field \emph{generated by}
all open sets on $\Omega^\infty$.
Namely, the Borel class $\mathcal{B}_{\Omega}$ is defined
as the intersection of all the $\sigma$-fields on $\Omega^\infty$ containing
all open sets
on
$\Omega^\infty$.
A real-valued function $\mu$ defined on the Borel class $\mathcal{B}_{\Omega}$ is called
a \emph{probability measure on $\Omega^\infty$} if the following conditions hold.
\begin{enumerate}
\item $\mu\left(\mathcal{C}\right)\ge 0$ for every set $\mathcal{C}$ in $\mathcal{B}_{\Omega}$,
\item $\mu\left(\Omega^\infty\right)=1$, and
\item $\mu\left(\bigcup_{i}\mathcal{C}_i\right)=\sum_{i}\mu\left(\mathcal{C}_i\right)$
  for every sequence $\{\mathcal{C}_i\}_{i\in\N}$ of sets in $\mathcal{B}_{\Omega}$ such that
  $\mathcal{C}_i\cap\mathcal{C}_j=\emptyset$ for all $i\neq j$.
\end{enumerate}
Then, for every probability measure representation $r$ over $\Omega$,
we can show that the restriction of the outer measure $\mu_r$ on $\Omega^\infty$
to the Borel class $\mathcal{B}_{\Omega}$ is
a probability measure on $\Omega^\infty$.
We denote the restriction of $\mu_r$ to $\mathcal{B}_{\Omega}$ by
$\mu_r$
just the same.
Then it is easy to see that
\begin{equation}\label{mr}
  \mu_r\left(\osg{\sigma}\right)=r(\sigma)
\end{equation}
for every probability measure representation $r$ over $\Omega$ and every $\sigma\in \Omega^*$.

\subsection{Computability}
\label{Computability}

A partial function $f\colon\N\to\Omega^*$ or $f\colon\N\to\Q$ is called
\emph{partial recursive}
if there exists a deterministic Turing machine $\mathcal{M}$ such that, for each $n\in\N$,
when executing $\mathcal{M}$ with the input $n$,
\begin{enumerate}
\item if $n\in\Dom f$ then the computation of $\mathcal{M}$ eventually terminates
  and
  thereupon
  $\mathcal{M}$ outputs $f(n)$, and
\item if $n\notin\Dom f$ then the computation of $\mathcal{M}$ does not terminate.
\end{enumerate}
A \emph{partial recursive function} whose domain of definition is $\Omega^*$ is defined in a similar manner. 
A partial recursive function is also called a \emph{partial computable function}.

A partial recursive function $f\colon\N\to\Omega^*$ or $f\colon\N\to\Q$ is called \emph{total recursive}
if $\Dom f=\N$.
A \emph{total recursive function} whose domain of definition is $\N^+$, $\N^+\times \N$, or $\N\times\Omega^*$
is defined in a similar manner to a total recursive function whose domain of definition is $\N$,
introduced
as above.
A total recursive function is also called a
\emph{computable function}.

We say that $\alpha\in\Omega^\infty$ is \emph{computable}
if the mapping $\N\ni n\mapsto\rest{\alpha}{n}$ is a computable function.

A real $a$ is called \textit{computable} if there exists a computable function $g\colon\N\to\Q$ such that $\abs{a-g(k)} < 2^{-k}$ for all $k\in\N$.
A real $a$ is called \emph{left-computable} if
there exists a computable, increasing sequence of rationals which converges to $a$, i.e., if
there exists a computable function $h\colon\N\to\Q$ such that
$h(n)\le h(n+1)$ for every $n\in\N$ and $\lim_{n\to\infty}h(n)=a$.
On the other hand, a real $a$ is called \textit{right-computable} if
$-a$ is left-computable.
It is then easy to see that, for every $a\in\R$,
the real
$a$ is computable if and only if $a$ is both left-computable and right-computable.

A subset $\mathcal{C}$ of $\N^+\times\Omega^*$ is called \emph{recursively enumerable}
(\emph{r.e.}, for short)
if there exists a deterministic Turing machine $\mathcal{M}$ such that,
for each $x\in\N^+\times\Omega^*$,
when executing $\mathcal{M}$ with the input $x$,
\begin{enumerate}
\item if $x\in\mathcal{C}$ then the computation of $\mathcal{M}$ eventually terminates, and
\item if $x\notin\mathcal{C}$ then the computation of $\mathcal{M}$ does not terminate.
\end{enumerate}

A probability measure $\mu$ on $\Omega^\infty$ is called \emph{computable} if
there exists a computable function $g\colon\N\times\Omega^*\to\Q$ such that
$\abs{\mu\left(\osg{\sigma}\right)-g(k,\sigma)} < 2^{-k}$ for all $k\in\N$ and $\sigma\in\Omega^*$.

\subsection{Algorithmic randomness}
\label{AR}

In the following we concisely review some definitions and results of algorithmic randomness 
\cite{C75,C87b,N09,DH10}.

\emph{Martin-L\"of randomness} is a randomness notion for an infinite binary sequence,
and is one of the major notions in algorithmic randomness.
The notion of Martin-L\"of randomness is introduced as follows:
We use $\mathcal{L}$ to denote Lebesgue measure on $\XI$.
Namely, $\mathcal{L}=\mu_r$ where
$r$ is a probability measure representation over $\XI$ defined by the condition that
$r(\sigma)=2^{-\abs{\sigma}}$ for every $\sigma\in\X$.
The basic idea of Martin-L\"of randomness
is as follows
(see Martin-L\"{o}f \cite{M66}, Nies~\cite{N09}, Downey and Hirschfeldt~\cite{DH10},
and Brattka, Miller, and Nies~\cite{BMiN12}).
\begin{quote}
\textbf{Basic idea of Martin-L\"of randomness}:
The \emph{random} infinite binary sequences
are precisely sequences
which are not contained in any \emph{effective null set} on
$\XI$.
\end{quote}
Here, an \emph{effective null set} on
$\XI$
is
a set
$\mathcal{S}\in\mathcal{B}_{\{0,1\}}$
such that
$\mathcal{L}(\mathcal{S})=0$
and moreover $\mathcal{S}$ has
some sort of
\emph{effective} property.
As a specific implementation of the idea of effective null set,
we introduce the following notion.

\begin{definition}[Martin-L\"{o}f test, Martin-L\"{o}f~\cite{M66}]
\label{ML-testM}
A subset $\mathcal{C}$ of $\N^+\times\X$
is called a
\emph{Martin-L\"{o}f test} if
$\mathcal{C}$ is
an r.e.~set
and
for every $n\in\N^+$ it holds that
$\mathcal{C}_n$ is a prefix-free subset of $\X$ and
\begin{equation}\label{muocn<2n}
  \mathcal{L}\left(\osg{\mathcal{C}_n}\right)<2^{-n},
\end{equation}
where $\mathcal{C}_n$ denotes the set
$\left\{\,
    \sigma\mid (n,\sigma)\in\mathcal{C}
\,\right\}$.
\qed
\end{definition}

Let $\mathcal{C}$ be a Martin-L\"{o}f test.
Then,
it follows from \eqref{muocn<2n}
that
$\mathcal{L}\left(\bigcap_{n=1}^{\infty}\osg{\mathcal{C}_n}\right)=0$.
Therefore,
the set $\bigcap_{n=1}^{\infty}\osg{\mathcal{C}_n}$ serves as an effective null set.
In this manner, the notion of an effective null set is implemented
as
a Martin-L\"{o}f test
in
Definition~\ref{ML-testM}.
Then,
the notion of \emph{Martin-L\"of randomness} is
defined as follows, according to the basic idea of Martin-L\"{o}f randomness stated above.

\begin{definition}[Martin-L\"{o}f randomness, Martin-L\"{o}f \cite{M66}]
\label{ML-randomness}
For any $\alpha\in\XI$, we say that $\alpha$ is
\emph{Martin-L\"{o}f random} if
\[\alpha\notin\bigcap_{n=1}^{\infty}\osg{\mathcal{C}_n}\]
for every Martin-L\"{o}f test $\mathcal{C}$.\qed
\end{definition}

As a specific implementation of the idea of effective null set
in defining
a notion of the randomness
for an infinite binary sequence,
we
have adopted
the notion of Martin-L\"{o}f test
in Definition~\ref{ML-randomness},
which leads to the notion of Martin-L\"of randomness.
If we implement the idea of effective null set in a different manner,
we obtain a corresponding randomness notion for an infinite binary sequence,
which is
usually
different from the notion of Martin-L\"of randomness.
In this manner, various randomness notions,
such as $2$-randomness, weak $2$-randomness, Demuth randomness, Martin-L\"of randomness, Schnorr randomness, and Kurtz randomness,
have been introduced so far.
Consequently, they form a hierarchy of the randomness notions for an infinite binary sequence.
See Nies~\cite{N09} and Downey and Hirschfeldt~\cite{DH10} for the detail of the hierarchy.

In the hierarchy of randomness notions, Martin-L\"of randomness is the oldest historically, and
plays a central role
in the recent development of algorithmic randomness.
Martin-L\"of randomness is considered to be natural
as the notion of the randomness for an infinite binary sequence.
There are
many kinds of definitions
of the randomness for an infinite binary sequence,
which are equivalent to Martin-L\"of randomness and which are considered to be natural
as a randomness notion for an infinite binary sequence.
In particular,
the notion of
Martin-L\"of randomness is equivalently characterized
by means of
the notion of program-size complexity,
as shown in Theorem~\ref{equivMLR} below.
The \emph{program-size complexity} (or \emph{Kolmogorov complexity})
$K(\sigma)$ of a finite binary string $\sigma$ is defined as the length of the shortest binary input
for a universal decoding algorithm $U$,
called an \textit{optimal prefix-free machine},
to output $\sigma$ (see Chaitin~\cite{C75} for the detail).
By the definition, $K(\sigma)$ can be thought of as
the randomness contained in the individual finite binary string $\sigma$.

\begin{theorem}[Schnorr~\cite{Sch73}, Chaitin~\cite{C75}]\label{equivMLR}
For every $\alpha\in\XI$, the following conditions are equivalent:
\begin{enumerate}
  \item $\alpha$ is Martin-L\"{o}f random.
  \item There exists $c\in\N$ such that, for all $n\in\N^+$, $n-c \le K(\rest{\alpha}{n})$.\qed
\end{enumerate}
\end{theorem}

The condition (ii) means that the infinite binary sequence $\alpha$ is \emph{incompressible}.

\section{Finite probability spaces}
\label{FPS}

In this paper
we give an operational characterization of the notion of probability for a \emph{finite probability space}.
A finite probability space is formally defined as follows.

\begin{definition}\label{def-FPS}
Let $\Omega$ be an alphabet. A \emph{finite probability space on $\Omega$} is a function $P\colon\Omega\to\R$
such that
\begin{enumerate}
  \item $P(a)\ge 0$ for every $a\in \Omega$, and
  \item $\sum_{a\in \Omega}P(a)=1$.
\end{enumerate}
The set of all finite probability spaces on $\Omega$ is denoted by $\PS(\Omega)$.

Let $P\in\PS(\Omega)$.
The set $\Omega$ is called the \emph{sample space} of $P$,
and elements
of
$\Omega$ are called \emph{sample points} or \emph{elementary events}
of $P$.
For each $A\subset\Omega$, we define $P(A)$ by
\begin{equation}\label{eq:PA=sumainAPa}
  P(A):=\sum_{a\in A}P(a).
\end{equation}
A subset of $\Omega$ is called an \emph{event} on $P$, and
$P(A)$ is called the \emph{probability} of $A$
for every event $A$
on $P$.
\qed
\end{definition}

We use $\Omega$ to denote an arbitrary alphabet through out the rest of this paper.
It plays a role of the set of all possible outcomes of experiments or observations
in a probabilistic phenomenon in this paper.
An operational characterization of the notion of probability which we give for a finite probability space
on $\Omega$ is an infinite sequence over $\Omega$.

It is convenient to introduce the notion of
\emph{computable} finite probability space
as follows.

\begin{definition}
Let $\Omega$ be an alphabet, and let $P\in\PS(\Omega)$.
We say that $P$ is \emph{computable} if $P(a)$ is a computable real for every $a\in\Omega$.
\qed
\end{definition}

We may try to weaken the notion of the computability for a finite probability space as follows:
Let $\Omega$ be an alphabet, and let $P\in\PS(\Omega)$.
We say that $P$ is \emph{left-computable} if $P(a)$ is left-computable for every $a\in\Omega$.
On the other hand,
we say that $P$ is \emph{right-computable} if $P(a)$ is right-computable for every $a\in\Omega$.
However, using the condition (ii) of Definition~\ref{def-FPS} we can see that
these three computable notions for a finite probability space coincide with one another,
as the following proposition states.

\begin{proposition}
Let $\Omega$ be an alphabet, and let $P\in\PS(\Omega)$.
The following conditions are equivalent to one another.
\begin{enumerate}
  \item $P$ is computable.
  \item $P$ is left-computable.
  \item $P$ is right-computable.\qed
\end{enumerate}
\end{proposition}

\section{\boldmath Martin-L\"of $P$-randomness}
\label{MLRP}

In order to provide an operational characterization of the notion of probability
we use a generalization of Martin-L\"of randomness over Bernoulli measure.

Let $\Omega$ be an alphabet, and let $P\in\PS(\Omega)$.
For each $\sigma\in\Omega^*$, we use $P(\sigma)$ to denote
$P(\sigma_1)P(\sigma_2)\dots P(\sigma_n)$
where $\sigma=\sigma_1\sigma_2\dots\sigma_n$ with $\sigma_i\in\Omega$.
Therefore $P(\lambda)=1$, in particular.
For each subset $S$ of $\Omega^*$, we use $P(S)$ to denote
\[
  \sum_{\sigma\in S}P(\sigma).
\]
Therefore $P(\emptyset)=0$, in particular.

Consider a function $r\colon\Omega^*\to[0,1]$ such that $r(\sigma)=P(\sigma)$ for every $\sigma\in\Omega^*$.
It is then easy to see that the function $r$ is a probability measure representation over $\Omega$.
The probability measure $\mu_r$ induced by $r$ is
called
a \emph{Bernoulli measure on $\Omega^\infty$}, denoted
$\lambda_{P}$.
The Bernoulli measure $\lambda_{P}$ on $\Omega^\infty$ has the following property:
For every $\sigma\in \Omega^*$,
\begin{equation}\label{pBm}
  \Bm{P}{\osg{\sigma}}=P(\sigma),
\end{equation}
which results from \eqref{mr}.
It is easy to see that if a finite probability space $P\in\PS(\Omega)$ is computable then the
Bernoulli measure $\lambda_{P}$ itself is computable.

Martin-L\"of randomness with respect to Bernoulli measure,
which is called \emph{Martin-L\"of $P$-randomness} in this paper, is defined as follows.
This notion was, in essence, introduced by Martin-L\"{o}f~\cite{M66},
as well as the notion of Martin-L\"of randomness which we have described
in Definition~\ref{ML-randomness}.

\begin{definition}[%
Martin-L\"of $P$-randomness,
Martin-L\"{o}f \cite{M66}]\label{ML_P-randomness}
Let $\Omega$ be an alphabet, and let $P\in\PS(\Omega)$.
\begin{enumerate}
  \item A subset $\mathcal{C}$ of $\N^+\times \Omega^*$ is called a \emph{Martin-L\"{o}f $P$-test} if
    $\mathcal{C}$ is an r.e.~set
    and
    for every $n\in\N^+$ it holds that
    $\mathcal{C}_n$ is a prefix-free subset of $\Omega^*$ and
    \[\Bm{P}{\osg{\mathcal{C}_n}}<2^{-n},\]
    where $\mathcal{C}_n$ denotes the set
    $\left\{\,
      \sigma\mid (n,\sigma)\in\mathcal{C}
    \,\right\}$.
  \item For any $\alpha\in\Omega^\infty$ and Martin-L\"{o}f $P$-test $\mathcal{C}$,
    we say that $\alpha$ \emph{passes} $\mathcal{C}$ if there exists $n\in\N^+$ such that
    $\alpha\notin\osg{\mathcal{C}_n}$.
  \item For any $\alpha\in\Omega^\infty$, we say that $\alpha$ is \emph{Martin-L\"{o}f $P$-random} if
    for every Martin-L\"{o}f $P$-test $\mathcal{C}$
    it holds that $\alpha$ passes $\mathcal{C}$.\qed
\end{enumerate}
\end{definition}

Note that in Definition~\ref{ML_P-randomness} \emph{the finite probability space $P$ is quite arbitrary and thus
$P$ is not required to be computable at all, in particular}.
Thus,
the
Bernoulli measure $\lambda_{P}$ itself is not necessarily computable
in Definition~\ref{ML_P-randomness}.
This is a crucial point in this paper.
Note also that in Definition~\ref{ML_P-randomness} we do not require that $P(a)>0$ for all $a\in\Omega$.
Therefore, $P(a_0)$ may be $0$ for some $a_0\in\Omega$.
In the case where $\Omega=\{0,1\}$ and
$P$
satisfies that $P(0)=P(1)=1/2$,
the Martin-L\"of $P$-randomness results in the Martin-L\"of randomness
in Definition~\ref{ML-randomness}.

In Definition~\ref{ML_P-randomness}, we require that
the set $\mathcal{C}_n$ is prefix-free in the definition of a Martin-L\"{o}f $P$-test $\mathcal{C}$.
However, we can eliminate this
requirement
while keeping the notion of Martin-L\"of $P$-randomness
the same, as the following theorem states.

\begin{theorem}\label{ML_P-randomness_eliminated-prefix-freeness}
Let $\Omega$ be an alphabet, and let $P\in\PS(\Omega)$.
Let $\alpha\in\Omega^\infty$.
Then the following conditions are equivalent to each other.
\begin{enumerate} 
  \item The infinite sequence $\alpha$ is Martin-L\"{o}f $P$-random.
  \item For every r.e.~subset $\mathcal{C}$ of $\N^+\times \Omega^*$ such that
    $\Bm{P}{\osg{\mathcal{C}_n}}<2^{-n}$ for every $n\in\N^+$,
    there exists $n\in\N^+$ such that $\alpha\notin\osg{\mathcal{C}_n}$.
    \qed
\end{enumerate}
\end{theorem}

Based on Lemma~\ref{eliminate-prefix-freeness} below,
Theorem~\ref{ML_P-randomness_eliminated-prefix-freeness} can be proved
in almost the same manner as in the proof of the corresponding result
for a usual Martin-L\"of test for an infinite binary sequence with respective to Lebesgue measure $\mathcal{L}$,
whose definition is stated in Definition~\ref{ML-testM}.
For completeness and convenience, we include
the proof of Lemma~\ref{eliminate-prefix-freeness}.

\begin{lemma}\label{eliminate-prefix-freeness}
Let $\Omega$ be an alphabet.
For every r.e.~subset $\mathcal{C}$ of $\N^+\times \Omega^*$
there exists an r.e.~subset $\mathcal{D}$ of $\N^+\times \Omega^*$ such that
$\mathcal{D}_n$ is a prefix-free subset of $\Omega^*$ and
$\osg{\mathcal{C}_n}=\osg{\mathcal{D}_n}$ for every $n\in\N^+$,
where $\mathcal{C}_n$ and $\mathcal{D}_n$ denote
the sets
$\left\{\,
    \sigma\mid (n,\sigma)\in\mathcal{C}
\,\right\}$
and
$\left\{\,
    \sigma\mid (n,\sigma)\in\mathcal{D}
\,\right\}$,
respectively.
\end{lemma}

\begin{proof}
Let $\mathcal{C}$ be an r.e.~subset of $\N^+\times \Omega^*$.
We
choose a particular recursive enumeration
\begin{equation*}
  (n_1,\sigma_1),(n_2,\sigma_2),(n_3,\sigma_3),\dotsc\dotsc
\end{equation*}
of the r.e.~set $\mathcal{C}$.
Note that this list may be finite.
We construct
an r.e.~subset $\mathcal{D}$ of $\N^+\times \Omega^*$
by constructing a
double
sequence $\{\mathcal{D}(n,k)\}_{n,k}$ of finite subsets of $\N^+\times \Omega^*$ such that
(i)~$\mathcal{D}(n,k) \subset \mathcal{D}(n,k+1)$ for
any
$n$ and $k$,
and
(ii)~$\mathcal{D}=\{(n,\sigma)\in\N^+\times \Omega^*\mid \sigma\in\bigcup_k \mathcal{D}(n,k)\}$,
while enumerating $\mathcal{C}$
just
as in the above list.

For that purpose, we introduce some notation.
For each $\sigma\in\Omega^*$ and finite subset $S$ of $\Omega^*$,
we denote by $G(\sigma,S)$ the set of the shortest strings $\tau\in\Omega^*$ such that
$\tau$ is incompatible with any string in $S$ and $\sigma$ is a prefix of $\tau$.
It is then easy to see that for every $\sigma\in\Omega^*$ and finite prefix-free subset $S$ of $\Omega^*$
it holds that
(i)~$\osg{S\cup G(\sigma,S)}=\osg{S}\cup\osg{\sigma}$
and (ii)~$S\cup G(\sigma,S)$ is
a finite prefix-free set.
Moreover, note that,
given a finite string $\sigma\in\Omega^*$ and a finite subset $S$ of $\Omega^*$,
one can effectively calculate the finite set $G(\sigma,S)$.

Now, the construction of the double sequence $\{\mathcal{D}(n,k)\}_{n,k}$ is performed as follows,
in order of increasing $k$:
Initially, one sets $\mathcal{D}(n,0):=\emptyset$ for all $n\in\N^+$.
In general,
one waits for a new element of $\mathcal{C}$ to be generated
in the recursive enumeration of $\mathcal{C}$ by the list above.
Whenever
$(n_k,\sigma_k)$ is generated in the list,
for each $n\in\N^+$
one sets
\begin{equation*}
  \mathcal{D}(n,k):=\mathcal{D}(n,k-1)\cup G(\sigma_k,\mathcal{D}(n,k-1))
\end{equation*}
if $n=n_k$ and
\begin{equation*}
  \mathcal{D}(n,k):=\mathcal{D}(n,k-1)
\end{equation*}
otherwise.
Then one repeats this procedure for $k+1$
instead of $k$.

Thus, for each $n\in\N^+$ we can show, by induction on $k$, that
for all $k$ it holds that $\mathcal{D}(n,k)$ is
a finite prefix-free set
and
\begin{equation*}
  \osg{\mathcal{D}(n,k)}=\osg{\{\sigma_i\mid i\le k\;\&\;n_i=n\}}.
\end{equation*}
Hence, by setting $\mathcal{D}:=\{(n,\sigma)\in\N^+\times \Omega^*\mid\sigma\in\bigcup_k \mathcal{D}(n,k)\}$,
we have that
(i)~$\mathcal{D}$ is an r.e.~subset of $\N^+\times \Omega^*$
and (ii)~$\mathcal{D}_n$ is prefix-free and
$\osg{\mathcal{C}_n}=\osg{\mathcal{D}_n}$ for every $n\in\N^+$.
This completes the proof.
\end{proof}

Note that
we did not use any specific property of the finite probability space $P$
in proving Theorem~\ref{ML_P-randomness_eliminated-prefix-freeness}
based on Lemma~\ref{eliminate-prefix-freeness}.
Actually,
in Theorem~\ref{ML_P-randomness_eliminated-prefix-freeness}
we do not require $P$ to be computable at all,
in particular.
However, we
have certainly been able to confirm
that Theorem~\ref{ML_P-randomness_eliminated-prefix-freeness} holds
true
for an arbitrary finite probability space $P$.
\emph{This quite arbitrariness of the underlying finite probability spaces
is a specific feature of the results
on Martin-L\"of $P$-randomness
proved
in this paper},
except for the results given in Section~\ref{FENICFPS} and
Sections~\ref{subsec:AC} and~\ref{AFBC},
where we consider computable finite probability spaces.

Since there are only countably infinitely many algorithms,
we can show the following
theorem, as is
shown for the usual Martin-L\"of randomness
for infinite binary sequences with respective to Lebesgue measure.

\begin{theorem}\label{Bmae}
Let $\Omega$ be an alphabet, and let $P\in\PS(\Omega)$.
Then $\mathrm{ML}_P\in\mathcal{B}_{\Omega}$ and $\Bm{P}{\mathrm{ML}_P}=1$,
where $\mathrm{ML}_P$ is the set of all Martin-L\"of $P$-random
infinite
sequences over $\Omega$.
\end{theorem}

\begin{proof}
Since there are only countably infinitely many Turing machines,
there are only countably infinitely many Martin-L\"{o}f $P$-tests
$\mathcal{C}^1,\mathcal{C}^2,\mathcal{C}^3,\dotsc$.
For each $i\in\N^+$, let
$\mathrm{NML}_P^i$
be the set of all $\alpha\in\Omega^\infty$ which does not pass $\mathcal{C}^i$.

Let $i\in\N^+$.
We see that $\mathrm{NML}_P^i=\bigcap_{n=1}^\infty\osg{\mathcal{C}^i_n}$, where $\mathcal{C}^i_n$ denotes the set
$\left\{\,
    \sigma\mid (n,\sigma)\in\mathcal{C}^i
\,\right\}$.
Therefore we have $\mathrm{NML}_P^i\in\mathcal{B}_{\Omega}$.
Since
\[
  \Bm{P}{\mathrm{NML}_P^i}\le\Bm{P}{\osg{\mathcal{C}^i_n}}<2^{-n}
\]
for every $n\in\N^+$,
we have $\Bm{P}{\mathrm{NML}_P^i}=0$.
Thus, since $\Omega^\infty\setminus\mathrm{ML}_P=\bigcup_{i=1}^\infty \mathrm{NML}_P^i$,
it follows that $\mathrm{ML}_P\in\mathcal{B}_{\Omega}$ and
$\Bm{P}{\Omega^\infty\setminus\mathrm{ML}_P}=0$.
In particular, the latter
implies that $\Bm{P}{\mathrm{ML}_P}=1$, as desired.
\end{proof}

Note that Theorem~\ref{Bmae} holds for an arbitrary finite probability space $P$,
which is not necessarily computable.

\section{Ensemble}
\label{OPT}

In this section
we present an operational characterization of the notion of probability for a finite probability space,
and consider its validity.
We propose to regard a Martin-L\"of $P$-random sequence of sample points
as an \emph{operational characterization of the notion of probability}
for a finite probability space $P$.
Namely,
we propose to identify a Martin-L\"of $P$-random sequence of sample points
with the \emph{substance} of the notion of probability for a finite probability space $P$.
Thus, since the notion of Martin-L\"of $P$-random sequence plays a central role in our framework,
in particular we call it an \emph{ensemble}, as in Definition~\ref{ensemble},
instead of collective for distinction.
The name ``ensemble'' comes from physics, in particular, from quantum mechanics and
statistical mechanics.%
\footnote{The notion of ensemble plays a fundamental role in quantum mechanics and statistical mechanics.
However,
this notion
in physics
is very vague
from a mathematical point of view.
In a series of works~\cite{T15Kokyuroku,T15WiNF-Tadaki_rule,T16QIT35,T18arXiv}
we propose to regard a Martin-L\"of $P$-random sequence of quantum states
as a formal definition of the notion of ensemble in quantum mechanics and statistical mechanics, i.e.,
as a formal definition of the notion of ensemble of quantum states.
See Tadaki~\cite{T18arXiv} for the detail.}

\begin{definition}[Ensemble]\label{ensemble}
Let $\Omega$ be an alphabet, and let $P\in\PS(\Omega)$.
A Martin-L\"of $P$-random
infinite sequence over $\Omega$
is called an \emph{ensemble} for the finite probability space $P$ on $\Omega$.
\qed
\end{definition}

Let $\Omega$ be an alphabet, and let $P\in\PS(\Omega)$.
Consider an infinite sequence $\alpha\in\Omega^\infty$ of outcomes which is being generated
by infinitely repeated trials \emph{described by} the finite probability space $P$.
The operational characterization of the notion of probability for the finite probability space $P$
is thought to be completed
if the property which the infinite sequence $\alpha$ has to satisfy is determined.
We thus propose the following thesis.

\begin{thesis}\label{thesis}
Let $\Omega$ be an alphabet, and let $P\in\PS(\Omega)$.
An infinite sequence of outcomes in $\Omega$ which is
being generated by infinitely repeated trials \emph{described by} the finite probability space $P$ on $\Omega$
is an ensemble for $P$.
\qed
\end{thesis}

The notion of probability plays a crucial role in almost all fields of science and technology.
In such a field, the notion of probability is formally handled by means of the use of a \emph{probabilistic model}
for understanding a specific phenomenon investigated in the field.
For the use of a probabilistic model in an arbitrary field of science and technology,
we propose to regard Thesis~\ref{thesis} as serving
in the following form:
When we make use of a probabilistic model which is especially based on
the notion of
a finite probability space
given in Definition~\ref{def-FPS},
we think that we are \emph{implicitly} assuming that Thesis~\ref{thesis} holds
for the specific application of the notion of probability
through
the probabilistic model.

Now, let us
check
the validity of Thesis~\ref{thesis}.
First of all, what is ``probability''?
It would seem very difficult to answer this question
\emph{completely} and \emph{sufficiently}.
However, we may enumerate the \emph{necessary} conditions which the notion of probability is
considered to have to satisfy
\emph{according to
our intuitive understanding of the notion of probability}.
In the subsequent subsections,
we check that the notion of ensemble satisfies these necessary conditions.

\subsection{Sure occurrences of an elementary event with probability one}
\label{Subsec:SOEEPO}

Let $\Omega$ be an alphabet, and let $P\in\PS(\Omega)$. Let us consider
an infinite sequence $\alpha\in\Omega^\infty$ of outcomes
which is being generated by infinitely repeated trials described
by the finite probability space $P$ on $\Omega$.
The
first
necessary condition which the notion of probability for the finite probability space $P$
is considered to have to satisfy is
the condition that
\emph{an elementary event with probability one always occurs in the infinite sequence $\alpha$},
i.e., the condition that for every $a\in\Omega$ if $P(a)=1$ then $\alpha$ is of the form
$\alpha=aaaaa\dotsc\dotsc$.
This intuition that \emph{an elementary event with probability one occurs certainly}
is particularly supported by the notion of probability in \emph{quantum mechanics}.

In what follows, we will present
\emph{five}
reasons of this fact
that \emph{an
elementary
event with probability one occurs certainly in quantum mechanics}.
We will do this
through Sections~\ref{Reason-MQSE} --- \ref{Reason-DLQM} below,
according to the notion of probability in quantum mechanics,
sometimes
appealing
to our intuitive understanding of the notion of probability.
Then,
in Section~\ref{one_probability-zero_probability-for-ensembles},
we will go back to the consideration on the notion of ensemble \emph{with full mathematical rigor},
and
show that
the fact above in quantum mechanics
can be
reproduced properly for an ensemble.
Namely, we will prove Theorem~\ref{one_probability} below, which means
that \emph{an elementary event with probability one always occurs in an ensemble}.

To begin with, we recall
the central postulates of quantum mechanics.
For simplicity, we here consider the postulates of quantum mechanics
for a \emph{finite-dimensional} quantum system, i.e.,
a quantum system whose state space is a finite-dimensional Hilbert space.%
\footnote{Based on a similar argument to Section~\ref{Reason-MQSE} of this paper,
we can demonstrate the fact that an event with probability one occurs certainly,
for an \emph{infinite-dimensional} quantum system.
See Tadaki~\cite[Section~7.1]{T19arXiv} for the detail.}
See e.g.~Nielsen and Chuang~\cite[Section 2.2]{NC00}
for the detail of the postulates of quantum mechanics,
in particular, in the finite-dimensional case.
We refer to some of the postulates from
\cite[Section 2.2]{NC00} in what follows.

The first postulate of quantum mechanics is about \emph{state space} and \emph{state vector}.
\begin{postulate}[State space and state vector]\label{state_space}
Associated to any isolated physical system is a complex vector space with inner product
(i.e., Hilbert space) known as the \emph{state space} of the system.
The system is completely described by its \emph{state vector},
which is a unit vector in the system's state space.
\qed
\end{postulate}

The second postulate of quantum mechanics is about the \emph{composition} of systems.

\begin{postulate}[Composition of systems]\label{composition}
The state space of a composite physical system is the tensor product of the state spaces of the component physical systems.
Moreover, if we have systems numbered $1$ through $n$, and system number~$i$ is
prepared
in the state~$\Psi_i$,
then the joint state of the total system is
$\Psi_1\otimes\Psi_2\otimes\dots\otimes\Psi_n$.
\qed
\end{postulate}

The third postulate of quantum mechanics is about the \emph{time-evolution} of
\emph{closed} quantum systems.

\begin{postulate}[Unitary time-evolution]\label{evolution}
The evolution of a \emph{closed} quantum system is described by a \emph{unitary transformation}.
That is,
the state~$\Psi_1$ of the system at time~$t_1$ is related to the state~$\Psi_2$ of the system at time~$t_2$
by a unitary operator~$U$, which depends only on the times~$t_1$ and $t_2$,
in such a way that
$\Psi_2=U\Psi_1$.
\qed
\end{postulate}

The forth postulate of quantum mechanics is about \emph{measurements} on quantum systems.
\begin{postulate}[Measurements]\label{measurements}
Quantum measurements are described by an \emph{observable}, $M$, a Hermitian matrix
on the state space of the system being measured.
The observable has a spectral decomposition,
\begin{equation*}
  M=\sum_{m\in\Omega}m E_m,
\end{equation*}
where
$\Omega$ is the spectrum of $M$ and
$E_m$ is the projector onto the eigenspace of $M$ with eigenvalue $m$ for every $m\in\Omega$.
\begin{enumerate}
\item
The set of possible outcomes of the measurement is the spectrum $\Omega$ of $M$.
If the state of the quantum system is $\Psi$
immediately before the measurement then the \emph{probability} that
an outcome $m\in\Omega$
occurs is given by
$(\Psi,E_m\Psi)$,
where $(\cdot, \cdot)$ denotes the inner-product defined on the state space of the system.
\item Given that outcome $m$ occurred,
the state of the quantum system immediately after the measurement is
\[
  \frac{E_m\Psi}{\sqrt{(\Psi, E_m\Psi)}}.
\]
\qed
\end{enumerate}
\end{postulate}
Postulate~\ref{measurements}~(i) is the so-called \emph{Born rule}, i.e,
\emph{the probability interpretation of the wave function},
while Postulate~\ref{measurements}~(ii) is called the \emph{projection hypothesis}.
Note that Postulate~\ref{measurements}
describes the effects of measurements on quantum systems using the notion of \emph{probability}.
However, it does not mention the \emph{operational definition} of the notion of probability.

\subsubsection{Reason 1: Measurement of a quantum system in an eigenstate}
\label{Reason-MQSE}

The first reason of the fact that
\emph{an
elementary
event with probability one occurs certainly in quantum mechanics}
is given as follows.
As mentioned above,
Postulate~\ref{measurements}
does not refer to the \emph{operational definition} of the notion of probability,
whereas
it is heavily based on the notion of \emph{probability}.
In contrast,
there is a postulate about quantum measurements with no reference to the notion of probability.
This is given in
Dirac~\cite[p.35]{D58},
and describes a spacial case of  quantum measurements which are performed upon a quantum system
in an \emph{eigenstate} of an observable, i.e., a state represented by an eigenvector of an observable.
\begin{postulate}[Dirac~\cite{D58}, p.35]\label{Dirac}
If the dynamical system is in an eigenstate of a real dynamical variable $\xi$, belonging to the eigenvalue $\xi'$,
then a measurement of $\xi$ will certainly gives as result the number $\xi'$.
\qed
\end{postulate}
Here, the ``dynamical system'' means quantum system, and the ``real dynamical variable'' means observable.

Based on Postulates~\ref{state_space}, \ref{measurements}, and \ref{Dirac} above,
we can show that an elementary event \emph{with probability one} occurs certainly in quantum mechanics.
To see this, let us consider a quantum system with finite-dimensional state space,
and a measurement described by an observable $M$ performed upon the system.
Let
\begin{equation*}
  M=\sum_{m\in\Omega}m E_m
\end{equation*}
be a spectral decomposition of the observable $M$, where
$\Omega$ is the spectrum of $M$ and
$E_m$ is the projector onto the eigenspace of $M$ with eigenvalue $m$ for every $m\in\Omega$.

Suppose that
the probability of getting
an outcome $m_0\in\Omega$
is one in the measurement
of $M$
performed upon the system
in a state represented by a state vector $\Psi$.
Then, it follows from Postulate~\ref{measurements}~(i) that $(\Psi, E_{m_0}\Psi)=1$.
This implies that
$\Psi$ is an eigenvector of $M$ belonging to the eigenvalue $m_0$,
i.e., $M\Psi=m_0\Psi$,
since $\Psi$ is a unit vector.
Therefore, we have that immediately before the measurement, the quantum system is
in an eigenstate of the observable $M$, belonging to the eigenvalue $m_0$.
It follows from Postulate~\ref{Dirac} that
the measurement of $M$ will \emph{certainly} gives as result the number $m_0$.
Hence, it turns out that
\emph{an elementary event with probability one occurs certainly in quantum mechanics}.

\subsubsection{Reason 2: Consistency of the Born rule with the projection hypothesis}
\label{Reason-CBRPH}

The consistency of the Born rule,
Postulate~\ref{measurements}~(i),
with the projection hypothesis,
Postulate~\ref{measurements}~(ii),
\emph{suggests} that
\emph{an elementary event with probability one occurs certainly in quantum mechanics}.
This is the second reason of the fact.
To see this, let us consider a quantum system with finite-dimensional state space,
and a measurement described by an observable $M$ performed upon the system.
Let
\begin{equation*}
  M=\sum_{m\in\Omega}m E_m
\end{equation*}
be a spectral decomposition of the observable $M$, where $\Omega$ is the spectrum of $M$ and
$E_m$ is the projector onto the eigenspace of $M$ with eigenvalue $m$ for every $m\in\Omega$.

Suppose that
the probability of getting
an outcome $m_0\in\Omega$
is one in the measurement of $M$ performed upon the system
in a state represented by a state vector $\Psi$.
Then, it follows from Postulate~\ref{measurements}~(i) that
\begin{equation}\label{sec5-1-2:eqPEm0P=1}
  (\Psi, E_{m_0}\Psi)=1.
\end{equation}
Assume contrarily that the outcome $m_0$ does not occur certainly in the measurements of $M$
performed upon the system in the state represented by the state vector $\Psi$.
This implies that there surely comes a chance when
some different outcome $m_1$ from $m_0$ occurs in
a
measurement of $M$
performed upon the system in the state represented by the state vector $\Psi$,
while repeating
this
measurement.
For this $m_1$ we see that $(\Psi, E_{m_1}\Psi)=0$,
using
\eqref{sec5-1-2:eqPEm0P=1} and the completeness relation
\[
\sum_{m\in\Omega}E_m=I
\]
where $I$ is the identity operator.
Therefore we have that $E_{m_1}\Psi=0$.
On the other hand, according to Postulate~\ref{measurements}~(ii),
the state of the system immediately after the measurement
is
\[
  \frac{E_{m_1}\Psi}{\sqrt{(\Psi, E_{m_1}\Psi)}}
\]
in this chance
where the outcome $m_1$ has occurred.
However,
this quantity is not well-defined
and cannot represent a quantum state,
since $E_{m_1}\Psi=0$.
Thus, we face with a difficulty.
In this way,
if an outcome with probability one does not occur certainly in the measurements of $M$,
then the Born rule and the projection hypothesis become inconsistent with each other.
Hence, it would be consistent to assume that
the outcome $m_0$
with probability one
occurs certainly in the measurements of $M$.

Instead,
we could modify
the projection hypothesis
in the following manner in order to escape
the
difficulty: 
\begin{quote}
``Given that outcome $m$ occurred,
the state of the quantum system immediately after the measurement is
represented by an eigenvector of the observable $M$, belonging to the eigenvalue $m$.''
\end{quote}
However,
this modification brings difficulties to us again.
Namely,
it brings us
trouble in determining the post-measurement state 
when the eigenspace of $M$ with the eigenvalue $m$ is \emph{degenerate}.
In the first place, the merit of the original projection hypothesis,
Postulate~\ref{measurements}~(ii),
is that
it
enables us to determine
the state of a quantum system immediately after the measurement
of an
\emph{arbitrary}
observable whose spectrum
may have
degeneracy.
Thus,
the modification above spoils this merit of the original projection hypothesis,
and significantly weakens
it.

Thus,
it would be consistent
to assume that
\emph{an elementary event with probability one occurs certainly in quantum mechanics}.

\subsubsection{Reason 3: Consequence of the repeatability hypothesis}
\label{Reason-CRH}

The third reason of the fact that
\emph{an
elementary
event with probability one occurs certainly in quantum mechanics} is given as follows.
In addition to Postulate~\ref{Dirac} mentioned above,
there is another postulate about quantum measurements with no reference to the notion of probability.
It is called the \emph{repeatability hypothesis},
and
makes a statement about
the relation between the outcomes of two successive measurements of an identical observable.
We here refer to the repeatability hypothesis
in the following form which is given in Dirac~\cite[p.36]{D58}.

\begin{postulate}[The repeatability hypothesis,
Dirac~\cite{D58}, p.36]\label{Dirac2}
When we measure a real dynamical variable $\xi$,
the disturbance involved in the act of measurement causes a jump in the state of the dynamical system.
From physical continuity,
if we make a second measurement of the same dynamical variable $\xi$
immediately after the first, the result of the second measurement must be the same as that of the first.
\qed
\end{postulate}

Specifically, the second sentence of Postulate~\ref{Dirac2} is
the statement of the repeatability hypothesis.
Based on Postulates~\ref{state_space}, \ref{measurements}, and \ref{Dirac2} above,
we can show that
\emph{an elementary event with probability one occurs certainly in quantum mechanics}.
To see this, let us consider a quantum system with finite-dimensional state space,
and a measurement described by an observable $M$ performed upon the system.
Let $\Omega$ be the spectrum of $M$,
and let $E_m$ be the projector onto the eigenspace of $M$ with eigenvalue $m\in\Omega$.

Suppose that
the probability of getting
an outcome $m_0\in\Omega$
is one in the measurement
of $M$
performed upon the system
in a state represented by a state vector $\Psi$.
Then, it follows from Postulate~\ref{measurements}~(i) that $(\Psi, E_{m_0}\Psi)=1$.
This implies that
\begin{equation}\label{QMProb1:EmP=P}
  E_{m_0}\Psi=\Psi,
\end{equation}
since $\Psi$ is a unit vector and $E_{m_0}$ is a projector.

Now, based on our intuitive understanding of the notion of probability,
we further make the following two assumptions~(a) and (b):
\renewcommand{\labelenumi}{(\alph{enumi})}
\begin{enumerate}
\item
We can repeat the measurement of $M$ performed upon the system
prepared
in the state represented by the state vector $\Psi$ as many times as we want.
\item We eventually get the outcome $m_0$ while repeating the measurement of $M$
performed upon the system prepared in the state represented by the state vector $\Psi$.
\end{enumerate}
\renewcommand{\labelenumi}{(\roman{enumi})}

The plausibility of the assumption~(a) is
demonstrated
as follows:
In this argument, we have supposed that
the probability of getting
the outcome $m_0\in\Omega$
is one in the measurement of $M$ performed upon the system
in the state represented by the state vector $\Psi$, as above.
Thus, we are considering the probability of a specific outcome in the measurement of $M$
over the system
prepared in the state $\Psi$. 
In general,
\emph{according to our intuitive understanding of the notion of probability},
when we refer to the probability of a specific outcome,
we
\emph{implicitly}
assume the infinite repeatability of the trial regarding that probabilistic phenomenon,
which is
the infinite repeatability of
the measurement of $M$
performed upon the system prepared in the state represented by the state vector $\Psi$
in the current situation.
Thus, the assumption~(a) is
considered to be
plausible.

On the other hand, the plausibility of the assumption~(b) is
demonstrated
as follows:
Recall
again
that we are assuming that the probability of getting the outcome $m_0$ is one
(and therefore \emph{positive},
in particular)
in the measurement of $M$
performed upon the system in the state represented by the state vector $\Psi$.
In general,
\emph{according to our intuitive understanding of the notion of probability},
if the probability of a specific outcome is
\emph{positive},
then we eventually get this outcome while repeating the trial regarding that probabilistic phenomenon,
which is
the measurement of $M$ performed upon the system prepared in the state represented by the state vector $\Psi$
in the current situation.
Thus, according to our intuitive understanding of the notion of probability,
the assumption~(b) is also
considered to be
plausible.

Now,
we repeat the measurement
of $M$
over the system
prepared in the state $\Psi$
until we obtain the outcome $m_0$.
This
repetition
is possible due to the assumption~(a).
Then there surely comes a chance when the outcome $m_0$ occurs in
a
measurement of $M$
over the system prepared in the state $\Psi$,
while repeating
the measurement
of $M$.
This chance surely
comes
due to the assumption~(b).
In this chance,
according to the projection hypothesis, Postulate~\ref{measurements}~(ii),
the state of the system immediately after the measurement
is
\[
  \frac{E_{m_0}\Psi}{\sqrt{(\Psi, E_{m_0}\Psi)}}=\Psi,
\]
where the equality
follows from
\eqref{QMProb1:EmP=P}.
We then perform a second measurement of $M$
over the system
immediately after the first
measurement of $M$
in this chance.
According to the repeatability hypothesis, Postulate~\ref{Dirac2},
the outcome of the second measurement
of $M$
\emph{must be}
the same as the first, i.e., $m_0$.
Recall here that
\emph{a quantum system is completely described by its state vector} due to Postulate~\ref{state_space}.
In this light,
note that
the system was in the state $\Psi$
immediately before the second measurement of $M$.
Thus, by focusing on the second measurement of $M$ in this chance,
we must conclude that,
in general,
the outcome $m_0$ occurs certainly
in the measurement
of $M$
performed upon the system in the state represented by the state vector $\Psi$.

We have
derived
this conclusion under the assumption that
the probability of getting the outcome $m_0$ is one in the measurement of $M$
performed upon the system in the state represented by the state vector $\Psi$.
Thus, we have shown that
\emph{an elementary event with probability one occurs certainly in quantum mechanics}.

\subsubsection{Reason 4: The unitary description of measurement proccess}
\label{Reason-UDMP}

The forth reason of the fact that
\emph{an
elementary
event with probability one occurs certainly in quantum mechanics}
is brought from a different point of view than
the three reasons above, presented in Sections~\ref{Reason-MQSE} --- \ref{Reason-CRH}.
Consider a quantum system $\mathcal{S}$ with finite-dimensional state space,
and a measurement described by an observable $M$ performed upon the system~$\mathcal{S}$.
Let $\Omega$ be the spectrum of $M$,
and let $E_m$ be the projector onto the eigenspace of $M$ with eigenvalue $m\in\Omega$.

According to Postulates~\ref{state_space}, \ref{composition}, and \ref{evolution} above,
the measurement process of $M$ is described by a unitary operator $U$ such that
\begin{equation}\label{sec5-1-4:eqU}
  U(\Psi\otimes\Phi_{\mathrm{init}})=\sum_{m\in\Omega}(E_m\Psi)\otimes\Phi[m]
\end{equation}
for every state vector $\Psi$ of the system $\mathcal{S}$,
as an interaction between the quantum system $\mathcal{S}$ and an apparatus $\mathcal{A}$
performing the measurement of the observable $M$ of the system $\mathcal{S}$
(von Neumann~\cite[VI.3.]{vN55}).
Here, the apparatus $\mathcal{A}$ is also treated as a quantum system,
which interacts with the quantum system~$\mathcal{S}$.
On the one hand, the vector $\Phi_{\mathrm{init}}$ is the initial state of the apparatus $\mathcal{A}$.
Thus the initial state of the composite system consisting of
the system $\mathcal{S}$ and the apparatus $\mathcal{A}$ before the measurement is
$\Psi\otimes\Phi_{\mathrm{init}}$,
as seen in the left-hand side of \eqref{sec5-1-4:eqU}.
On the other hand, the vector $\Phi[m]$ is a final state of the apparatus $\mathcal{A}$,
and indicates that \emph{the apparatus $\mathcal{A}$ records the value $m$ of $M$}.
The final state of the composite system after the measurement
is a superposition of the states $(E_m\Psi)\otimes\Phi[m]$,
as seen in the right-hand side of \eqref{sec5-1-4:eqU}.
Therefore,
the measurement process of $M$
\emph{generally} ends up with
an entanglement
shared
between the system~$\mathcal{S}$
being measured
and
the apparatus $\mathcal{A}$
measuring it.

Now, let us assume that
the probability of getting
an outcome $m_0\in\Omega$
is one in the measurement
of $M$
performed upon the system~$\mathcal{S}$
in a state represented by a state vector $\Psi$.
Then, it follows from Postulate~\ref{measurements}~(i) that
$(\Psi, E_{m_0}\Psi)=1$.
Thus,
since $\Psi$ is a unit vector ,
using the completeness relation
\[
  \sum_{m\in\Omega}E_m=I
\]
we have that
\begin{equation}\label{seq5-1-4:eqMmP=dmm0P}
  E_{m}\Psi=\delta_{m,m_0}\Psi
\end{equation}
for every $m\in\Omega$.
On the other hand,
the assumption implies that
the system~$\mathcal{S}$ is prepared in the state $\Psi$ immediately before the measurement of $M$.
Thus, the equation~\eqref{sec5-1-4:eqU} holds for
this
$\Psi$.
However,
due to \eqref{seq5-1-4:eqMmP=dmm0P},
this equation
results in
\[
  U(\Psi\otimes\Phi_{\mathrm{init}})=\Psi\otimes\Phi[m_0].
\]
The resulting equation above shows that
the
entanglement
shared between the system $\mathcal{S}$ and the apparatus $\mathcal{A}$,
which generally appears in the formula~\eqref{sec5-1-4:eqU},
no longer exists
immediately after the measurement process of $M$
performed upon the system $\mathcal{S}$ in the state
$\Psi$.
That is,
immediately
after the measurement process
of $M$
the apparatus $\mathcal{A}$ is in the \emph{definite} state $\Phi[m_0]$,
which indicates that \emph{the apparatus $\mathcal{A}$ records the value $m_0$ of $M$}.
This means that \emph{the outcome $m_0$ occurs certainly in the
measurement
of $M$
performed upon the system $\mathcal{S}$ in the state represented by the state vector $\Psi$}.

Hence, we must conclude that
\emph{an elementary event with probability one occurs certainly in quantum mechanics}.

\subsubsection{Reason 5: Descriptions in
literature
on quantum mechanics}
\label{Reason-DLQM}

The fifth reason of the fact that
\emph{an elementary event with probability one occurs certainly in quantum mechanics}
is
the realization of this fact which
can be found in various seminal
literature on quantum mechanics,
although
this
realization
does not seem to be shared widely up to the present.
In what follows, we enumerate
descriptions regarding
the
realization
of this fact
among seminal literature on quantum mechanics.

First, we refer to the descriptions in seminal textbooks on quantum mechanics.
Dirac~\cite{D58} states that
``Only in special cases when the probability for some result is unity is
the result of the experiment determinate'' (p.14).
This precisely states that event with probability one occurs certainly.
In the same context,
Landau and Lifshitz~\cite{LL77} states the following:
\begin{quotation}
The problem in quantum mechanics consists in determining the probability of obtaining various results
on performing this measurement.
It is understood, of course, that in some cases the probability of a given result of measurement may be
equal to unity, i.e. certainty, so that the result of that measurement is unique. (p.5)
\end{quotation}
The second sentence of this quotation is, in essence, the same statement as Dirac's above.

Von Neumann~\cite{vN55} is the most fundamental
work
on
the mathematical foundations of quantum mechanics.
In von Neumann~\cite{vN55} we can find several
descriptions
which suggest
the fact
that an elementary event with probability one occurs certainly.
In particular,
we here focus on the following description.
Von Neumann~\cite[VI.2.]{vN55} considers simultaneous measurements of observables $A$ and $B$
performed on a state $\Phi$ of the composite system of consisting of a system I and a system II,
where $A$ and $B$ are observables of the systems I and II, respectively,
and the spectrum of each of $A$ and $B$ has no degeneracy.
Then von Neumann~\cite[VI.2.]{vN55} states the following:
\begin{quotation}
An $a_m$ with all $f_{mn}=0$ cannot result, because its total probability
\[
\sum_{n=1}^\infty \abs{f_{mn}}^2
\]
cannot be $0$, if $a_m$ is ever observed -- --
therefore for exactly one $n$, $f_{mn}\neq 0$; likewise for $b_n$.
(in parentheses, pp.434--435)
\end{quotation}
Here $a_m$ and $b_n$ are eigenvalues of observables $A$ and $B$,
respectively,
and $f_{mn}$ is an expansion coefficient of
the
state $\Phi$ of the composite system I$+$II by
a
complete orthonormal set $\{\phi_m\otimes\xi_n\}$,
where $\phi_m$ and $\xi_n$ are eigenvectors of $A$ and $B$,
belonging to the eigenvalues $a_m$ and $b_n$, respectively.
Namely, the double sequence $\{f_{mn}\}$ satisfies that
\[
\Phi=\sum_{m,n=1}^\infty f_{mn}\phi_m\otimes\xi_n.
\]
We focus on the first half of
the above description in von Neumann~\cite[VI.2.]{vN55}.
It can be rephrased as follows:
In the measurement of the observable $A$ performed upon the composite system I$+$II
in the state represented by the state vector $\Phi$,
if 
an outcome
$a_m$ is ever observed then the
probability $\sum_{n=1}^\infty \abs{f_{mn}}^2$
of getting the
outcome
$a_m$ cannot be $0$,
and therefore,
as the contraposition of this,
if $f_{mn}=0$ for all $n\ge 1$ then
the outcome $a_m$ cannot result.
Hence,
this description
in von Neumann~\cite[VI.2.]{vN55}
is based on
the fact
that
\emph{an elementary event with probability
zero
never occurs}.
As we will see in
Section~\ref{Equiv-NOEEPZ} below,
according to our intuitive understanding of the notion of probability,
the statement that
an elementary event with probability
zero
never occurs
is equivalent to the statement that
an elementary event with probability one occurs certainly.
Thus, von Neumann~\cite{vN55} assumes that
\emph{an elementary event with probability one occurs certainly}.

Nielsen and Chuang~\cite{NC00} is a seminal textbook on
quantum computation and quantum information.
In Nielsen and Chuang~\cite{NC00} we can also find several
descriptions
which are based the fact that an elementary event with probability one occurs certainly.
We here point out two descriptions among them
as follows.

First,
Nielsen and Chuang~\cite[Section~2.2.4]{NC00} considers a quantum measurement to
distinguish
states which are chosen
from
a fixed
\emph{orthonormal} set $\{\ket{\Psi_i}\}$ of quantum states.
Nielsen and Chuang~\cite[Section~2.2.4]{NC00} then
states that
``if
the state $\ket{\Psi_i}$ is prepared then $p(i)=\bra{\Psi_i}M_i\ket{\Psi_i}=1$,
so the result $i$ occurs with certainty'' (p.86),
where $M_i:=\ket{\Psi_i}\bra{\Psi_i}$.
In the context in which this description appears, the quantities $p(i)$ and $\bra{\Psi_i}M_i\ket{\Psi_i}$
mean the probability that
result $i$ occurs.
Therefore,
this description is based on the fact that
\emph{an elementary event with probability one occurs certainly}.

Mathematically, the statistics of outcomes in a quantum measurement are described by
a positive operator-valued measure (POVM) in the most general setting.
Such a quantum measurement is called a \emph{POVM measurement},
and is realized based on Postulates~\ref{state_space},
\ref{composition}, \ref{evolution}, and \ref{measurements}.
Nielsen and Chuang~\cite[Section~2.2.6]{NC00} considers a situation that
Alice gives Bob a qubit prepared in one of two states
$\ket{\Psi_1}$ or $\ket{\Psi_2}$, which are not orthogonal to each other,
and states the following:
\begin{quotation}
it is impossible for Bob to determine whether he has been given $\ket{\Psi_1}$ or $\ket{\Psi_2}$
with perfect reliability.
However, it is possible for him to perform a measurement which distinguishes the states some
of the
time,
but \emph{never}
makes an error of mis-identification.
(italics not mine, p.92)
\end{quotation}
Nielsen and Chuang~\cite[Section~2.2.6]{NC00} then
explains that a certain POVM measurement can
properly
serve
as the measurement mentioned in the description above.
Subsequently,
Nielsen and Chuang~\cite[Section~2.2.6]{NC00} again states that
``Bob \emph{never}
makes a mistake identifying the state he has been given''
(italics not mine, p.92).
According to the argument of Nielsen and Chuang~\cite[Section~2.2.6]{NC00},
this infallibility is thought to be confirmed by requiring that
\emph{an elementary event with probability zero never occurs}.
As we will see in
Section~\ref{Equiv-NOEEPZ} below,
according to our intuitive understanding of the notion of probability,
the statement that
an elementary event with probability
zero
never occurs
is equivalent to the statement that
an elementary event with probability one occurs certainly.

In this manner,
Nielsen and Chuang~\cite{NC00}
assumes
that
\emph{an elementary event with probability one occurs certainly}.

Finally, we refer to
a
description about the meaning of probability one,
found in a seminal paper
discussing the
\emph{completeness}
of the description given by quantum mechanics.
Einstein, Podolsky, and Rosen~\cite{EPR35} tried to show that
the description given by quantum mechanics is not complete.
For that purpose,
they first
propose
a necessary condition for a physical theory to be \emph{complete}.
This
condition is that
every \emph{element of the physical reality} must have a counterpart in
a
physical theory.
Then,
as the criterion of an element of the physical reality,
Einstein, Podolsky, and Rosen~\cite{EPR35}
propose
the following:
``\emph{If, without in any way disturbing a system, we can predict with certainty
(i.e., with probability equal to unity) the value of a physical quantity, then there exists an element of
physical reality corresponding to this physical quantity}''
(italics not mine, p.777).
This criterion applies to a physical theory in general,
including quantum mechanics.

Now, we can find a description about the meaning of probability one
in the criterion referred to above.
Namely,
we can identify
there
a description which states
that \emph{the prediction with certainty is equivalent to the prediction with probability equal to unity}.
Thus,
apart from the validity of their argument to show that the description given by quantum mechanics is not complete,
Einstein, Podolsky, and Rosen~\cite{EPR35} realize that
\emph{the prediction with certainty is equivalent to the prediction with probability one}
and therefore
they
realize that \emph{an
event with probability one occurs certainly}.

\subsubsection{Non-occurrence of an elementary event with probability zero}
\label{Equiv-NOEEPZ}

In general,
\emph{according to our intuitive understanding of the notion of probability},
the following two \emph{naive} statements are \emph{intuitively} equivalent to each other:
\begin{enumerate}
\item \emph{An elementary event with probability one occurs certainly}.
That is,
for every alphabet $\Omega$, $P\in\PS(\Omega)$, and $a\in\Omega$, if $P(a)=1$ then
an infinite sequence $\alpha\in\Omega^\infty$ of outcomes which is being generated
by infinitely repeated trials
\emph{described by}
the finite probability space $P$ on $\Omega$
consists only of $a$.
\item \emph{An elementary event with probability zero never occurs}.
That is,
for every alphabet $\Omega$, $P\in\PS(\Omega)$, and $a\in\Omega$, if $P(a)=0$ then
an infinite sequence $\alpha\in\Omega^\infty$ of outcomes which is being generated
by infinitely repeated trials
\emph{described by}
the finite probability space $P$ on $\Omega$
does not contain $a$.
\end{enumerate}
Based on an \emph{intuitive} arguments, the
intuitive
equivalence
of the above two naive statements
can be
demonstrated as follows:
First, assuming the statement~(i) we
show the statement~(ii).
Let $\Omega$ be an alphabet. Let $P\in\PS(\Omega)$, and let $a\in\Omega$.
Suppose that $P(a)=0$.
Let $\alpha\in\Omega^\infty$ be an infinite sequence of outcomes which is being generated
by infinitely repeated trials described by the finite probability space $P$ on $\Omega$.
We choose any particular $t\notin\Omega$ and define
$\beta$
as an infinite sequence over the alphabet $\{a,t\}$
obtained by replacing all occurrences of elements of $\Omega\setminus \{a\}$ in $\alpha$ by $t$.
Then,
according to our intuitive understanding of the notion of probability,
$\beta$
\emph{can be regarded as}
an infinite sequence of outcomes which is being generated
by infinitely repeated trials described by a finite probability space $Q$ on $\{a,t\}$,
where
$Q\in\PS(\{a,t\})$ is defined by the condition that
$Q(a):=P(a)$ and
\[
  Q(t):=\sum_{x\in\Omega\setminus\{a\}}P(x).
\]
Since $P(a)=0$ and $\sum_{x\in\Omega} P(x)=1$, we see that $Q(t)=1$.
Thus, it follows from the statement~(i) that the infinite sequence $\beta$
consists only of $t$.
This implies that the infinite sequence $\alpha$ does not contain $a$, as desired.

Next, assuming the statement~(ii) we show the statement~(i).
Let $\Omega$ be an alphabet. Let $P\in\PS(\Omega)$, and let $a\in\Omega$.
Suppose that $P(a)=1$.
Let $\alpha\in\Omega^\infty$ be an infinite sequence of outcomes which is being generated
by infinitely repeated trials described by the finite probability space $P$ on $\Omega$.
Then,
since $P(a)=1$ and $\sum_{x\in\Omega} P(x)=1$,
we see that $P(x)=0$ for every $x\in\Omega\setminus\{a\}$.
Thus, it follows from the statement~(ii) that the infinite sequence $\alpha$ does not contain $x$
for every $x\in\Omega\setminus\{a\}$.
This implies that the infinite sequence $\alpha$ consists only of $a$, as desired.

In Sections~\ref{Reason-MQSE} --- \ref{Reason-UDMP},
we have demonstrated that
an elementary event with probability one occurs certainly in quantum mechanics.
Combining this with
the intuitive equivalence demonstrated above
suggests that
\emph{an elementary event with probability zero never occurs in quantum mechanics}.

\subsubsection{Ensembles}
\label{one_probability-zero_probability-for-ensembles}

Now, let us
go back to
the consideration on the notion of ensemble \emph{with full mathematical rigor}.
We
see that both the naive statements~(i) and (ii) above \emph{can be justified} for an ensemble.
Theorem~\ref{one_probability} below states that
\emph{an elementary event with probability one always occurs in an ensemble},
and thus shows that the notion of ensemble coincides with our intuition about
the notion of probability, in particular, in quantum mechanics,
which we have investigated so far.

\begin{theorem}\label{one_probability}
Let $\Omega$ be an alphabet, and let $P\in\PS(\Omega)$. Let $a\in\Omega$.
Suppose that $\alpha$ is an ensemble for the finite probability space $P$ and $P(a)=1$.
Then $\alpha$ consists only of $a$,
i.e.,
$\alpha=aaaaaa\dotsc\dotsc$.
\qed
\end{theorem}

To show Theorem~\ref{one_probability}, we first show Theorem~\ref{zero_probability} below.
The result~(i) of Theorem~\ref{zero_probability} states that
\emph{an elementary event with probability zero never occurs in an ensemble}.
This result~(i)
was, in essence, pointed out by Martin-L\"{o}f~\cite{M66}.

\begin{theorem}\label{zero_probability}
Let $\Omega$ be an alphabet, and let $P\in\PS(\Omega)$.
\begin{enumerate}
\item Let $a\in\Omega$.
  Suppose that $\alpha$ is an ensemble for the finite probability space $P$ and $P(a)=0$.
  Then $\alpha$ does not contain $a$.
\item Actually, there exists a single Martin-L\"of $P$-test
  $\mathcal{C}\subset\N^+\times\Omega^*$
  such that, for every $\alpha\in \Omega^\infty$, if $\alpha$ passes
  $\mathcal{C}$
  then $\alpha$ does not contain any element of $P^{-1}(\{0\})$.
\end{enumerate}
\end{theorem}

\begin{proof}
It is sufficient to prove
the result~(ii)
of Theorem~\ref{zero_probability}.
For that purpose, we first define a prefix-free subset $S$ of $\Omega^*$ by
\[
  S:=\left\{\tau x\bigm| \tau\in(\Omega\setminus P^{-1}(\{0\}))^*\;\&\; x\in P^{-1}(\{0\})\right\},
\]
and then define $\mathcal{C}$ as the set $\{(n,\sigma)\mid n\in\N^+\;\&\;\sigma\in S\}$.
Since $S$ is r.e., $\mathcal{C}$ is also r.e., obviously.
Since $S$ is prefix-free, $\mathcal{C}_n$ is also prefix-free for every $n\in\N^+$.
Moreover, since $P(\sigma)=0$ for every $\sigma\in S$,
we have $\Bm{P}{\osg{\mathcal{C}_n}}=P(\mathcal{C}_n)=P(S)=0$ for each $n\in\N^+$.
Hence, $\mathcal{C}$ is Martin-L\"of $P$-test.

Let $\alpha\in \Omega^\infty$. Suppose that $\alpha$ passes $\mathcal{C}$.
Assume contrarily that $\alpha$ contains some element of $P^{-1}(\{0\})$.
Then we denote by $\sigma_0$ the shortest prefix of $\alpha$ which contains some element of $P^{-1}(\{0\})$.
It follows that $\sigma_0\in S$, and therefore $\alpha\in\osg{\mathcal{C}_n}$ for all $n\in\N^+$.
Thus,
we have a contradiction.
Hence,
the proof of the result~(ii) is completed.
\end{proof}

Note that we do not require the underlying finite probability space $P$ to be computable at all
in Theorem~\ref{zero_probability}.
Theorem~\ref{one_probability} is then proved as follows.

\begin{proof}[Proof of Theorem~\ref{one_probability}]
Let $\Omega$ be an alphabet, and let $P\in\PS(\Omega)$.
Let $a\in\Omega$. Suppose that $\alpha$ is an ensemble for the finite probability space $P$ and $P(a)=1$.
Then, since $P(a)=1$ and $\sum_{x\in\Omega} P(x)=1$,
we see that $P(x)=0$ for every $x\in\Omega\setminus\{a\}$.
Hence,
it follows from the result~(i) of Theorem~\ref{zero_probability} that
$\alpha$ does not contain $x$ for every $x\in\Omega\setminus\{a\}$.
This implies that $\alpha$ consists only of $a$, as desired.
\end{proof}

\subsection{The law of large numbers}
\label{Subsec:The law of large numbers}

Let $\Omega$ be an alphabet, and let $P\in\PS(\Omega)$. Let us consider
an infinite sequence $\alpha\in\Omega^\infty$ of outcomes
which is being generated by infinitely repeated trials described
by the finite probability space $P$ on $\Omega$.
The
second
necessary condition which the notion of probability for the finite probability space $P$
is considered to have to satisfy is the condition that
\emph{the law of large numbers holds for $\alpha$}.
Theorem~\ref{LLN} below confirms that this condition certainly holds for an \emph{arbitrary} ensemble.

\begin{theorem}[The law of large numbers]\label{LLN}
Let $\Omega$ be an alphabet, and let $P\in\PS(\Omega)$.
For every $\alpha\in\Omega^\infty$, if $\alpha$ is an ensemble for $P$
then for every $a\in\Omega$ it holds that
\begin{equation*}
  \lim_{n\to\infty}\frac{N_a(\rest{\alpha}{n})}{n}=P(a),
\end{equation*}
where $N_a(\sigma)$ denotes the number of the occurrences of $a$ in $\sigma$ for every $a\in\Omega$ and $\sigma\in\Omega^*$.
\qed
\end{theorem}

We remark that,
in Theorem~\ref{LLN},
the underlying finite probability space $P$ is \emph{quite arbitrary}.
In particular,
we do not require $P$ to be computable at all in Theorem~\ref{LLN}.
This
generality
is \emph{essential} to the
universality
of our framework.

In order to prove Theorem~\ref{LLN}, we need the following theorem, Chernoff bound.
This form of Chernoff bound follows from Theorem~4.2 of Motwani and Raghavan~\cite{MR95}.

\begin{theorem}[Chernoff bound]\label{Chernoff_bound}
Let $P\in\PS(\{0,1\})$ with $0<P(1)<1$,
and let
$\varepsilon\in\R$
with $0<\varepsilon\le \min\{P(0),P(1)\}$.
Then, for every $n\in\N^+$, we have
$$\Bm{P}{\osg{S_n}}<2\exp(-\varepsilon^2 n/2),$$
where $S_n$ is the set of all $\sigma\in\{0,1\}^n$ such that
$\abs{N_1(\sigma)/n-P(1)}>\varepsilon$.
\qed
\end{theorem}

In order to prove Theorem~\ref{LLN}, we also need the following theorem.

\begin{theorem}\label{contraction}
Let $\Omega$ be an alphabet, and let $P\in\PS(\Omega)$.
Let $\alpha$ be an ensemble for $P$, and
let $a$ and $b$ be distinct elements
of
$\Omega$.
Suppose that $\beta$ is an infinite sequence
over $\Omega\setminus\{b\}$
obtained by replacing all occurrences of $b$ by $a$ in $\alpha$.
Then $\beta$ is
an ensemble for $Q$,
where $Q\in\PS(\Omega\setminus\{b\})$ such that $Q(x):=P(a)+P(b)$ if $x=a$ and $Q(x):=P(x)$ otherwise.
\end{theorem}

\begin{proof}
We show the contraposition. Suppose that $\beta$ is not Martin-L\"of $Q$-random.
Then there exists a Martin-L\"of $Q$-test $\mathcal{S}\subset\N^+\times(\Omega\setminus\{b\})^*$
such that
\begin{equation}\label{betainosgmSn_LemmaLLN}
  \beta\in\osg{\mathcal{S}_n}
\end{equation}
for every $n\in\N^+$.
For each $\sigma\in(\Omega\setminus\{b\})^*$,
let $F(\sigma)$ be the set of all $\tau\in\Omega^*$ such that
$\tau$ is obtained by replacing some or none of the occurrences of $a$ in $\sigma$, if exists, by $b$.
Note that if $\sigma$ has exactly $n$ occurrences of $a$ then $\#F(\sigma)=2^n$.
Then, since $Q(a)=P(a)+P(b)$, we have that
\begin{equation}\label{BmQ=Q=P=BmP_LemmaLLN}
  \Bm{Q}{\osg{\sigma}}=Q(\sigma)=P(F(\sigma))=\Bm{P}{\osg{F(\sigma)}}
\end{equation}
for each $\sigma\in(\Omega\setminus\{b\})^*$.
We then define $\mathcal{T}$ to be a subset of $\N^+\times \Omega^*$ such that
$\mathcal{T}_n=\bigcup_{\sigma\in\mathcal{S}_n} F(\sigma)$ for every $n\in\N^+$.
Since $\mathcal{S}_n$ is a prefix-free subset of $(\Omega\setminus\{b\})^*$ for every $n\in\N^+$,
we see that $\mathcal{T}_n$ is a prefix-free subset of $\Omega^*$ for every $n\in\N^+$.
For each $n\in\N^+$, we also see that
\[
  \Bm{P}{\osg{\mathcal{T}_n}}
  \le\sum_{\sigma\in\mathcal{S}_n}\Bm{P}{\osg{F(\sigma)}}
  =\sum_{\sigma\in\mathcal{S}_n}\Bm{Q}{\osg{\sigma}}
  =\Bm{Q}{\osg{\mathcal{S}_n}}<2^{-n},
\]
where the first equality follows from \eqref{BmQ=Q=P=BmP_LemmaLLN} and
the second equality follows from the prefix-freeness of $\mathcal{S}_n$.%
\footnote{Actually, the first inequality can be strengthened to an equality.}
Moreover, since $\mathcal{S}$ is r.e., $\mathcal{T}$ is also r.e.
Thus, $\mathcal{T}$ is a Martin-L\"of $P$-test.

On the other hand,
note that, for every $n\in\N^+$, if $\beta\in\osg{\mathcal{S}_n}$ then $\alpha\in\osg{\mathcal{T}_n}$.
Thus, it follows from \eqref{betainosgmSn_LemmaLLN}
that $\alpha\in\osg{\mathcal{T}_n}$ for every $n\in\N^+$.
Hence, $\alpha$ is not Martin-L\"of $P$-random.
This completes the proof.
\end{proof}

Theorem~\ref{LLN} is then proved as follows.

\begin{proof}[Proof of Theorem~\ref{LLN}]
Let $\Omega$ be an alphabet, and let $P\in\PS(\Omega)$.
Suppose that
$\alpha$
is an ensemble for $P$.
Let $a\in\Omega$.
In the case of $P(a)=0$, the result follows immediately from
the result~(i)
of Theorem~\ref{zero_probability}.
In the case of $P(a)=1$, the result follows immediately from Theorem~\ref{one_probability}.
Thus we assume that $0<P(a)<1$, in what follows.

We define $Q\in\PS(\{0,1\})$
by the condition that
$Q(1):=P(a)$ and $Q(0):=1-P(a)$.
Then
$0<Q(1)<1$.
Let $\beta$ be the infinite binary sequence obtained from $\alpha$
by replacing all $a$ by $1$ and
all other elements of $\Omega$
by $0$ in $\alpha$.
Then, by using Theorem~\ref{contraction}
repeatedly, it is easy to show that
$\beta$ is Martin-L\"of $Q$-random and
$N_1(\rest{\beta}{n})=N_a(\rest{\alpha}{n})$ for every $n\in\N^+$.

Now, let us assume contrarily
that $\lim_{n\to\infty}N_a(\rest{\alpha}{n})/n\neq P(a)$.
Then $\lim_{n\to\infty}N_1(\rest{\beta}{n})/n\neq Q(1)$ and
therefore there exists
$\varepsilon\in\R$ with $0<\varepsilon\le \min\{Q(0),Q(1)\}$ such that
\begin{equation}\label{LLN-EN0}
  \abs{N_1(\rest{\beta}{n})/n-Q(1)}>2\varepsilon
\end{equation}
for infinitely many $n\in\N^+$.
On the other hand, it follows from Theorem~\ref{Chernoff_bound} that
\begin{equation}\label{LLN-rCb}
  \Bm{Q}{\osg{\{\sigma\in\{0,1\}^n\mid\abs{N_1(\sigma)/n-Q(1)}>\varepsilon\}}}
  <2\exp(-\varepsilon^2 n/2)
\end{equation}
for every $n\in\N^+$.
Recall that $Q(1)$ is not necessarily computable.%
\footnote{In Theorem~\ref{LLN} we do not require the finite probability space $P$ to be computable at all.
For
a
\emph{computable} finite probability space
in contrast,
by \emph{modifying
the proof of
Theorem~\ref{LLN}
slightly}
we can
generally prove the following:
Let $R\in\PS(\{0,1\})$ with $0<R(1)<1$.
Suppose that $R$ is computable.
Then there exists a \emph{single} Martin-L\"of $R$-test $\mathcal{U}$ such that,
for every $\gamma\in\{0,1\}^\infty$,
if $\gamma$ passes $\mathcal{U}$ then the law of large numbers holds for $\gamma$, i.e.,
$\lim_{n\to\infty}N_a(\rest{\gamma}{n})/n=R(a)$ holds for all $a\in\{0,1\}$.}
Thus,
we choose $r_L,r_R\in\Q$ such that
$Q(1)-2\varepsilon<r_L<Q(1)-\varepsilon$ and $Q(1)+\varepsilon<r_R<Q(1)+2\varepsilon$.
For each $n\in\N^+$,
let $S_n$ be the set
\begin{equation*}
  \{\sigma\in\{0,1\}^n\mid N_1(\sigma)/n<r_L\text{ or }r_R<N_1(\sigma)/n\}
\end{equation*}
and let $T_n=\bigcup_{m=n}^\infty S_m$.
Then, on the one hand, using \eqref{LLN-EN0}
we have that $\beta\in\osg{T_n}$ for every $n\in\N^+$.
Using \eqref{LLN-rCb}, on the other hand, for each $n\in\N^+$ we see that
\begin{equation*}
  \Bm{Q}{\osg{T_n}}\le\sum_{m=n}^\infty\Bm{Q}{\osg{S_m}}
  <\sum_{m=n}^\infty 2\exp(-\varepsilon^2 m/2)
  \le\sum_{m=n}^\infty 2c^{m}=\frac{2c^{n}}{1-c},
\end{equation*}
where $c$ is a specific rational
with $\exp(-\varepsilon^2/2)\le c<1$.
It
is easy to show that
there exists a total recursive function $f\colon\N^+\to\N^+$ such that
$2c^{f(n)}/(1-c)\le 2^{-n}$.
Thus, we see that $\beta\in\osg{T_{f(n)}}$ and $\Bm{Q}{\osg{T_{f(n)}}}<2^{-n}$ for every $n\in\N^+$,
and the subset $\{(n,\sigma)\mid n\in\N^+\;\&\;\sigma\in T_{f(n)}\}$ of $\N^+\times \Omega^*$ is r.e.
It follows from Theorem~\ref{ML_P-randomness_eliminated-prefix-freeness} that
$\beta$ is not Martin-L\"of $Q$-random.
Hence we have a contradiction, and the result follows.
\end{proof}

We here remark that
the notion of probability is more than the law of large numbers.
To see this, let us consider a finite probability space $P\in\PS(\{a,b\})$ such that $P(a)=1$ and $P(b)=0$,
and consider an infinite sequence
$$\alpha=abaaaaaaaaaa\dotsc\dotsc$$
over $\{a,b\}$.
Then, since
$\lim_{n\to\infty}N_a(\rest{\alpha}{n})/n=1=P(a)$
and $\lim_{n\to\infty}N_b(\rest{\alpha}{n})/n=0=P(b)$,
the law of large numbers certainly holds for $\alpha$.
However,
$\alpha$
cannot be thought of as
an infinite sequence of outcomes
which is being generated by infinitely repeated trials described by
the finite probability space $P$ on $\{a,b\}$.
This is because the elementary event $b$ with probability zero
has occurred
once in $\alpha$.
Such an occurrence
contradicts our intuition that
\emph{an elementary event with probability zero never occurs},
which is our conclusion
in Section~\ref{Subsec:SOEEPO} above
from the aspect of the notion of probability, in particular, in quantum mechanics.
This
example
shows that the law of large numbers is insufficient to characterize the notion of probability,
and \emph{the notion of probability is more than the law of large numbers}.

The following is immediate from Theorem~\ref{LLN},
and will be often used for developing our framework.

\begin{corollary}\label{uniquness}
Let $\Omega$ be an alphabet, and let $P,Q\in\PS(\Omega)$.
If there exists $\alpha\in\Omega^\infty$ which is both
an ensemble for $P$ and an ensemble for $Q$,
then $P=Q$.
\qed
\end{corollary}

It is worthwhile to
investigate
some specific
application of Theorem~\ref{LLN} here.
The \emph{asymptotic equipartition property} (AEP)
plays an important role in the source coding problem in information theory,
and is a direct consequence of the weak law of large numbers
for independent, identically distributed random variables
in the \emph{conventional} probability theory.
The AEP is stated in terms of the notion of Shannon entropy.
See Cover and Thomas~\cite[Chapter 3]{CT06} for the details
of the AEP and its applications, where the AEP is stated as Theorem~3.1.1.

Using Theorem~\ref{LLN} we can show
that ensembles for an arbitrary finite probability space $P$ have the AEP in a sense stated in Theorem~\ref{AEP} below.
To show this,
we first introduce the notion of Shannon entropy:
Let $\Omega$ be an alphabet, and let $P\in\PS(\Omega)$. The \emph{Shannon entropy} $H(P)$ of $P$ is defined by
\begin{equation}\label{def:Shannon-entropy}
  H(P):=-\sum_{a\in\Omega}P(a)\log_2 P(a),
\end{equation}
where $0\log_2 0$ is defined to be $0$ as usual.

\begin{theorem}[AEP]\label{AEP}
Let $\Omega$ be an alphabet, and let $P\in\PS(\Omega)$.
For every $\alpha\in\Omega^\infty$, if $\alpha$ is an ensemble for $P$ then
$\Bm{P}{\osg{\rest{\alpha}{n}}}>0$ for every $n\in\N^+$ and
\begin{equation*}
  \lim_{n\to\infty}\frac{-\log_2\Bm{P}{\osg{\rest{\alpha}{n}}}}{n}=H(P).
\end{equation*}
\end{theorem}

\begin{proof}
We denote by $\Omega_e$ the set $\{a\in\Omega\mid P(a)>0\}$.
Since $\alpha$ is an ensemble for $P$,
it follows from
the result~(i)
of Theorem~\ref{zero_probability} that
$\alpha(n)\in\Omega_e$ for every $n\in\N^+$.
Therefore, for each $n\in\N^+$, using~\eqref{pBm} we have that
\begin{equation*}
  \Bm{P}{\osg{\rest{\alpha}{n}}}=P(\rest{\alpha}{n})=\prod_{k=1}^n P(\alpha(k))>0.
\end{equation*}
Thus, for each $n\in\N^+$, we see that
\begin{equation}\label{AEP:eq1-0}
\frac{\log_2\Bm{P}{\osg{\rest{\alpha}{n}}}}{n}=\frac{1}{n}\sum_{k=1}^n\log_2 P(\alpha(k))
=\frac{1}{n}\sum_{a\in\Omega_e}N_a(\rest{\alpha}{n})\log_2 P(a)
=\sum_{a\in\Omega_e}\frac{N_a(\rest{\alpha}{n})}{n}\log_2 P(a),
\end{equation}
where $N_a(\rest{\alpha}{n})$ denotes the number of the occurrences of $a$ in $\rest{\alpha}{n}$
for every $a\in\Omega$ as before.
It follows from Theorem~\ref{LLN} that,
on letting $n\to\infty$,
the limit value of the most right-hand side of \eqref{AEP:eq1-0} equals
\begin{equation*}
  \sum_{a\in\Omega_e}P(a)\log_2 P(a)=\sum_{a\in\Omega}P(a)\log_2 P(a)=-H(P),
\end{equation*}
as desired.
This completes the proof.
\end{proof}

Note that the underlying finite probability space $P$ is \emph{quite arbitrary} in Theorem~\ref{AEP}.

\emph{Intuitively},
Theorem~\ref{AEP}
means
that, for every ensemble $\alpha$ for $P$,
\begin{equation}\label{AEP_Intuition:eq1-0}
  \Bm{P}{\osg{\rest{\alpha}{n}}} \approx 2^{-nH(P)}
\end{equation}
holds
\emph{asymptotically} with respect to $n$.
Since the right-hand side of \eqref{AEP_Intuition:eq1-0} does not depend on $\alpha$,
all ensembles for $P$ have the same measure $2^{-nH(P)}$
\emph{asymptotically}
in a certain sense.
On the other hand, recall Theorem~\ref{Bmae} which states that
the set of all ensembles for $P$ has
measure one with respect to the Bernoulli measure $\lambda_{P}$.
These observations
may justify
the name ``asymptotic equipartition property.''

\subsection{Computable shuffling}

This subsection considers the third necessary condition which
the notion of probability for a finite probability space is considered to have to satisfy.

Let $\Omega$ be an alphabet, and let $P\in\PS(\Omega)$.
Assume that an observer $A$ performs an infinite repetition of
a trial
described by
the finite probability space $P$,
and thus is generating an infinite sequence $\alpha\in\Omega^\infty$ of outcomes of
the trials as
\[
  \alpha=a_1 a_2 a_3 a_4 a_5 a_6 a_7 a_8 \dotsc\dotsc
\]
with $a_i\in\Omega$. 
According to our thesis, Thesis~\ref{thesis},
$\alpha$ is an ensemble for $P$.
Consider another observer $B$ who wants to adopt the following subsequence $\beta$ of $\alpha$
as the outcomes of
the trials:
\[
  \beta=a_2 a_3 a_5 a_7 a_{11} a_{13} a_{17} \dotsc\dotsc,
\]
where the observer $B$
takes into account
only
the $n$th elements $a_n$ in the original sequence $\alpha$
such that $n$ is a prime number.
According to Thesis~\ref{thesis},
$\beta$ has to be an ensemble for $P$, as well.
However, is this true?

Consider this problem in a general setting.
Assume as before that an observer $A$ performs an infinite repetition of
a trial
described
by the finite probability space $P$,
and thus is generating an infinite sequence $\alpha\in\Omega^\infty$ of outcomes of
the trials.
According to Thesis~\ref{thesis},
$\alpha$ is an ensemble for $P$.
Now, let $f\colon\N^+\to\N^+$ be an injection.
Consider another observer $B$ who wants to adopt the following sequence $\beta$
as the outcomes of
the trials:
\[
  \beta=\alpha(f(1))\alpha(f(2))\alpha(f(3))\alpha(f(4))\alpha(f(5))\dotsc\dotsc,
\]
instead of $\alpha$.
According to Thesis~\ref{thesis},
$\beta$ has to be an ensemble for $P$, as well.
However, is this true?

We can confirm this \emph{by restricting the ability of $B$}, that is, by assuming that
every observer can select elements from the original sequence $\alpha$
\emph{only in an effective manner}.
This means that the function $f\colon\N^+\to\N^+$ has to be a
\emph{computable function}.
Theorem~\ref{cpucs} below shows this result.

In other words,
Theorem~\ref{cpucs} states that ensembles for $P$ are \emph{closed under computable shuffling}.
Note
also
that the underlying finite probability space $P$ \emph{itself} is \emph{quite arbitrary} in Theorem~\ref{cpucs}.

\begin{theorem}[Closure property under computable shuffling%
]\label{cpucs}
Let $\Omega$ be an alphabet, and let $P\in\PS(\Omega)$. Let $\alpha$ be an ensemble for $P$.
Then, for every injective function $f\colon\N^+\to\N^+$, if $f$ is computable then the infinite sequence
\begin{equation*}
  \alpha_f:=\alpha(f(1)) \alpha(f(2)) \alpha(f(3)) \alpha(f(4)) \dotsc\dotsc\dotsc
\end{equation*}
is an ensemble for $P$.
\end{theorem}

\begin{proof}
We show the contraposition.
Suppose that $\alpha_f$ is not Martin-L\"of $P$-random.
Then there exists a Martin-L\"of $P$-test $\mathcal{C}\subset\N^+\times \Omega^*$ such that
\begin{equation}\label{afinosgmCn_CS}
  \alpha_f\in\osg{\mathcal{C}_n}
\end{equation}
for every $n\in\N^+$.
For each $\sigma\in\Omega^+$, let $F(\sigma)$ be the set of all $\tau\in\Omega^+$ such that
\begin{enumerate}
  \item $\abs{\tau}=\max f(\{1,2,\dots,\abs{\sigma}\})$, and
  \item for every $k=1,2,\dots,\abs{\sigma}$ it holds that $\sigma(k)=\tau(f(k))$.
\end{enumerate}
Then, since $f$ is an injection and $\sum_{a\in\Omega}P(a)=1$, we have that
\begin{equation}\label{BPoFslePFs=Ps=BPos_CS}
  \Bm{P}{\osg{F(\sigma)}}=P(F(\sigma))=P(\sigma)=\Bm{P}{\osg{\sigma}}
\end{equation}
for each $\sigma\in\Omega^+$.
We then define $\mathcal{D}$ to be a subset of $\N^+\times \Omega^*$ such that
$\mathcal{D}_n=\bigcup_{\sigma\in\mathcal{C}_n} F(\sigma)$ for every $n\in\N^+$.
Note here that, for each $n\in\N^+$,
$\lambda\notin\mathcal{C}_n$ since $\Bm{P}{\osg{\mathcal{C}_n}}<2^{-n}<1$.
Then, since $\mathcal{C}_n$ is a prefix-free subset of $\Omega^*$ for every $n\in\N^+$,
we see that $\mathcal{D}_n$ is also a prefix-free subset of $\Omega^*$ for every $n\in\N^+$.
For each $n\in\N^+$, we see that
\[
  \Bm{P}{\osg{\mathcal{D}_n}}
  \le\sum_{\sigma\in\mathcal{C}_n}\Bm{P}{\osg{F(\sigma)}}
  =\sum_{\sigma\in\mathcal{C}_n}\Bm{P}{\osg{\sigma}}
  =\Bm{P}{\osg{\mathcal{C}_n}}<2^{-n},
\]
where the first equality follows from \eqref{BPoFslePFs=Ps=BPos_CS} and
the second equality follows from the prefix-freeness of $\mathcal{C}_n$.
Moreover,
since $f$ is an injective computable function and $\mathcal{C}$ is r.e.,
it is easy to see that $\mathcal{D}$ is r.e.
Thus, $\mathcal{D}$ is a Martin-L\"of $P$-test.

On the other hand, we see that,
for every $n\in\N^+$, if $\alpha_f\in\osg{\mathcal{C}_n}$ then $\alpha\in\osg{\mathcal{D}_n}$.
Thus, it follows from \eqref{afinosgmCn_CS} that
$\alpha\in\osg{\mathcal{D}_n}$ for every $n\in\N^+$.
Hence, $\alpha$ is not Martin-L\"of $P$-random.
This completes the proof.
\end{proof}

\subsection{Selection by partial computable selection functions}

As the forth necessary condition which
the notion of probability for a finite probability space $P$ on
an alphabet
$\Omega$ is considered to have to satisfy,
in this subsection
we consider the condition that infinite sequences
of outcomes in $\Omega$
each of which
is obtained by an infinite repetition of
a trial
described by the finite probability space $P$
are \emph{closed under
the selection by a partial computable selection function used in the definition of
von Mises-Wald-Church stochasticity}.
The notion of \emph{von Mises-Wald-Church stochasticity}
is investigated in the theory of collectives \cite{vM57,vM64,Wa36,Wa37,Ch40}.
To state the forth necessary condition,
we use the notion of the \emph{selection by a partial computable selection function},
by means of which the notion of von Mises-Wald-Church stochasticity is defined.
See Downey and Hirschfeldt~\cite[Section 7.4]{DH10}
for a treatment of the mathematics of the notion of von Mises-Wald-Church stochasticity
from a modern point of view,
although
we do not study
von Mises-Wald-Church stochasticity \emph{itself} in this paper.

For motivating the forth necessary condition,
we carry out a thought experiment
in what follows,
as in the preceding subsection:
Let $\Omega$ be an alphabet, and let $P\in\PS(\Omega)$. Assume that
an observer $A$ performs an infinite repetition of
a trial
described by
the finite probability space $P$,
and thus is generating an infinite sequence $\alpha\in\Omega^\infty$ of outcomes of
the trials as
\[
  \alpha=a_1 a_2 a_3 a_4 a_5 a_6 \dotsc\dotsc
\]
with $a_i\in\Omega$.
According to Thesis~\ref{thesis},
$\alpha$ is an ensemble for $P$.

Consider another observer $B$ who wants to \emph{refute} Thesis~\ref{thesis}.
For that purpose, the observer $B$ adopts a subsequence
\[
  \beta=b_1 b_2 b_3 b_4 b_5 b_6 \dotsc\dotsc
\]
of $\alpha$ in the following manner:
Whenever a new outcome $a_n$ is generated by the observer $A$,
the observer $B$
investigates
the prefix $a_1 a_2 a_3 \dots a_n$ of $\alpha$ generated so far
by the observer $A$.
Then,
based \emph{only} on the prefix $a_1 a_2 a_3 \dots a_n$,
for aiming at refuting Thesis~\ref{thesis} on $\beta$,
the observer $B$ decides whether the next outcome $a_{n+1}$ should be appended
to the tail of $b_1 b_2 b_3 \dots b_k$
which
have
been adopted so far by $B$ as a prefix of $\beta$.
In this manner,
the observer $B$ is generating  the subsequence $\beta$ of $\alpha$
for aiming at refuting Thesis~\ref{thesis}
on $\beta$.
Note that the length of
the subsequence
$\beta$
generated in this manner
may or may not be infinite.

In contrast,
the observer $A$ is a \emph{defender} of Thesis~\ref{thesis}.
Therefore, the observer $A$ tries to inhibit the observer $B$ from breaking Thesis~\ref{thesis}.
For that purpose, the observer $A$
never
generates
the next
outcome $a_{n+1}$
before the observer $B$ decides whether
this
$a_{n+1}$ should be appended
to the tail of $b_1 b_2 b_3 \dots b_k$.
This is because if for each $n$ the observer $B$ knows the outcome $a_{n+1}$
before the decision for $a_{n+1}$ to be appended or to be ignored,
then the observer $B$ can easily generate an infinite subsequence $\beta$ of $\alpha$
which does not satisfy Thesis~\ref{thesis}.
Thus, due to this careful behavior of the observer $A$,
the observer $B$
\emph{must}
make the decision of
adoption of
the next outcome $a_{n+1}$,
based
\emph{only}
on the prefix $a_1 a_2 a_3 \dots a_n$ of $\alpha$, without knowing the outcome $a_{n+1}$.
Then, according to Thesis~\ref{thesis},
this
$\beta$ has to be an ensemble for $P$,
as well as $\alpha$ is.
However, is this true?

We can confirm this by restricting the ability of $B$, that is, by assuming that
the observer $B$ can make the decision of
adoption of
the next outcome,
\emph{only in an effective manner}
based on the prefix $a_1 a_2 a_3 \dots a_n$ of $\alpha$ generated so far by the observer $A$.

Put more mathematically,
we introduce some notation
from Downey and Hirschfeldt~\cite[Sections~6.5 and~7.4]{DH10}.
A \emph{selection function} is a partial function $f\colon\Omega^*\to\{\mathrm{YES}, \mathrm{NO}\}$.
We think of $f$ as the decision of $B$ whether or not
to adopt
the next outcome $\alpha(n+1)$
based on the prefix
$\rest{\alpha}{n}$
of $\alpha$,
in generating $\beta$.
For any $\gamma\in\Omega^\infty$, $k\in\N^+$, and selection function $g$,
let $s_g (\gamma, k)$ be the $k$th number $\ell\in\N$ such that $g(\rest{\gamma}{\ell})=\mathrm{YES}$,
i.e., the least number $\ell\in\N$ such that $\#\{m\in\N\mid m\le \ell\;\&\;g(\rest{\gamma}{m})=\text{YES}\}=k$,
if such an $\ell$ exists.

First, consider the case where $f(\rest{\alpha}{n})$ is not defined for some $n\in\N$.
Let $m$ be the least number of such
an
$n$.
Then, this case means that
the observer $B$ does not make the decision of
adoption of
the next outcome $\alpha(m+1)$
based on the prefix $\rest{\alpha}{m}$, and is stalled.
Therefore, the length of
the subsequence
$\beta$ is
\emph{finite}
in this case.
Thus, the observer $B$ \emph{cannot refute} Thesis~\ref{thesis} in this case,
since Thesis~\ref{thesis} only refers to the property of an \emph{infinite} sequence of outcomes
which is being generated by \emph{infinitely} repeated trials.
Hence, Thesis~\ref{thesis} survives in this case.

Secondly, consider the case where
$f(\rest{\alpha}{n})$ is defined for all $n\in\N$ and
$\{n\in\N\mid f(\rest{\alpha}{n})=\mathrm{YES}\}$ is a finite set.
In this case, the length of
the subsequence
$\beta$ is also
\emph{finite}.
Thus, the observer $B$ fails to refute Thesis~\ref{thesis}, and therefore
Thesis~\ref{thesis} survives also in this case.

Finally, consider the remaining case, where $f(\rest{\alpha}{n})$ is defined for all $n\in\N$ and
the set $\{n\in\N\mid f(\rest{\alpha}{n})=\mathrm{YES}\}$ is infinite.
Then it follows that
$s_f (\alpha, k)$ is defined and $\beta(k)=\alpha(s_f (\alpha, k)+1)$ for all $k\in\N^+$.
Hence, $\beta$ is an infinite sequence over $\Omega$, 
and thus Thesis~\ref{thesis} can be applied to $\beta$ in this case.
Therefore,
according to Thesis~\ref{thesis}, $\beta$ has to be an ensemble for $P$,
as well as $\alpha$ is.
However, is this true?
Actually, we can confirm this \emph{by restricting the ability of $B$}, that is,
\emph{by assuming that
$f$ has to be a partial computable selection function}.
Here, a selection function $g\colon\Omega^*\to\{\mathrm{YES}, \mathrm{NO}\}$ is called
a \emph{partial computable selection function} if
$g\colon\Omega^*\to\{\mathrm{YES}, \mathrm{NO}\}$ is a partial computable function.
Theorem~\ref{cpuscsf} below shows this result.
It states that ensembles for an \emph{arbitrary} finite probability space are
\emph{closed under the selection by a partial computable selection function}.
Hence, Thesis~\ref{thesis} survives
in this case as well.

In this way, based on Theorem~\ref{cpuscsf}, we confirm that
the forth necessary condition certainly holds for ensembles for an \emph{arbitrary} finite probability space.

\begin{theorem}[Closure property under the selection by a partial computable selection function%
]\label{cpuscsf}
Let $\Omega$ be an alphabet, and let $P\in\PS(\Omega)$.
Let $\alpha$ be an ensemble for $P$, and let $f$ be a partial computable selection function.
Suppose that $f(\rest{\alpha}{k})$ is defined for all $k\in\N$ and
$\{k\in\N\mid f(\rest{\alpha}{k})=\mathrm{YES}\}$ is an infinite set. Then
an infinite sequence $\beta$ such that $\beta(k)=\alpha(s_f (\alpha, k)+1)$ for all $k\in\N^+$
is an ensemble for $P$.
\end{theorem}

\begin{proof}
We show the contraposition.
Suppose that $\beta$ is not Martin-L\"of $P$-random.
Then there exists a Martin-L\"of $P$-test $\mathcal{C}\subset\N^+\times \Omega^*$ such that
\begin{equation}\label{betainosgmathCn_CPUSPCF}
  \beta\in\osg{\mathcal{C}_n}
\end{equation}
for every $n\in\N^+$.
For
any
$\sigma,\tau\in\Omega^+$, we say that \emph{$\sigma$ is selected by $f$ from $\tau$} if
$f(\rest{\tau}{k})$ is defined for all $k=0,1,\dots,\abs{\tau}-1$ and
there exists a strictly increasing function $h\colon\{1,\dots,\abs{\sigma}\}\to\N$ such that
\begin{enumerate}
  \item $\{k\in\{1,\dots,\abs{\tau}\}\mid f(\rest{\tau}{k-1})=\mathrm{YES}\}=h(\{1,\dots,\abs{\sigma}\})$,
  \item $h(\abs{\sigma})=\abs{\tau}$, and
  \item $\tau(h(k))=\sigma(k)$ for all $k=1,\dots,\abs{\sigma}$.
\end{enumerate}
For each $\sigma\in\Omega^+$, let $F(\sigma)$ be the set of all $\tau\in\Omega^*$ such that
$\sigma$ is selected by $f$ from $\tau$.
We also set $F(\lambda):=\{\lambda\}$.
It is then easy to see that $F(\sigma)$ is a prefix-free
subset of $\Omega^*$
for every $\sigma\in\Omega^*$.

We show that
\begin{equation}\label{PFleP-Bs}
  \Bm{P}{\osg{F(\sigma)}}\le\Bm{P}{\osg{\sigma}}
\end{equation}
for all $\sigma\in\Omega^*$ by the induction on the length of $\abs{\sigma}$.
First, the inequality \eqref{PFleP-Bs} holds for the case of $\abs{\sigma}=0$, obviously.
For an arbitrary $n\in\N$, assume that \eqref{PFleP-Bs} holds for all $\sigma\in\Omega^n$.
Let $\sigma\in\Omega^{n+1}$.
We then denote the prefix of $\sigma$ of length $n$ by $\rho$, and denote $\sigma(\abs{\sigma})$ by $a$.
Therefore $\sigma=\rho a$.
Note that
\[
  G(\tau):=\{\upsilon\in\Omega^*\mid \tau\upsilon a\in F(\sigma)\}
\]
is a prefix-free
subset of $\Omega^*$
for every $\tau\in\Omega^*$.
Therefore, we have that
\begin{equation}\label{lPGosgtle1}
  \sum_{\upsilon\in G(\tau)}\Bm{P}{\osg{\upsilon}}=\Bm{P}{\osg{G(\tau)}}\le 1
\end{equation}
for each $\tau\in\Omega^*$.
Thus,
for each $\sigma\in\Omega^*$,
we see that
\begin{align*}
  \Bm{P}{\osg{F(\sigma)}}
  &=\sum_{\nu\in F(\sigma)}\Bm{P}{\osg{\nu}}
  =\sum_{\tau\in F(\rho)}\sum_{\upsilon\in G(\tau)}\Bm{P}{\osg{\tau\upsilon a}} \\
  &=\sum_{\tau\in F(\rho)}\sum_{\upsilon\in G(\tau)}\Bm{P}{\osg{\tau}}\Bm{P}{\osg{\upsilon}}P(a) \\
  &\le\sum_{\tau\in F(\rho)}\Bm{P}{\osg{\tau}}P(a)
  =\Bm{P}{\osg{F(\rho)}}P(a) \\
  &\le\Bm{P}{\osg{\rho}}P(a)
  =\Bm{P}{\osg{\sigma}},
\end{align*}
where the second equality follows from the fact that the mapping
\[
  \{(\tau,\upsilon)\mid\tau\in F(\rho)\;\&\; \upsilon\in G(\tau)\}\ni(\tau,\upsilon)\mapsto
\tau \upsilon a\in F(\sigma)
\]
is a bijection,
the first inequality follows from \eqref{lPGosgtle1},
and the second inequality follows from the assumption.
Therefore \eqref{PFleP-Bs} holds for all $\sigma\in\Omega^{n+1}$.
Hence, \eqref{PFleP-Bs} holds for all $\sigma\in\Omega^*$, as desired.

We then define $\mathcal{D}$ to be a subset of $\N^+\times \Omega^*$ such that
$\mathcal{D}_n=\bigcup_{\sigma\in \mathcal{C}_n}F(\sigma)$ for every $n\in\N^+$.
Since $\mathcal{C}_n$ is a prefix-free subset of $\Omega^*$ for every $n\in\N^+$,
we see that $\mathcal{D}_n$ is also a prefix-free subset of $\Omega^*$ for every $n\in\N^+$.
For each $n\in\N^+$, we see that
\begin{equation*}
  \Bm{P}{\osg{\mathcal{D}_n}}\le\sum_{\sigma\in \mathcal{C}_n}\Bm{P}{\osg{F(\sigma)}}
  \le\sum_{\sigma\in \mathcal{C}_n}\Bm{P}{\osg{\sigma}}
  =\Bm{P}{\osg{\mathcal{C}_n}}<2^{-n},
\end{equation*}
where the second inequality follows from \eqref{PFleP-Bs} and
the equality follows from the prefix-freeness of $\mathcal{C}_n$.
Moreover,
since $\mathcal{C}$ is r.e.,
we see
that $\mathcal{D}$ is also r.e.
Thus, $\mathcal{D}$ is a Martin-L\"of $P$-test.

On the other hand, we see that,
for every $n\in\N^+$, if $\beta\in\osg{\mathcal{C}_n}$ then $\alpha\in\osg{\mathcal{D}_n}$.
Thus, it follows from \eqref{betainosgmathCn_CPUSPCF} that
$\alpha\in\osg{\mathcal{D}_n}$ for every $n\in\N^+$.
Hence, $\alpha$ is not Martin-L\"of $P$-random.
This completes the proof.
\end{proof}

Theorems~\ref{cpucs} and~\ref{cpuscsf}
show that
certain closure properties hold for ensembles for an \emph{arbitrary} finite probability space.
In Sections~\ref{CPITW} and~\ref{IANERV} below,
we will see that various strong closure properties of another type hold for ensembles
for \emph{arbitrary} finite probability spaces.

\subsection{Weaker randomness notions than Martin-L\"of $P$-randomness}

Before concluding this section we
investigate
the reason why we adopt
the notion of Martin-L\"of $P$-randomness
as the randomness notion in Thesis~\ref{thesis},
and not other notions of randomness than Martin-L\"of $P$-randomness.
In the field of algorithmic randomness, the notions of Schnorr randomness~\cite{Sch71} and
Kurtz randomness~\cite{Kurtz81} are two of major randomness notions weaker than
the notion of Martin-L\"of randomness~\cite{M66} for infinite binary sequences
with respect to Lebesgue measure,
where the Kurtz randomness is
weaker than the Schnorr randomness
(see e.g.~Nies~\cite{N09} and Downey and Hirschfeldt~\cite{DH10} for the detail of the relation
among these randomness notions).
For increasing the universality of Thesis~\ref{thesis}, it would be desirable for the randomness notion
adopted in Thesis~\ref{thesis} to be as weak as possible.
Actually, we have proposed to use the notion of Martin-L\"of $P$-randomness as
an operational characterization of the notion of probability
in Thesis~\ref{thesis}.
Can we replace the Martin-L\"of $P$-randomness
in Thesis~\ref{thesis}
by a $P$-randomness notion
corresponding to the Schnorr randomness or the Kurtz randomness?
In what follows, we show that neither the Schnorr $P$-randomness nor the Kurtz $P$-randomness
is considered to be appropriate as the randomness notion in Thesis~\ref{thesis}.

First,
the notion of Schnorr randomness~\cite{Sch71} is
naturally
generalized over the notion of $P$-randomness as follows.

\begin{definition}[Schnorr $P$-randomness%
]\label{Schnorr_P-randomness}
Let $\Omega$ be an alphabet, and let $P\in\PS(\Omega)$.
A subset $\mathcal{C}$ of $\N^+\times \Omega^*$ is called a \emph{Schnorr $P$-test} if
$\mathcal{C}$ is an r.e.~set such that
\begin{enumerate}
\item for every $n\in\N^+$ it holds that
$\mathcal{C}_n$ is a prefix-free subset of $\Omega^*$ and
$\Bm{P}{\osg{\mathcal{C}_n}}<2^{-n}$,
where
$\mathcal{C}_n$ denotes the set
$\left\{\,
  \sigma\bigm|(n,\sigma)\in\mathcal{C}
\,\right\}$,
and
\item $\Bm{P}{\osg{\mathcal{C}_n}}$ is uniformly compuatble in $n$, i.e.,
there exists a computable function $f\colon\N^+\times\N\to\Q$ such that
$\abs{\Bm{P}{\osg{\mathcal{C}_n}}-f(n,k)} < 2^{-k}$
for all $n\in\N^+$ and $k\in\N$.
\end{enumerate}

For any $\alpha\in\Omega^\infty$, we say that $\alpha$ is \emph{Schnorr $P$-random} if
for every Schnorr $P$-test $\mathcal{C}$
there exists $n\in\N^+$ such that $\alpha\notin\osg{\mathcal{C}_n}$.
\qed
\end{definition}

In the case where $\Omega=\{0,1\}$ and $P$ satisfies that $P(0)=P(1)=1/2$,
the Schnorr $P$-randomness
results in the
original
Schnorr randomness~\cite{Sch71}.
Recall that
in Thesis~\ref{thesis} we do not require the finite probability space $P\in\PS(\Omega)$
to be computable at all.
Thus, the Bernoulli measure $\lambda_{P}$ itself is not necessarily computable therein.
In Definition~\ref{Schnorr_P-randomness}, however, a Schnorr $P$-test $\mathcal{C}$ must satisfy that
$\Bm{P}{\osg{\mathcal{C}_n}}$ is uniformly computable in $n$.
This is unnatural,
because
the value $\Bm{P}{\osg{\mathcal{C}_n}}$ of the Bernoulli measure $\lambda_{P}$ must be
(uniformly) computable
whereas
the Bernoulli measure $\lambda_{P}$ itself is uncomputable in general.
Thus, the use of Schnorr $P$-randomness in Thesis~\ref{thesis} is considered to be unnatural.
Hence, we did not propose to use the Schnorr $P$-randomness
as the randomness notion
to represent an operational characterization of the notion of probability
in Thesis~\ref{thesis}.

On the other hand,
the notion of Kurtz randomness~\cite{Kurtz81} is
naturally
generalized over the notion of $P$-randomness as follows.

\begin{definition}[Kurtz $P$-randomness%
]\label{Kurtz_P-randomness}
Let $\Omega$ be an alphabet, and let $P\in\PS(\Omega)$.
\begin{enumerate}
\item A subset $\mathcal{C}$ of $\Omega^*$ is called a \emph{Kurtz $P$-test} if
$\mathcal{C}$ is an r.e.~set and $\Bm{P}{\osg{\mathcal{C}}}=1$.
\item For any $\alpha\in\Omega^\infty$, we say that $\alpha$ is \emph{Kurtz $P$-random} if
for every Kurtz $P$-test $\mathcal{C}$ it holds that $\alpha\in\osg{\mathcal{C}}$.
\qed
\end{enumerate}
\end{definition}

For the notion of Kurtz $P$-randomness,
the set $\Omega^\infty\setminus\osg{\mathcal{C}}$
regarding
a Kurtz $P$-test $\mathcal{C}$
plays a role
as the effective null set explained in Section~\ref{AR}.
In the case where $\Omega=\{0,1\}$ and $P$ satisfies that $P(0)=P(1)=1/2$,
the Kurtz $P$-randomness results in the
original
Kurtz randomness~\cite{Kurtz81}.

As we stated in
Section~\ref{Subsec:The law of large numbers},
the law of large numbers \emph{must} hold in an operational characterization of the notion of probability, i.e.,
in an infinite sequence of outcomes which is
being generated by infinitely repeated trials described by a finite probability space
under study.
In fact, we have confirmed in Theorem~\ref{LLN} that
the law of large numbers certainly holds for every ensemble for an arbitrary finite probability space, i.e.,
for every Martin-L\"{o}f $P$-random infinite sequence for an arbitrary finite probability space $P$.
However, the law of large numbers does not necessarily hold for
a Kurtz $P$-random infinite sequence,
as Theorem~\ref{LLNnotforKPR}~(iii) below states.
Thus, the use of Kurtz $P$-randomness in Thesis~\ref{thesis} is not considered to be
natural.
Hence, we did not propose to use the Kurtz $P$-randomness as a randomness notion
to represent an operational characterization of the notion of probability
in Thesis~\ref{thesis}.

Note however that
due to Theorem~\ref{LLNnotforKPR}~(i) and~(ii) below
the Kurtz $P$-randomness is consistent with the facts that
an elementary event with probability zero never occurs
and an elementary event with probability one occurs certainly,
which we have demonstrated
in Section~\ref{Subsec:SOEEPO}
for
the notion of probability,
in particular,
in quantum mechanics.

\begin{theorem}\label{LLNnotforKPR}
Let $\Omega$ be an alphabet, and let $P\in\PS(\Omega)$.
\begin{enumerate}
\item For every $\alpha\in\Omega^\infty$ and $a\in\Omega$,
if $\alpha$ is Kurtz $P$-random and $P(a)=0$ then $\alpha$ does not contain $a$.
\item For every $\alpha\in\Omega^\infty$ and $a\in\Omega$,
if $\alpha$ is Kurtz $P$-random and $P(a)=1$ then $\alpha$ consists only of $a$.
\item There exists a Kurtz $P$-random infinite sequence $\alpha\in\Omega^\infty$ such that
for every $a\in\Omega$ if $0<P(a)<1$ then both
$\liminf_{n\to\infty}N_a(\rest{\alpha}{n})/n=0$
and
$\limsup_{n\to\infty}N_a(\rest{\alpha}{n})/n=1$ hold,
which implies that
$N_a(\rest{\alpha}{n})/n$ does not converge to $P(a)$ as $n\to\infty$.
Here $N_a(\sigma)$ denotes the number of the occurrences of $a$ in $\sigma$ for every $a\in\Omega$ and $\sigma\in\Omega^*$.
\qed
\end{enumerate}
\end{theorem}

The corresponding result to Theorem~\ref{LLNnotforKPR}~(iii) holds
for the notion of
the original
Kurtz randomness~\cite{Kurtz81}.
That is, the law of large numbers does not necessarily hold for
a Kurtz random infinite binary sequence
(see Nies~\cite[Fact~3.5.4 and Proposition~3.5.5]{N09}).
The proof of Theorem~\ref{LLNnotforKPR}~(iii)
is obtained by
generalizing
the arguments
presented by Nies~\cite[Sections~1.8 and 3.5]{N09}
to prove the failure of the law of large numbers
for the case where $\Omega=\{0,1\}$ and $P$ satisfies that $P(0)=P(1)=1/2$.
We remark that
in Theorem~\ref{LLNnotforKPR}~(i) --- (iii)
the underlying finite probability space $P$ is \emph{quite arbitrary}.
In particular,
we do not require $P$ to be computable at all in Theorem~\ref{LLNnotforKPR}.

In order to prove Theorem~\ref{LLNnotforKPR}~(iii),
we need the following lemma, which is a generalization of
the argument in
the proof of Fact~3.5.4 of Nies [22].

\begin{lemma}\label{Subsec-WRN:lemma-KR}
Let $\Omega$ be an alphabet, and let $P\in\PS(\Omega)$.
We denote by $\Omega_u$ the set $\{a\in\Omega\mid 0<P(a)<1\}$.
Suppose that $\Omega_u\neq\emptyset$.
Let $\mathcal{C}$ be a subset of $\Omega^*$ such that $\Bm{P}{\osg{\mathcal{C}}}=1$.
Then, for every $\sigma\in\Omega_u^*$ there exists $\tau\in\Omega_u^*$ such that
$\sigma$ is a prefix of $\tau$ and $\osg{\tau}\subset\osg{\mathcal{C}}$.
\end{lemma}

\begin{proof}
Since $\Omega_u\neq\emptyset$, it follows that $P(a)<1$ for every $a\in\Omega$.
Thus, we have that
\begin{equation}\label{Subsec-WRN:lemma-KR:eq-0}
  P(a)=0
\end{equation}
for every $a\in\Omega\setminus\Omega_u$.

Let $\mathcal{C}$ be a subset of $\Omega^*$ such that $\Bm{P}{\osg{\mathcal{C}}}=1$.
We denote by $\mathcal{P}$ the set of the shortest strings $\tau\in\Omega^*$ such that
$\osg{\tau}\subset\osg{\mathcal{C}}$.
It follows that $\mathcal{P}$ is a prefix-free subset of $\Omega^*$ and $\osg{\mathcal{P}}=\osg{\mathcal{C}}$.
We then define $\mathcal{T}$ as the subset $\{\tau\in\Omega_u^*\mid \tau\in\mathcal{P}\}$ of $\mathcal{P}$.
Note that $\Bm{P}{\osg{\mathcal{P}}}=\Bm{P}{\osg{\mathcal{C}}}=1$. Thus, we see that
\begin{equation}\label{Subsec-WRN:lemma-KR:eq-1}
  \Bm{P}{\osg{\mathcal{T}}}
  =\sum_{\tau\in\mathcal{T}}P(\tau)
  =\sum_{\tau\in\mathcal{P}}P(\tau)
  =\Bm{P}{\osg{\mathcal{P}}}=1,
\end{equation}
where
the first and third equalities follow from the prefix-freeness of $\mathcal{T}$ and $\mathcal{P}$,
respectively, and
the second equality follows from
\eqref{Subsec-WRN:lemma-KR:eq-0}.

Let $\sigma\in\Omega_u^*$.
Now, assume contrarily that $\osg{\mathcal{T}}\cap\osg{\sigma}=\emptyset$.
Then we see that
\[
  \Bm{P}{\osg{\mathcal{T}}}
  =\Bm{P}{\osg{\mathcal{T}}\cup\osg{\sigma}}-\Bm{P}{\osg{\sigma}}
  \le 1-P(\sigma)<1,
\]
where the last inequality follows from the property that $P(a)>0$ for every $a\in\Omega_u$.
However,
this contradicts
the equation~\eqref{Subsec-WRN:lemma-KR:eq-1}.
Thus, we
have that $\osg{\mathcal{T}}\cap\osg{\sigma}\neq\emptyset$,
and therefore there exists $\alpha\in\Omega^\infty$
extending $\sigma$ such that
$\rest{\alpha}{n}\in\mathcal{T}$ for some $n\in\N$.
Note that $\rest{\alpha}{n}\in\Omega_u^*$ and
$\osg{\rest{\alpha}{n}}\subset\osg{\mathcal{T}}\subset\osg{\mathcal{P}}=\osg{\mathcal{C}}$.
If $\abs{\sigma}\le n$
then
we choose $\rest{\alpha}{n}$ as $\tau$
since $\sigma$ is a prefix of $\rest{\alpha}{n}$ and $\osg{\rest{\alpha}{n}}\subset\osg{\mathcal{C}}$.
Otherwise,
we choose $\sigma$ as $\tau$
since $\osg{\sigma}\subset\osg{\rest{\alpha}{n}}\subset\osg{\mathcal{C}}$.
This completes the proof.
\end{proof}

Theorem~\ref{LLNnotforKPR} is then proved as follows.

\begin{proof}[Proof of Theorem~\ref{LLNnotforKPR}]
Let $\Omega$ be an alphabet, and let $P\in\PS(\Omega)$.

(i) We denote by $\Omega_e$ the set $\{a\in\Omega\mid P(a)>0\}$.
Let $\alpha\in\Omega^\infty$, and let $a\in\Omega$.
Then,
assuming that $P(a)=0$ and $\alpha$ contains $a$, we will show that $\alpha$ is not Kurtz $P$-random.
Now, based on the assumption, we have that $\alpha(n)=a$ for some $n\in\N^+$.
Note then that $\Bm{P}{\osg{\Omega^{n}}}=1$
obviously.
Thus, we see that
\[
  \Bm{P}{\osg{\Omega_e^{n}}}
  =\sum_{\sigma\in\Omega_e^{n}}P(\sigma)
  =\sum_{\sigma\in\Omega^{n}}P(\sigma)
  =\Bm{P}{\osg{\Omega^{n}}}=1,
\]
where
the first and third equalities follow from the prefix-freeness of $\Omega_e^{n}$ and $\Omega^{n}$,
respectively, and
the second equality follows from the assumption that $P(a)=0$.
It follows that
$\Omega_e^{n}$ is a Kurtz $P$-test.
However, $\alpha\notin\osg{\Omega_e^{n}}$ obviously, and therefore $\alpha$ is not Kurtz $P$-random.
This completes the proof of Theorem~\ref{LLNnotforKPR}~(i).

(ii) Theorem~\ref{LLNnotforKPR}~(ii) follows immediately from Theorem~\ref{LLNnotforKPR}~(i).

(iii) We denote by $\Omega_u$ the set $\{a\in\Omega\mid 0<P(a)<1\}$.
In the case of $\Omega_u=\emptyset$,
the result holds formally.
Thus, in what follows, we assume that $\Omega_u\neq\emptyset$.

We denote $\#\Omega_u$ by $N$
and then
denote the elements of $\Omega_u$ as $a_0,a_1,\dots,a_{N-1}$, so that $\Omega_u=\{a_0,a_1,\dots,a_{N-1}\}$.
Moreover,
for each $n\in\N^+$, we denote $a_{n\bmod N}$ by $b_n$,
where $n\bmod N$ denotes the remainder of the division of $n$ by $N$ as usual.
Since there are only countably infinitely many Turing machines,
there are only countably infinitely many Kurtz $P$-tests $\mathcal{C}_1,\mathcal{C}_2,\mathcal{C}_3,\dotsc$.

Now, we choose infinite sequences $\{\sigma_n\}_{n\in\N}\subset\Omega_u^*$ and
$\{\tau_n\}_{n\in\N^+}\subset\Omega_u^*$ 
in the following manner:
First, we choose $\sigma_0$
to be
any particular element of
$\Omega_u^+$.
Assume that $\sigma_{n-1}$ has already been chosen. Then we choose $\sigma_n$
to be
the concatenation of $\tau_n$ and
${b_n}^{(n-1)\abs{\tau_n}}$
where $\tau_n$ is any particular element of $\Omega_u^*$ with the properties that
$\sigma_{n-1}$ is a prefix of $\tau_n$ and $\osg{\tau_n}\subset\osg{\mathcal{C}_n}$.
Note that
such a $\tau_n$ exists due to Lemma~\ref{Subsec-WRN:lemma-KR}.

Then, since
$\sigma_{n}$ is a proper prefix of $\sigma_{n+1}$ for every $n\in\N^+$,
there exists
$\alpha\in\Omega_u^\infty$ such that
$\sigma_n$ is a prefix of $\alpha$ for every $n\in\N^+$.
On the one hand,
for each $n\in\N^+$, we see that $\alpha\in\osg{\sigma_n}\subset\osg{\tau_n}\subset\osg{\mathcal{C}_n}$.
Thus, $\alpha$ is Kurtz $P$-random.
On the other hand, for each $n\in\N^+$, it follows from
the manner of choice of $\sigma_n$
that
$N_{x}(\sigma_n)/\abs{\sigma_n} \ge 1-1/n$
if $x=b_n$ and
$N_{x}(\sigma_n)/\abs{\sigma_n}\le 1/n$
otherwise.
Thus,
for every $n\in\N^+$ and $k\in\{0,1,\dots,N-1\}$, if n mod N=k then we have that
\[
  \frac{N_{x}\left(\rest{\alpha}{\abs{\sigma_n}}\right)}{\abs{\sigma_n}}\ge 1-\frac{1}{n}
\]
if $x = a_k$ and
\[
  \frac{N_{x}\left(\rest{\alpha}{\abs{\sigma_n}}\right)}{\abs{\sigma_n}}\le \frac{1}{n}
\]
otherwise.
Hence,
since
$\lim_{n\to\infty}\abs{\sigma_n}=\infty$,
we have that $\liminf_{n\to\infty}N_a(\rest{\alpha}{n})/n=0$ and $\limsup_{n\to\infty}N_a(\rest{\alpha}{n})/n=1$
for every $a\in\Omega_u$.
This completes the proof of Theorem~\ref{LLNnotforKPR}~(iii).
\end{proof}

\section{Conditional probability and the independence between two events}
\label{CPITW}

In this section we operationally characterize the notion of \emph{conditional probability} and
the notion of the \emph{independence between two events} on a finite probability space, in terms of ensembles.

Let $\Omega$ be an alphabet, and let $P\in\PS(\Omega)$. Let $A\subset\Omega$ be
an event on the finite probability space $P$.
On the one hand, based on the pair $(P,A)$,
we define a finite probability space $\charaps{P}{A}\in\PS(\{0,1\})$ by the condition that
$(\charaps{P}{A})(1):=P(A)$ and $(\charaps{P}{A})(0):=1-P(A)$
(see the equation~\eqref{eq:PA=sumainAPa} for the definition of $P(A)$).
On the other hand,
for each ensemble $\alpha$ for $P$,
we use $\chara{A}{\alpha}$ to denote the infinite binary sequence such that,
for every $n\in\N^+$,
its $n$th element $(\chara{A}{\alpha})(n)$ is $1$ if $\alpha(n)\in A$ and $0$ otherwise.
Note that the notions of $\chara{A}{\alpha}$ and $\charaps{P}{A}$ in our theory together correspond to
the notion of \emph{mixing} in the theory of collectives by von Mises \cite{vM64}.
We can then show the following theorem.

\begin{theorem}\label{charaA}
Let $\Omega$ be an alphabet, and let $P\in\PS(\Omega)$. Let $A\subset\Omega$.
Suppose that $\alpha$ is an ensemble for the finite probability space $P$.
Then $\chara{A}{\alpha}$ is an ensemble for the finite probability space $\charaps{P}{A}$.
\end{theorem}

\begin{proof}
We show the contraposition.
Suppose that $\chara{A}{\alpha}$ is not Martin-L\"of $\charaps{P}{A}$-random.
Then there exists a Martin-L\"of $\charaps{P}{A}$-test $\mathcal{S}\subset\N^+\times\X$
such that
\begin{equation}\label{betainosgmSn_CPA}
  \chara{A}{\alpha}\in\osg{\mathcal{S}_n}
\end{equation}
for every $n\in\N^+$.
For each $\sigma\in\{0,1\}^+$,
let $F(\sigma)$ be the set of all $\tau\in\Omega^*$ such that
$\tau$ is obtained
by replacing each occurrence of $1$ in $\sigma$, if exists, by some
element of $A$ and
by replacing each occurrence of $0$ in $\sigma$, if exists, by some
element of $\Omega\setminus A$.
Namely,
for each $\sigma\in\{0,1\}^+$,
we define $F(\sigma)$ as the set of all
$\tau\in\Omega^{\abs{\sigma}}$
such that for every $n\in\N^+$ with $n\le\abs{\sigma}$
it holds that $\tau(n)\in A$ if $\sigma(n)=1$ and $\tau(n)\in \Omega\setminus A$ otherwise.
For example, if $\Omega=\{x,y,z\}$ and $A=\{x,y\}$ then $F(011)=\{zxx,zxy,zyx,zyy\}$.
Then, since $(\charaps{P}{A})(1)=\sum_{a\in A}P(a)$ and
$(\charaps{P}{A})(0)=\sum_{a\in\Omega\setminus A}P(a)$,
we have that
\begin{equation}\label{BmQ=Q=P=BmP_CPA}
  \Bm{\charapse{P}{A}}{\osg{\sigma}}=(\charaps{P}{A})(\sigma)=P(F(\sigma))=\Bm{P}{\osg{F(\sigma)}}
\end{equation}
for each $\sigma\in\{0,1\}^+$.
We then define $\mathcal{T}$ to be a subset of $\N^+\times \Omega^*$ such that
$\mathcal{T}_n=\bigcup_{\sigma\in\mathcal{S}_n} F(\sigma)$ for every $n\in\N^+$.
Note here that, for each $n\in\N^+$,
$\lambda\notin\mathcal{S}_n$ since $\Bm{\charapse{P}{A}}{\osg{\mathcal{S}_n}}<2^{-n}<1$.
Then, since
$\mathcal{S}_n$ is a prefix-free subset of $\X$ for every $n\in\N^+$,
we see that $\mathcal{T}_n$ is a prefix-free subset of $\Omega^*$ for every $n\in\N^+$.
For each $n\in\N^+$, we also see that
\[
  \Bm{P}{\osg{\mathcal{T}_n}}
  \le\sum_{\sigma\in\mathcal{S}_n}\Bm{P}{\osg{F(\sigma)}}
  =\sum_{\sigma\in\mathcal{S}_n}\Bm{\charapse{P}{A}}{\osg{\sigma}}
  =\Bm{\charapse{P}{A}}{\osg{\mathcal{S}_n}}<2^{-n},
\]
where the first equality follows from \eqref{BmQ=Q=P=BmP_CPA} and
the second equality follows from the prefix-freeness of $\mathcal{S}_n$.
Moreover, since $\mathcal{S}$ is r.e., $\mathcal{T}$ is also r.e.
Thus, $\mathcal{T}$ is a Martin-L\"of $P$-test.

On the other hand,
note that, for every $n\in\N^+$,
if $\chara{A}{\alpha}\in\osg{\mathcal{S}_n}$ then $\alpha\in\osg{\mathcal{T}_n}$.
Thus, it follows from \eqref{betainosgmSn_CPA}
that $\alpha\in\osg{\mathcal{T}_n}$ for every $n\in\N^+$.
Hence, $\alpha$ is not Martin-L\"of $P$-random.
This completes the proof.
\end{proof}

We show that the notion of conditional probability in a finite probability space can be
represented by an
ensemble
in a \emph{natural manner}.
For that purpose, first we recall the notion of conditional probability in a finite probability space.

Let $\Omega$ be an alphabet, and let $P\in\PS(\Omega)$. Let $B\subset\Omega$ be
an event on the finite probability space $P$.
Suppose that $P(B)>0$.
Then, for each event $A\subset\Omega$,
the \emph{conditional probability of A given B}, denoted
$P(A|B)$, is defined as $P(A\cap B)/P(B)$.
This notion defines a finite probability space $P_B\in\PS(B)$
by the condition that $P_B(a):=P(\{a\}|B)$ for every $a\in B$.

When an infinite sequence $\alpha\in\Omega^\infty$ contains infinitely many elements from $B$,
$\cond{B}{\alpha}$ is defined as an infinite sequence in $B^\infty$ obtained from $\alpha$
by eliminating all elements
of
$\Omega\setminus B$ occurring in $\alpha$.
If $\alpha$ is an ensemble for the finite probability space $P$ and $P(B)>0$,
then $\alpha$ contains infinitely many elements from $B$ due to Theorem~\ref{LLN}.
Therefore, $\cond{B}{\alpha}$ is
properly
defined in this case.
Note that the notion of $\cond{B}{\alpha}$ in our theory corresponds to
the notion of \emph{partition} in the theory of collectives by von Mises \cite{vM64}.

We can then show Theorem~\ref{conditional_probability} below, which states that ensembles are
\emph{closed under conditioning}.

\begin{theorem}[Closure property under conditioning]\label{conditional_probability}
Let $\Omega$ be an alphabet, and let $P\in\PS(\Omega)$. Let $B\subset\Omega$ be
an event on the finite probability space $P$ with $P(B)>0$.
For every ensemble $\alpha$ for
$P$,
it holds that $\cond{B}{\alpha}$ is an ensemble for the finite probability space $P_B$.
\end{theorem}

\begin{proof}
In the case of $B=\Omega$, we have $P_B=P$ and $\cond{B}{\alpha}=\alpha$.
Therefore the result is obvious.
Thus, in what follows, we assume that
$B$ is a proper subset of $\Omega$.

First, we choose any
particular
$a\in\Omega\setminus B$ and
define a finite probability space $Q\in\PS(B\cup\{a\})$ by the condition that
$Q(x):=P(\Omega\setminus B)$ if $x=a$ and $Q(x):=P(x)$ otherwise.
Note here that
\begin{equation}\label{1-Q=PB}
 1-Q(a)=P(B),
\end{equation}
and therefore
\begin{equation}\label{Q<1}
  Q(a)<1.
\end{equation}
Let $\beta$ be the infinite sequence over $B\cup\{a\}$
obtained by replacing all occurrences of elements of $\Omega\setminus B$ in $\alpha$ by $a$.
Then, by using Theorem~\ref{contraction} repeatedly,
it is easy to show that $\beta$ is Martin-L\"of $Q$-random.
Hence,
in order to complete the proof,
it is sufficient to show that if $\cond{B}{\alpha}$ is not Martin-L\"of $P_B$-random
then $\beta$ is not Martin-L\"of $Q$-random.

Thus, let us assume that $\cond{B}{\alpha}$ is not Martin-L\"of $P_B$-random.
Then there exists a Martin-L\"of $P_B$-test $\mathcal{C}\subset B\times\N^+$ such that
\begin{equation}\label{fBaC}
  \cond{B}{\alpha}\in\osg{\mathcal{C}_n}
\end{equation}
for every $n\in\N^+$.
For each $\sigma\in B^+$, let $F(\sigma)$ be the set of all finite strings over $B\cup\{a\}$ of the form
$a^{k_1}\sigma_1a^{k_2}\sigma_2\dots \sigma_{L-1}a^{k_L}\sigma_{L}$ for some $k_1,k_2,\dots,k_L\in\N$,
where
$L:=\abs{\sigma}$ and $\sigma_1\sigma_2\dots\sigma_L:=\sigma$ with $\sigma_i\in B$.
Note that $F(\sigma)$ is a prefix-free subset of $(B\cup\{a\})^*$ for every $\sigma\in B^+$.
Thus,
for each $\sigma\in B^+$, we see that
\begin{equation}\label{BmQosgFs=BmPBosgs_CPEuC}
\begin{split}
  \Bm{Q}{\osg{F(\sigma)}}
  &=\sum_{k_1,k_2,\dots,k_L=0}^\infty
    \Bm{Q}{\osg{a^{k_1}\sigma_1a^{k_2}\sigma_2\dots \sigma_{L-1}a^{k_L}\sigma_{L}}} \\
  &=\sum_{k_1,k_2,\dots,k_L=0}^\infty\Bm{Q}{\osg{\sigma}}Q(a)^{k_1}Q(a)^{k_2}\dots Q(a)^{k_L} \\
  &=\Bm{Q}{\osg{\sigma}}\left(\sum_{k=0}^\infty Q(a)^k\right)^L \\
  &=\Bm{Q}{\osg{\sigma}}\frac{1}{(1-Q(a))^L} \\
  &=\Bm{Q}{\osg{\sigma}}\frac{1}{P(B)^L} \\
  &=\Bm{P_B}{\osg{\sigma}},
\end{split}
\end{equation}
where we use \eqref{Q<1} and \eqref{1-Q=PB} in the forth and fifth equalities, respectively.
We then define $\mathcal{D}$ to be a subset of $\N^+\times(B\cup\{a\})^*$ such that
$\mathcal{D}_n=\bigcup_{\sigma\in \mathcal{C}_n}F(\sigma)$ for every $n\in\N^+$.
Note here that, for each $n\in\N^+$,
$\lambda\notin\mathcal{C}_n$ since $\Bm{P_B}{\osg{\mathcal{C}_n}}<2^{-n}<1$.
Then, since $\mathcal{C}_n$ is a prefix-free subset of $B^*$ for every $n\in\N^+$,
we see that $\mathcal{D}_n$ is a prefix-free subset of $(B\cup\{a\})^*$ for every $n\in\N^+$.
For each $n\in\N^+$, we
also
see that
\[
  \Bm{Q}{\osg{\mathcal{D}_n}}
  \le\sum_{\sigma\in\mathcal{C}_n}\Bm{Q}{\osg{F(\sigma)}}
  =\sum_{\sigma\in\mathcal{C}_n}\Bm{P_B}{\osg{\sigma}}
  =\Bm{P_B}{\osg{\mathcal{C}_n}}<2^{-n},
\]
where the first equality follows from \eqref{BmQosgFs=BmPBosgs_CPEuC} and
the second equality follows from the prefix-freeness of $\mathcal{C}_n$.
Moreover, since $\mathcal{C}$ is r.e.,
$\mathcal{D}$ is also r.e.
Thus, $\mathcal{D}$ is a Martin-L\"of $Q$-test.

On the other hand,
since $\cond{B}{\alpha}$ is the infinite sequence over $B$ obtained from $\beta$
by eliminating all occurrences of the symbol $a$ in $\beta$,
we see that,
for every $n\in\N^+$, if $\cond{B}{\alpha}\in\osg{\mathcal{C}_n}$ then $\beta\in\osg{\mathcal{D}_n}$.
Thus, it follows from \eqref{fBaC} that $\beta\in\osg{\mathcal{D}_n}$ for every $n\in\N^+$.
Hence, $\beta$ is not Martin-L\"of $Q$-random.
This completes the proof.
\end{proof}

As an application of Theorem~\ref{conditional_probability},
we consider
the Von Neumann extractor~\cite{vN51}
as follows.

\begin{example}[Von Neumann extractor,
von Neumann~\cite{vN51}]\label{Von-Neumann-extractor}
According to \cite{W14},
\emph{in the terminology of the conventional probability theory}, consider a Bernoulli sequence.
The \emph{Von Neumann extractor} takes successive pairs of consecutive bits from the Bernoulli sequence.
If the two bits matches, no output is generated.
If the bits differ, the value of the first bit is output.
The Von Neumann extractor can be shown to produce a uniform binary output.
This description of the Von Neumann extractor is intuitive.
But
it is
vague and incomplete from the measure-theoretic point of view
on which the conventional probability theory relies.

In our framework, in contrast, the Von Neumann extractor operates
\emph{in a rigorous way but appealing to intuition}, as follows:
Let $P\in\PS(\{0,1\})$ and let $\alpha$ be an ensemble for $P$.
Then $\alpha$ can be regarded as an ensemble for
a finite probability space $Q\in\PS(\{00,01,10,11\})$ where
$Q(ab)=P(a)P(b)$ for every $a,b\in\{0,1\}$. Consider the event $B=\{01,10\}$ on $Q$.
It follows from Theorem~\ref{conditional_probability} that
$\cond{B}{\alpha}$ is an ensemble for $Q_B\in\PS(\{01,10\})$ with $Q_B(01)=Q_B(10)=1/2$.
Namely, $\alpha$ is a Martin-L\"of random infinite sequence over the alphabet $\{01,10\}$
instead of $\{0,1\}$.
Hence, a random individual infinite sequence is certainly extracted by the Von Neumann extractor in our framework.
\qed
\end{example}

Let $P\in\PS(\{0,1\})$, and
let us consider an infinite sequence $\alpha\in\XI$ of outcomes
which is being generated by infinitely repeated trials described
by the finite probability space $P$ on $\{0,1\}$.
According to Thesis~\ref{thesis}, we have that $\alpha$ is an ensemble for $P$.
Thus, via the
infinitely repeated
trials
we obtain each element of the Martin-L\"of $P$-random infinite sequence $\alpha$
one by one in a sequential order from the top of
$\alpha$.
Now, applying the procedure of the Von Neumann extractor described in
Example~\ref{Von-Neumann-extractor}
to the elements of $\alpha$ one by one in a sequential order from the top of $\alpha$,
we can \emph{effectively} convert $\alpha$ into a Martin-L\"of random infinite
binary
sequence.
What about the converse?
That is,
is it possible to generate
elements of a Martin-L\"of $P$-random infinite sequence one by one in a sequential order
from the top of it,
given arbitrary elements of a Martin-L\"of random infinite binary sequence?
In Section~\ref{AFBC},
we describe a
\emph{simple}
procedure for performing this task
in the case where the finite probability space $P$ on $\{0,1\}$ is computable.

Let $\Omega$ be an alphabet, and let $P\in\PS(\Omega)$.
For any events $A,B\subset\Omega$ on the finite probability space $P$,
we say that $A$ and $B$ are \emph{independent on} $P$ if $P(A\cap B)=P(A)P(B)$.
In the case of $P(B)>0$, it holds that $A$ and $B$ are independent on $P$ if and only if $P(A|B)=P(A)$.

Theorem~\ref{cond-ind} below gives operational characterizations of the notion of the independence between two events
in terms of ensembles.
For any $\alpha,\beta\in\Omega^\infty$,
we say that $\alpha$ and $\beta$ are \emph{equivalent}
if there exists $P\in\PS(\Omega)$ such that $\alpha$ and $\beta$ are both an ensemble for $P$.

\begin{theorem}\label{cond-ind}
Let $\Omega$ be an alphabet, and let $P\in\PS(\Omega)$. Let $A,B\subset\Omega$ be
events on the finite probability space $P$.
Suppose that $P(B)>0$.
Then the following conditions are equivalent to one another.
\begin{enumerate}
  \item The events $A$ and $B$ are independent on $P$.
  \item For every ensemble $\alpha$ for the finite probability space $P$,
    it holds that $\chara{A}{\alpha}$ is equivalent to $\chara{A\cap B}{\cond{B}{\alpha}}$.
  \item There exists an ensemble $\alpha$ for the finite probability space $P$ such that
    $\chara{A}{\alpha}$ is equivalent to $\chara{A\cap B}{\cond{B}{\alpha}}$.
\end{enumerate}
\end{theorem}

\begin{proof}
Let $\alpha$ be an arbitrary ensemble for the finite probability space $P$.
Then, on the one hand,
it follows from Theorem~\ref{charaA} that $\chara{A}{\alpha}$ is Martin-L\"of $\charaps{P}{A}$-random.
On the other hand, it follows from $P(B)>0$ and Theorem~\ref{conditional_probability} that
$\cond{B}{\alpha}$ is an ensemble for the finite probability space $P_B$.
Therefore, by Theorem~\ref{charaA},
we see that $\chara{A\cap B}{\cond{B}{\alpha}}$ is Martin-L\"of $\charaps{P_B}{A\cap B}$-random.

Now, let us assume
that the condition~(i) holds. Then $P_B(A\cap B)=P(A)$.
It follows that $\charaps{P_B}{A\cap B}=\charaps{P}{A}$.
Therefore, for an arbitrary ensemble $\alpha$ for the finite probability space $P$,
we see that $\chara{A}{\alpha}$ and $\chara{A\cap B}{\cond{B}{\alpha}}$ are equivalent.
Thus, we have the implication (i) $\Rightarrow$ (ii).

Since there exists an ensemble $\alpha$ for the finite probability space $P$ by Theorem~\ref{Bmae},
the implication (ii) $\Rightarrow$ (iii) is obvious.

Finally, the implication (iii) $\Rightarrow$ (i) is shown as follows.
Assume that the condition~(iii) holds.
Then there exist
an ensemble $\alpha$ for the finite probability space $P$ and a finite probability space $Q\in\PS(\{0,1\})$
such that $\chara{A}{\alpha}$ and $\chara{A\cap B}{\cond{B}{\alpha}}$ are both Martin-L\"of $Q$-random.
It follows from the consideration at the beginning of this proof
that $\chara{A}{\alpha}$ is Martin-L\"of $\charaps{P}{A}$-random,
and $\chara{A\cap B}{\cond{B}{\alpha}}$ is Martin-L\"of $\charaps{P_B}{A\cap B}$-random.
Using Corollary~\ref{uniquness} we see that $\charaps{P}{A}=Q=\charaps{P_B}{A\cap B}$,
and therefore $P(A)=P_B(A\cap B)$.
This completes the proof.
\end{proof}

\section{The independence of an arbitrary number of events/random variables}
\label{IANERV}

In this section we operationally characterize the notion of the \emph{independence
of an arbitrary number of events/random variables} on a finite probability space, in terms of ensembles.
We will do this by introducing the notion of the \emph{independence of ensembles}
and then basing on it
in this section.

First, we operationally characterize the notion of
the \emph{independence of an arbitrary number of random variables},
in terms of the notion of the independence of ensembles.
For that purpose, we introduce some terminology regarding the conventional probability theory.
Let $\Omega$ be an alphabet, and let $P\in\PS(\Omega)$.
A \emph{random variable} on
$\Omega$ is a function
$X\colon\Omega\to\Omega'$ where $\Omega'$ is an alphabet.
Let $X_1\colon\Omega\to\Omega_1,\dots,X_n\colon\Omega\to\Omega_n$ be random variables on $\Omega$.
For any predicate
$F(v_1,\dotsc,v_n)$ with variables $v_1,\dots,v_n$,
we use $F(X_1,\dots,X_n)$ to denote the event
$$\{a\in\Omega\mid F(X_1(a),\dots,X_n(a))\}$$
on $P$.
We say that the random variables $X_1,\dots,X_n$ are
\emph{independent on} $P$
if for every $x_1\in\Omega_1,\dots,x_n\in\Omega_n$ it holds that
$$P(X_1=x_1\;\&\;\dotsc\;\&\;X_n=x_n)=P(X_1=x_n)\dotsm P(X_n=x_n).$$
We use $X_1\times\dots\times X_n$ to denote a random variable
$Y\colon\Omega\to\Omega_1\times\dots\times\Omega_n$ on $\Omega$ such that
$$Y(a)=(X_1(a),\dots,X_n(a))$$ for every $a\in\Omega$.

For any random variable $X\colon\Omega\to\Omega'$ on $\Omega$,
we use $X(P)$ to denote a finite probability space $P'\in\PS(\Omega')$ such that
$P'(x)=P(X=x)$ for every $x\in\Omega'$.

Let $\Omega_1,\dots,\Omega_n$ be alphabets.
For any $P_1\in\PS(\Omega_1),\dots,P_n\in\PS(\Omega_n)$,
we use
$$P_1\times\dots\times P_n$$
to denote a finite probability space
$Q\in\PS(\Omega_1\times\dots\times\Omega_n)$ such that
$$Q(a_1,\dots,a_n)=P_1(a_1)\dotsm P_n(a_n)$$
for every $a_1\in\Omega_1,\dotsc,a_n\in\Omega_n$.
Then the notion of the independence of random variables can be rephrased as follows.

\begin{proposition}\label{X(P)-independent}
Let $\Omega$ be an alphabet, and let $P\in\PS(\Omega)$. Let
$X_1\colon\Omega\to\Omega_1,\dots,X_n\colon\Omega\to\Omega_n$ be random variables on $\Omega$.
Then the random variables $X_1,\dots,X_n$ are independent on $P$ if and only if
$$(X_1\times\dots\times X_n)(P)=X_1(P)\times\dots\times X_n(P).$$
\end{proposition}

\begin{proof}
First, note that both $(X_1\times\dots\times X_n)(P)$ and $X_1(P)\times\dots\times X_n(P)$ are
finite probability spaces on $\Omega_1\times\dots\times\Omega_n$.

Let $x_1\in\Omega_1,\dots,x_n\in\Omega_n$.
On the one hand, we have
\begin{equation*}
  ((X_1\times\dots\times X_n)(P))(x_1,\dots,x_n)
  =P((X_1\times\dots\times X_n)=(x_1,\dots,x_n))
  =P(X_1=x_1\;\&\;\dotsc\;\&\;X_n=x_n).
\end{equation*}
On the other hand, we have
\begin{equation*}
  (X_1(P)\times\dots\times X_n(P))(x_1,\dots,x_n)
  = (X_1(P))(x_1)\dotsm(X_n(P))(x_n)
  =P(X_1=x_n)\dotsm P(X_n=x_n).
\end{equation*}
Thus, the result follows.
\end{proof}

Let $\Omega$ be an alphabet, and let $X\colon\Omega\to\Omega'$ be a random variable on $\Omega$.
For any $\alpha\in\Omega^\infty$,
we use $X(\alpha)$ to denote an infinite sequence $\beta$ over $\Omega'$ such that
$\beta(k)=X(\alpha(k))$ for every $k\in\N^+$.
We can then show the following theorem, which states that
ensembles are \emph{closed under the mapping by a random variable}.

\begin{theorem}[Closure property under the mapping by a random variable]\label{X(p)-X(P)}
Let $\Omega$ be an alphabet, and let $P\in\PS(\Omega)$. Let
$X\colon\Omega\to\Omega'$ be a random variable on $\Omega$.
If $\alpha$ is an ensemble for $P$ then $X(\alpha)$ is an ensemble for $X(P)$.
\end{theorem}

\begin{proof}
We show the contraposition. Suppose that $X(\alpha)$ is not Martin-L\"of $X(P)$-random.
Then there exists a Martin-L\"of $X(P)$-test $\mathcal{S}\subset\N^+\times(\Omega')^*$ such that
\begin{equation}\label{XainosgmathSn_CPuMbRV}
  X(\alpha)\in\osg{\mathcal{S}_n}
\end{equation}
for every $n\in\N^+$.
For each
$\sigma\in(\Omega')^+$,
let $F(\sigma)$ be the set of all
$\tau\in\Omega^*$
such that
(i) $\abs{\tau}=\abs{\sigma}$ and (ii)~$X(\tau(k))=\sigma(k)$ for every $k=1,2,\dots,\abs{\sigma}$.
Then, since
$$(X(P))(x)=\sum_{a\in X^{-1}(\{x\})}P(a)$$
for every $x\in\Omega'$, we have that
\begin{equation}\label{BmXP=os=XPs=Pfs=BmPofs_CPuMbRV}
  \Bm{X(P)}{\osg{\sigma}}=(X(P))(\sigma)=P(F(\sigma))=\Bm{P}{\osg{F(\sigma)}}
\end{equation}
for each $\sigma\in(\Omega')^+$.
We then define $\mathcal{T}$ to be a subset of $\N^+\times \Omega^*$ such that
$\mathcal{T}_n=\bigcup_{\sigma\in\mathcal{S}_n} F(\sigma)$ for every $n\in\N^+$.
Note here that, for each $n\in\N^+$,
$\lambda\notin\mathcal{S}_n$ since $\Bm{X(P)}{\osg{\mathcal{S}_n}}<2^{-n}<1$.
Then, since $\mathcal{S}_n$ is a prefix-free subset of $(\Omega')^*$ for every $n\in\N^+$,
we see that $\mathcal{T}_n$ is a prefix-free subset of $\Omega^*$ for every $n\in\N^+$.
For each $n\in\N^+$, we also see that
\[
  \Bm{P}{\osg{\mathcal{T}_n}}
  \le\sum_{\sigma\in\mathcal{S}_n}\Bm{P}{\osg{F(\sigma)}}
  =\sum_{\sigma\in\mathcal{S}_n}\Bm{X(P)}{\osg{\sigma}}
  =\Bm{X(P)}{\osg{\mathcal{S}_n}}<2^{-n},
\]
where the first equality follows from \eqref{BmXP=os=XPs=Pfs=BmPofs_CPuMbRV} and
the second equality follows from the prefix-freeness of $\mathcal{S}_n$.
Moreover, since $\mathcal{S}$ is r.e.,
$\mathcal{T}$ is also r.e.
Thus, $\mathcal{T}$ is a Martin-L\"of $P$-test.

On the other hand,
note that, for every $n\in\N^+$, if $X(\alpha)\in\osg{\mathcal{S}_n}$ then $\alpha\in\osg{\mathcal{T}_n}$.
Thus, it follows from \eqref{XainosgmathSn_CPuMbRV}
that $\alpha\in\osg{\mathcal{T}_n}$ for every $n\in\N^+$.
Hence, $\alpha$ is not Martin-L\"of $P$-random.
This completes the proof.
\end{proof}

Now, let us introduce the notion of the \emph{independence} of ensembles, as follows.
Let $\Omega_1,\dots,\Omega_n$ be alphabets.
For any $\alpha_1\in\Omega_1^\infty,\dots,\alpha_n\in\Omega_n^\infty$,
we use
$$\alpha_1\times\dots\times\alpha_n$$
to denote an infinite sequence
$\alpha$ over $\Omega_1\times\dots\times\Omega_n$
such that $\alpha(k)=(\alpha_1(k),\dots,\alpha_n(k))$ for every $k\in\N^+$.
Thus, $\alpha_1\times\dots\times\alpha_n\in(\Omega_1\times\dots\times\Omega_n)^\infty$
for every $\alpha_1\in\Omega_1^\infty,\dots,\alpha_n\in\Omega_n^\infty$.
For any $\sigma_1\in\Omega_1^*,\dots,\sigma_n\in\Omega_n^*$
with $\abs{\sigma_1}=\dots =\abs{\sigma_n}$,
we define
$$\sigma_1\times\dots\times\sigma_n$$
in a similar manner,
where we define $\lambda\times\dots\times\lambda$ as $\lambda$, in particular.
Thus, $\sigma_1\times\dots\times\sigma_n\in(\Omega_1\times\dots\times\Omega_n)^*$
for every $\sigma_1\in\Omega_1^*,\dots,\sigma_n\in\Omega_n^*$
with $\abs{\sigma_1}=\dots =\abs{\sigma_n}$.

\begin{definition}[Independence of ensembles]\label{Def:Ind-Ensmbls}
Let $\Omega_1,\dotsc,\Omega_n$ be alphabets, and let
$P_1\in\PS(\Omega_1),\dots,$ $P_n\in\PS(\Omega_n)$.
Let $\alpha_1,\dots,\alpha_n$ be ensembles for $P_1,\dots,P_n$, respectively.
We say that $\alpha_1,\dots,\alpha_n$ are
\emph{independent}
if $\alpha_1\times\dots\times\alpha_n$ is
an ensemble for
$P_1\times\dots\times P_n$.
\qed
\end{definition}

Note that the notion of the independence of ensembles in our theory corresponds to
the notion of \emph{independence} of collectives in the theory of collectives by von Mises \cite{vM64}.

Theorem~\ref{independence-independence} below gives
\emph{equivalent characterizations of the notion of
the independence of random variables in terms of that of ensembles}.
To prove Theorem~\ref{independence-independence}, we first show the following proposition.

\begin{proposition}\label{X1xn-alpha}
Let $\Omega$ be an alphabet, and let $\alpha\in\Omega^\infty$. Let
$X_1\colon\Omega\to\Omega_1,\dots,X_n\colon\Omega\to\Omega_n$ be
random variables on $\Omega$. Then
$(X_1\times\dots\times X_n)(\alpha)=X_1(\alpha)\times\dots\times X_n(\alpha)$.
\end{proposition}

\begin{proof}
First note that
both $(X_1\times\dots\times X_n)(\alpha)$ and $X_1(\alpha)\times\dots\times X_n(\alpha)$ are
infinite sequences over $\Omega_1\times\dots\times\Omega_n$.
For each $k\in\N^+$, we see that
\begin{align*}
((X_1\times\dots\times X_n)(\alpha))(k)
&=(X_1\times\dots\times X_n)(\alpha(k))
=(X_1(\alpha(k)),\dots,X_n(\alpha(k))) \\
&=((X_1(\alpha))(k),\dots,(X_n(\alpha))(k)) \\
&=(X_1(\alpha)\times\dots\times X_n(\alpha))(k).
\end{align*}
This completes the proof.
\end{proof}

\begin{theorem}\label{independence-independence}
Let $\Omega$ be an alphabet, and let $P\in\PS(\Omega)$. Let
$X_1\colon\Omega\to\Omega_1,\dots,X_n\colon\Omega\to\Omega_n$ be random variables on $\Omega$.
Then the following conditions are equivalent to one another.
\begin{enumerate}
  \item The random variables $X_1,\dots,X_n$ are independent on $P$.
  \item For every ensemble $\alpha$ for
    $P$,
    the ensembles $X_1(\alpha),\dots,X_n(\alpha)$ are independent.
  \item There exists an ensemble $\alpha$ for $P$ such that
    the ensembles $X_1(\alpha),\dots,X_n(\alpha)$ are independent.
\end{enumerate}
\end{theorem}

\begin{proof}
First,
assume that the condition~(i) holds.
Let $\alpha$ be an arbitrary ensemble for the finite probability space $P$.
It follows from Theorem~\ref{X(p)-X(P)} that $X_i(\alpha)$ is Martin-L\"of $X_i(P)$-random
for every $i=1,2,\dots,n$,
and $(X_1\times\dots\times X_n)(\alpha)$ is Martin-L\"of $(X_1\times\dots\times X_n)(P)$-random.
Therefore, by Proposition~\ref{X1xn-alpha} and Proposition~\ref{X(P)-independent}, we see that
$X_1(\alpha)\times\dots\times X_n(\alpha)$ is Martin-L\"of $X_1(P)\times\dots\times X_n(P)$-random.
Hence, we have the implication (i) $\Rightarrow$ (ii).

Since there exists an ensemble $\alpha$ for the finite probability space $P$ by Theorem~\ref{Bmae},
the implication (ii) $\Rightarrow$ (iii) is obvious.

Finally, the implication (iii) $\Rightarrow$ (i) is shown as follows.
Assume that the condition~(iii) holds.
Since $\alpha$ is an ensemble for $P$, note
from Theorem~\ref{X(p)-X(P)}
that
$X_i(\alpha)$ is an ensemble for $X_i(P)$ for every $i=1,2,\dots,n$.
Then
it follows from Definition~\ref{Def:Ind-Ensmbls} and Proposition~\ref{X1xn-alpha} that
$(X_1\times\dots\times X_n)(\alpha)$ is an ensemble for $X_1(P)\times\dots\times X_n(P)$.
On the other hand, by Theorem~\ref{X(p)-X(P)} we see that
$(X_1\times\dots\times X_n)(\alpha)$ is an ensemble for $(X_1\times\dots\times X_n)(P)$.
Thus
it follows from Corollary~\ref{uniquness} that
$X_1(P)\times\dots\times X_n(P)=(X_1\times\dots\times X_n)(P)$.
Hence,
by Proposition~\ref{X(P)-independent} we have that $X_1,\dots,X_n$ are independent on $P$.
This completes the proof.
\end{proof}

Next, we operationally characterize the notion of the \emph{independence of an arbitrary number of events},
in terms of the notion of the independence of ensembles.
For that purpose,
we first recall the notion of the independence of an arbitrary number of events from the conventional probability theory.
Let $\Omega$ be an alphabet, and let $P\in\PS(\Omega)$.
Let $A_1,\dots,A_n$ be arbitrary events on
the finite probability space
$P$.
We say that the events $A_1,\dots,A_n$ are \emph{independent on} $P$ if
for every $i_1,\dots,i_k$ with $1\le i_1<\dots <i_k\le n$ it holds that
$$P(A_{i_1}\cap\dots\cap A_{i_k})=P(A_{i_1})\dotsm P(A_{i_k}).$$
For any $A\subset\Omega$, we use $\chi_A$ to denote a function $f\colon\Omega\to\{0,1\}$ such that
$f(a)=1$ if $a\in A$ and $f(a)=0$ otherwise.
Note that
$\chara{A}{\alpha}=\chi_A(\alpha)$ for every $A\subset\Omega$ and $\alpha\in\Omega^\infty$.
It is then easy to show the following proposition.

\begin{proposition}\label{independence-random-variables-events}
Let $\Omega$ be an alphabet, and let $P\in\PS(\Omega)$. Let
$A_1,\dots,A_n\subset\Omega$.
Then the events $A_1,\dots,A_n$ are independent on $P$ if and only if
the random variables $\chi_{A_1},\dots,\chi_{A_n}$ are independent on $P$.
\qed
\end{proposition}

Now,
using Proposition~\ref{independence-random-variables-events},
Theorem~\ref{independence-independence} results in
Theorem~\ref{independence-independence-events} below,
which gives
\emph{equivalent characterizations of the notion of
the independence of an arbitrary number of events in terms of that of ensembles}.

\begin{theorem}\label{independence-independence-events}
Let $\Omega$ be an alphabet, and let $P\in\PS(\Omega)$.
Let $A_1,\dots,A_n$ be events on the finite probability space $P$.
Then the following conditions are equivalent to one another.
\begin{enumerate}
  \item The events $A_1,\dots,A_n$ are independent on $P$.
  \item For every ensemble $\alpha$ for $P$,
    the ensembles $\chara{A_1}{\alpha},\dots,\chara{A_n}{\alpha}$ are independent.
  \item There exists an ensemble $\alpha$ for $P$ such that
    the ensembles $\chara{A_1}{\alpha},\dots,\chara{A_n}{\alpha}$ are independent.\qed
\end{enumerate}
\end{theorem}

Note that the underlying finite probability spaces which we consider
both in the preceding section and in this section
are quite arbitrary, and therefore
are not required to be computable at all, in particular.

\section{Further equivalence of the notions of independence on computable finite probability spaces}
\label{FENICFPS}

In the preceding section we saw that
the independence of an arbitrary number of events/random variables
and that of ensembles are equivalent to each other on an arbitrary finite probability space.
In this section we show that
these independence notions are further equivalent to
the notion of the independence in the sense of van Lambalgen's Theorem~\cite{vL87}
in the case where the underlying finite probability space is \emph{computable}.
Thus, \emph{the three independence notions are equivalent to one another} in this case.
To show
the equivalence,
we generalize van Lambalgen's Theorem~\cite{vL87} over our framework first.

\subsection{A generalization of van Lambalgen's Theorem}
\label{van Lambalgen}

To study a generalization of van Lambalgen's Theorem,
first we generalize the notion of Martin-L\"of $P$-randomness over \emph{relativized computation}
and introduce the notion of \emph{Martin-L\"of $P$-randomness relative to an oracle}.

\emph{Relativized computation}
is a generalization of normal computation.
Let $\beta_1,\dots,\beta_\ell$ be arbitrary infinite sequences over an alphabet.
In
relativized computation,
a (deterministic) Turing machine is allowed to refer to $\beta_1,\dots,\beta_\ell$ as an \emph{oracle} during the computation.
Namely, in
relativized computation,
a Turing machine can query $(k,n)\in\{1,\dots,\ell\}\times\N^+$ at any time and then
obtains the response $\beta_k(n)$ during the computation.
Such a Turing machine is called an \emph{oracle Turing machine}.
Relativized computation
is more powerful than normal computation, in general.
See for instance Rogers~\cite{Rog67} or Soare~\cite{Soa87} for the treatment of
relativized computation in general.
See Nies~\cite{N09} and Downey and Hirschfeldt~\cite{DH10}
for the detail of the treatment of
relativized computation
in the context of
algorithmic randomness.

We define
the notion of
a
\emph{Martin-L\"of $P$-test relative to $\beta_1,\dots,\beta_\ell$}
as a Martin-L\"{o}f $P$-test
where the Turing machine computing the Martin-L\"{o}f $P$-test is an oracle Turing machine which can refer to the sequences $\beta_1,\dots,\beta_\ell$ during the computation.
Based on this notion,
we define the notion of \emph{Martin-L\"of $P$-randomness relative to $\beta_1,\dots,\beta_\ell$}
in the same manner as
(ii) and (iii)
of Definition~\ref{ML_P-randomness}.
Formally, the notion of \emph{Martin-L\"of $P$-randomness relative to infinite sequences}
is defined as follows.

\begin{definition}[Martin-L\"of $P$-randomness relative to infinite sequences]\label{ML_P-randomness_rs}
Let $\Omega$ be an alphabet, and let $P\in\PS(\Omega)$. Let
$\beta_1,\dots,\beta_\ell$ be infinite sequences over an alphabet.
A subset $\mathcal{C}$ of $\N^+\times \Omega^*$ is called
a \emph{Martin-L\"{o}f $P$-test relative to $\beta_1,\dots,\beta_\ell$}
if the following (i) and (ii) hold:
\begin{enumerate}
  \item
    There exists an oracle Turing machine $\mathcal{M}$ such that
    $\mathcal{C}
    =\{x\in\N^+\times \Omega^*\mid
    \text{$\mathcal{M}$ accepts $x$ relative to $\beta_1,\dots,\beta_\ell$}\}$;
  \item For every $n\in\N^+$ it holds that
    $\mathcal{C}_n$ is a prefix-free subset of $\Omega^*$ and
    $\Bm{P}{\osg{\mathcal{C}_n}}< 2^{-n}$
    where
    $\mathcal{C}_n$ denotes the set
    $\left\{\,
      \sigma\mid (n,\sigma)\in\mathcal{C}
    \,\right\}$.
\end{enumerate}

For any $\alpha\in \Omega^\infty$,
we say that $\alpha$ is \emph{Martin-L\"{o}f $P$-random relative to $\beta_1,\dots,\beta_\ell$}
if for every Martin-L\"{o}f $P$-test $\mathcal{C}$ relative to $\beta_1,\dots,\beta_\ell$
there exists $n\in\N^+$ such that $\alpha\notin\osg{\mathcal{C}_n}$.
\qed
\end{definition}

Just like in the definition of a Martin-L\"{o}f $P$-test given in Definition~\ref{ML_P-randomness},
we require in Definition~\ref{ML_P-randomness_rs} that
the set $\mathcal{C}_n$ is prefix-free in the definition of
a Martin-L\"{o}f $P$-test $\mathcal{C}$ relative to $\beta_1,\dots,\beta_\ell$.
However, as in the case of a Martin-L\"{o}f $P$-test,
we can eliminate this
requirement
while keeping the notion of Martin-L\"of $P$-randomness relative to $\beta_1,\dots,\beta_\ell$
the same.
Namely, we can show the following theorem
by generalizing the proof of Lemma~\ref{eliminate-prefix-freeness} over
relativized computation
in an obvious way.

\begin{theorem}\label{eliminate-prefix-freeness-relative-to-infinite-sequences}
Let $\Omega$ be an alphabet, and let $P\in\PS(\Omega)$. Let
$\beta_1,\dots,\beta_\ell$ be infinite sequences over an alphabet.
Suppose that a subset $\mathcal{C}$ of $\N^+\times \Omega^*$ satisfies
the following two conditions~(i) and (ii):
\begin{enumerate}
  \item
    There exists an oracle Turing machine $\mathcal{M}$ such that
    $\mathcal{C}
    =\{x\in\N^+\times \Omega^*\mid
    \text{$\mathcal{M}$ accepts $x$ relative to $\beta_1,\dots,\beta_\ell$}\}$;
  \item For every $n\in\N^+$ it holds that
    $\Bm{P}{\osg{\mathcal{C}_n}}< 2^{-n}$
    where
    $\mathcal{C}_n$ denotes the set
    $\left\{\,
      \sigma\mid (n,\sigma)\in\mathcal{C}
    \,\right\}$.
\end{enumerate}
Then
there exists a Martin-L\"{o}f $P$-test
$\mathcal{D}$ relative to $\beta_1,\dots,\beta_\ell$
such that
$\osg{\mathcal{C}_n}=\osg{\mathcal{D}_n}$ for every $n\in\N^+$,
where $\mathcal{D}_n$ denotes the set
$\left\{\,
  \sigma\mid (n,\sigma)\in\mathcal{D}
\,\right\}$.
\qed
\end{theorem}

From Theorem~\ref{eliminate-prefix-freeness-relative-to-infinite-sequences}
we have the following theorem,
corresponding to Theorem~\ref{ML_P-randomness_eliminated-prefix-freeness}.

\begin{theorem}\label{ML_P-randomness_eliminated-prefix-freeness-relative-to-infinite-sequences}
Let $\Omega$ be an alphabet, and let $P\in\PS(\Omega)$. Let
$\beta_1,\dots,\beta_\ell$ be infinite sequences over an alphabet.
Let $\alpha\in\Omega^\infty$.
Then the following conditions~(i) and (ii) are equivalent to each other:
\begin{enumerate} 
  \item The infinite sequence $\alpha$ is Martin-L\"{o}f $P$-random relative to $\beta_1,\dots,\beta_\ell$.
  \item For every subset $\mathcal{C}$ of $\N^+\times \Omega^*$, if
    $\Bm{P}{\osg{\mathcal{C}_n}}< 2^{-n}$ for every $n\in\N^+$ and moreover
    there exists an oracle Turing machine $\mathcal{M}$ such that
    $\mathcal{C}
    =\{x\in\N^+\times \Omega^*\mid
    \text{$\mathcal{M}$ accepts $x$ relative to $\beta_1,\dots,\beta_\ell$}\}$,
    then there exists $n\in\N^+$ such that $\alpha\notin\osg{\mathcal{C}_n}$.\qed
\end{enumerate}
\end{theorem}

Note that the underlying finite probability space $P$ is quite arbitrary
in Theorems~\ref{eliminate-prefix-freeness-relative-to-infinite-sequences} and
\ref{ML_P-randomness_eliminated-prefix-freeness-relative-to-infinite-sequences}.

The following holds, obviously.

\begin{proposition}
Let $\Omega$ be an alphabet, and let $P\in\PS(\Omega)$. Let
$\beta_1,\dots,\beta_\ell$ be infinite sequences over an alphabet.
For every $\alpha\in\Omega^\infty$,
if $\alpha$ is Martin-L\"of $P$-random relative to $\beta_1,\dots,\beta_\ell$
then $\alpha$ is Martin-L\"of $P$-random.
\qed
\end{proposition}

The converse does not necessarily hold.
In the case where $\alpha$ is Martin-L\"of $P$-random, the converse means that
the Martin-L\"of $P$-randomness of $\alpha$ is \emph{independent} of $\beta_1,\dots,\beta_\ell$
in a certain sense.

We here recall van Lambalgen's Theorem.
Let $\beta$ be an infinite sequence over an alphabet.
For any $\alpha\in\XI$,
we say that $\alpha$ is \emph{Martin-L\"{o}f random relative to $\beta$} if
$\alpha$ is Martin-L\"{o}f $U$-random relative to $\beta$ where
$U\in\PS(\{0,1\})$
such that $U(0)=U(1)=1/2$.
Based on this notion of \emph{Martin-L\"of randomness relative to an
infinite sequence},
van Lambalgen's Theorem is stated as follows.

\begin{theorem}[van Lambalgen's Theorem, van Lambalgen \cite{vL87}]
Let $\alpha,\beta\in\{0,1\}^\infty$,
and let $\alpha\oplus\beta$ denote the infinite binary sequence
\[
  \alpha(1)\beta(1)\alpha(2)\beta(2)\alpha(3)\beta(3)\dotsc\dotsc.
\]
Then the following conditions~(i) and (ii) are equivalent to each other:
\begin{enumerate}
  \item $\alpha\oplus\beta$ is Martin-L\"of random.
  \item $\alpha$ is Martin-L\"of random relative to $\beta$, and moreover $\beta$ is Martin-L\"of random.\qed
\end{enumerate}
\end{theorem}

We generalize van Lambalgen's Theorem as follows.

\begin{theorem}[Generalization of van Lambalgen's Theorem I]\label{gvL}
Let $\Omega_1$ and $\Omega_2$ be alphabets, and let $P_1\in\PS(\Omega_1)$ and $P_2\in\PS(\Omega_2)$.
Let $\alpha_1\in\Omega_1^\infty$ and $\alpha_2\in\Omega_2^\infty$, and
let $\beta_1,\dots,\beta_\ell$ be infinite sequences over an alphabet. 
Suppose that $P_1$ is computable.
Then $\alpha_1\times\alpha_2$ is Martin-L\"{o}f $P_1\times P_2$-random relative to
$\beta_1,\dots,\beta_\ell$
if and only if
$\alpha_1$ is Martin-L\"{o}f $P_1$-random relative to
$\alpha_2,\beta_1,\dots,\beta_\ell$
and moreover $\alpha_2$ is Martin-L\"{o}f $P_2$-random relative to $\beta_1,\dots,\beta_\ell$.
\qed
\end{theorem}

The proof of Theorem~\ref{gvL} is obtained by
generalizing and elaborating
the proof of van Lambalgen's Theorem given in Nies~\cite[Section 3.4]{N09}.
The detail
of the proof of Theorem~\ref{gvL}
is given in the subsequent two subsections.
Note that in Theorem~\ref{gvL},
the computability of $P_1$ is
assumed
while that of $P_2$ is not required.

We have Theorem~\ref{cor_gvL} below based on Theorem~\ref{gvL}.
Note that the computability of $P_n$ is not required in Theorem~\ref{cor_gvL}.

\begin{theorem}[Generalization of van Lambalgen's Theorem II]\label{cor_gvL}
Let $n\ge 2$.
Let $\Omega_1,\dots,\Omega_n$ be alphabets, and let $P_1\in\PS(\Omega_1),\dots,P_n\in\PS(\Omega_n)$.
For each $i=1,\dots,n$, let $\alpha_i\in\Omega_i^\infty$, and
let $\beta_1,\dots,\beta_\ell$ be infinite sequences over an alphabet.
Suppose that $P_1,\dots,P_{n-1}$ are computable.
Then $\alpha_1\times\dots\times\alpha_n$ is
Martin-L\"of $P_1\times\dots\times P_n$-random relative to $\beta_1,\dots,\beta_\ell$ if and only if
for every $k=1,\dots,n$ it holds that
$\alpha_k$ is Martin-L\"of $P_k$-random relative to
$\alpha_{k+1},\dots, \alpha_n,\beta_1,\dots,\beta_\ell$.
\end{theorem}

\begin{proof}
We show the result by induction on $n\ge 2$.
In the case of $n=2$, the result holds since it is precisely Theorem~\ref{gvL}.

For an arbitrary $m\ge 2$, assume that the result holds for $n=m$.
Let $\Omega_1,\dots,\Omega_{m+1}$ be alphabets,
and let $P_1\in\PS(\Omega_1),\dots,P_{m+1}\in\PS(\Omega_{m+1})$.
For each $i=1,\dots,m+1$, let $\alpha_i\in\Omega_i^\infty$, and
let $\beta_1,\dots,\beta_\ell$ be infinite sequences over an alphabet.
Suppose that $P_1,\dots,P_m$ are computable.
Then, by applying Theorem~\ref{gvL} with $P_1\times\dots\times P_m$ as $P_1$, $P_{m+1}$ as $P_2$,
$\alpha_1\times\dots\times \alpha_m$ as $\alpha_1$, and $\alpha_{m+1}$ as $\alpha_2$
in Theorem~\ref{gvL},
we have that $(\alpha_1\times\dots\times\alpha_m)\times\alpha_{m+1}$ is
Martin-L\"{o}f $(P_1\times\dots\times P_m)\times P_{m+1}$-random relative to $\beta_1,\dots,\beta_\ell$
if and only if
$\alpha_1\times\dots\times\alpha_m$ is Martin-L\"{o}f $P_1\times\dots\times P_m$-random
relative to $\alpha_{m+1},\beta_1,\dots,\beta_\ell$ and
$\alpha_{m+1}$ is Martin-L\"{o}f $P_{m+1}$-random relative to $\beta_1,\dots,\beta_\ell$.
Thus, by applying the result for $n=m$ we have the result for $n=m+1$.
This completes the proof.
\end{proof}

\subsection{The proof of the ``only if'' part of Theorem~\ref{gvL}}
\label{section-only-if-part}

We prove the following theorem, from which the ``only if'' part of Theorem~\ref{gvL} follows.

\begin{theorem}\label{only-if-part}
Let $\Omega_1$ and $\Omega_2$ be alphabets, and let $P_1\in\PS(\Omega_1)$ and $P_2\in\PS(\Omega_2)$.
Let $\alpha_1\in\Omega_1^\infty$ and $\alpha_2\in\Omega_2^\infty$, and
let $\beta_1,\dots,\beta_\ell$ be infinite sequences over an alphabet $\Theta$. 
Suppose that $P_1$ is right-computable.
If $\alpha_1\times\alpha_2$ is Martin-L\"{o}f $P_1\times P_2$-random relative to
$\beta_1,\dots,\beta_\ell$, then $\alpha_1$ is Martin-L\"of $P_1$-random relative to
$\alpha_2,\beta_1,\dots,\beta_\ell$
and moreover $\alpha_2$ is Martin-L\"{o}f $P_2$-random relative to $\beta_1,\dots,\beta_\ell$.
\qed
\end{theorem}

In order to prove Theorem~\ref{only-if-part},
we use the notion of \emph{universal Martin-L\"of $P$-test relative to infinite sequences}.

\begin{definition}[Universal Martin-L\"of $P$-test relative to infinite sequences]\label{UML_P-test_rs}
Let $\Omega$ be an alphabet, and let $P\in\PS(\Omega)$.
Let $\Theta$ be an alphabet, and let $\ell\in\N^+$.
An oracle Turing machine $\mathcal{M}$ is called a
\emph{universal Martin-L\"of $P$-test relative to $\ell$ infinite sequences over $\Theta$}
if for every $\beta_1,\dots,\beta_\ell\in\Theta^\infty$ there exists
$\mathcal{C}$
such that
\begin{enumerate}
  \item 
    $\mathcal{C}=\{x\in\N^+\times \Omega^*\mid
    \text{$\mathcal{M}$ accepts $x$ relative to $\beta_1,\dots,\beta_\ell$}\}$,
  \item for every $n\in\N^+$ it holds that
    $\mathcal{C}_n$ is a prefix-free subset of $\Omega^*$ and
    $\Bm{P}{\osg{\mathcal{C}_n}}<2^{-n}$
    where
    $\mathcal{C}_n$ denotes the set
    $\left\{\,
      \sigma\mid (n,\sigma)\in\mathcal{C}
    \,\right\}$, and
  \item for every Martin-L\"{o}f $P$-test $\mathcal{D}$ relative to $\beta_1,\dots,\beta_\ell$,
    \[
      \bigcap_{n=1}^\infty \osg{\mathcal{D}_n}\subset \bigcap_{n=1}^\infty \osg{\mathcal{C}_n}.
    \]
    \qed
\end{enumerate}
\end{definition}

It is then easy to show the following theorem.

\begin{theorem}\label{ExistsUML_P-test_rs}
Let $\Omega$ be an alphabet, and let $P\in\PS(\Omega)$.
Let $\Theta$ be an alphabet, and let $\ell\in\N^+$.
Suppose that $P$ is right-computable.
Then there exists a universal Martin-L\"of $P$-test
relative to $\ell$ infinite sequences over $\Theta$.
\qed
\end{theorem}

In a similar manner to the proof of Theorem~\ref{contraction}
we can show the following theorem in the context of relativized computation.

\begin{theorem}\label{contraction-relative}
Let $\Omega$ be an alphabet, and let $P\in\PS(\Omega)$. Let
$\beta_1,\dots,\beta_\ell$ be infinite sequences over an alphabet.
Let $\alpha\in\Omega^\infty$, and
let $a$ and $b$ be distinct elements
of
$\Omega$.
Suppose that $\gamma$ is an infinite sequence over $\Omega\setminus\{b\}$
obtained by replacing all occurrences of $b$ by $a$ in $\alpha$.
If $\alpha$ is Martin-L\"of $P$-random relative to $\beta_1,\dots,\beta_\ell$
then $\gamma$ is Martin-L\"of $Q$-random relative to $\beta_1,\dots,\beta_\ell$,
where
$Q\in\PS(\Omega\setminus\{b\})$
such that $Q(x):=P(a)+P(b)$ if $x=a$ and $Q(x):=P(x)$ otherwise.\qed
\end{theorem}

Then, using Theorems~\ref{contraction-relative} and \ref{ExistsUML_P-test_rs},
we can prove Theorem~\ref{only-if-part} as follows.

\begin{proof}[Proof of Theorem~\ref{only-if-part}]
First, we show that
if $\alpha_1\times\alpha_2$ is Martin-L\"{o}f $P_1\times P_2$-random relative to
$\beta_1\dots,\beta_\ell$
then $\alpha_2$ is Martin-L\"of $P_2$-random relative to $\beta_1,\dots,\beta_\ell$.
This is easily shown using Theorem~\ref{contraction-relative} repeatedly.

Next, we show that if $\alpha_1\times\alpha_2$ is Martin-L\"{o}f $P_1\times P_2$-random relative to
$\beta_1\dots,\beta_\ell$
then $\alpha_1$ is Martin-L\"of $P_1$-random relative to $\alpha_2,\beta_1,\dots,\beta_\ell$.
Since $P_1$ is right-computable, it follows from Theorem~\ref{ExistsUML_P-test_rs} that
there exists a universal Martin-L\"of $P_1$-test
relative to $\ell+1$ infinite sequences over $\Omega_2\cup\Theta$.
Thus, there exists an oracle Turing machine $\mathcal{M}$ such that
for every $\gamma\in\Omega_2^\infty$ there exists
$\mathcal{C}$
such that
\begin{enumerate}
  \item
    $\mathcal{C}
    =\{x\in\N^+\times \Omega_1^*\mid
    \text{$\mathcal{M}$ accepts $x$ relative to $\gamma,\beta_1,\dots,\beta_\ell$}\}$,
  \item for every $n\in\N^+$ it holds that $\mathcal{C}_n$ is a prefix-free subset of $\Omega_1^*$ and
    $\Bm{P_1}{\osg{\mathcal{C}_n}}<2^{-n}$, and
  \item for every Martin-L\"{o}f $P_1$-test $\mathcal{D}$ relative to $\gamma,\beta_1,\dots,\beta_\ell$,
    \begin{equation}\label{bcDnssbcCn-UMLP1}
      \bigcap_{n=1}^\infty \osg{\mathcal{D}_n}\subset \bigcap_{n=1}^\infty \osg{\mathcal{C}_n}.
    \end{equation}
\end{enumerate}
We choose any particular $a\in\Omega_2$.
Then,
for each $\sigma\in\Omega_2^*$,
let $\mathcal{U}^\sigma$ be the set of all $x\in\N^+\times \Omega_1^*$ such that
$\mathcal{M}$ accepts $x$ relative to $\sigma a^\infty,\beta_1,\dots,\beta_\ell$
with oracle access only to the prefix of $\sigma a^\infty$ of length $\abs{\sigma}$
in the first infinite sequence.
Here, $\sigma a^\infty$ denotes the infinite sequence over $\Omega_2$
which is the concatenation of the finite string $\sigma$ and the infinite sequence consisting only of $a$.
It follows that
\begin{equation}\label{BmP1osgmUsn<2-n_RtIS}
  \Bm{P_1}{\osg{\mathcal{U}^\sigma_n}}<2^{-n}
\end{equation}
for every $\sigma\in\Omega_2^*$ and every $n\in\N^+$,
where
$\mathcal{U}^\sigma_n:=
    \left\{\,
      \tau\bigm|(n,\tau)\in\mathcal{U}^\sigma
    \,\right\}$.
For each $k,n\in\N^+$, let
\[
  G_n(k)=\{u\times\sigma\mid u\in\Omega_1^k\text{ \& }\sigma\in\Omega_2^k\text{ \& Some prefix of $u$ is in $\mathcal{U}^\sigma_n$}\}.
\]
Then,
it is easy to see that
$G_n(k)$ is r.e.~relative to $\beta_1,\dots,\beta_\ell$ uniformly in $n$ and $k$.
Note that
\[
  \osg{G_n(k)}=\bigcup_{\sigma\in\Omega_2^k}\bigcup_{u\in S_n(k,\sigma)}\osg{u\times\sigma},
\]
for every $n,k\in\N^+$,
where $S_n(k,\sigma):=\{u\in\Omega_1^k\mid\text{Some prefix of $u$ is in $\mathcal{U}^\sigma_n$}\}$.
Therefore, for each $n,k\in\N^+$,
we see that
\begin{align*}
  \Bm{P_1\times P_2}{\osg{G_n(k)}}
  &=\sum_{\sigma\in\Omega_2^k}\sum_{u\in S_n(k,\sigma)}\Bm{P_1\times P_2}{\osg{u\times\sigma}}
  =\sum_{\sigma\in\Omega_2^k}\Bm{P_2}{\osg{\sigma}}\sum_{u\in S_n(k,\sigma)}\Bm{P_1}{\osg{u}} \\
  &=\sum_{\sigma\in\Omega_2^k}\Bm{P_2}{\osg{\sigma}}
      \Bm{P_1}{\osg{\mathcal{U}^\sigma_n\cap\Omega_1^{\le k}}}
  \le \sum_{\sigma\in\Omega_2^k}\Bm{P_1}{\osg{\mathcal{U}^\sigma_n}}\Bm{P_2}{\osg{\sigma}} \\
  &<\sum_{\sigma\in\Omega_2^k}2^{-n}\Bm{P_2}{\osg{\sigma}}
  = 2^{-n},
\end{align*}
where the last inequality follows from \eqref{BmP1osgmUsn<2-n_RtIS} (and
the fact that $\Bm{P_2}{\osg{\sigma}}>0$ for some $\sigma\in\Omega_2^k$). 
On the other hand, it follows that $\osg{G_n(k)}\subset\osg{G_n(k+1)}$
for every $n,k\in\N^+$.
For each $n\in\N^+$, let $G_n=\bigcup_{k=1}^\infty G_n(k)$.
Then $G_n$ is r.e.~relative to $\beta_1,\dots,\beta_\ell$ uniformly in $n$, and
\[
  \Bm{P_1\times P_2}{\osg{G_n}}\le 2^{-n}
\]
for every $n\in\N^+$.
We define $\mathcal{A}$ to be a subset of $\N^+\times (\Omega_1\times\Omega_2)^*$ such that
$\mathcal{A}_n=G_{n+1}$ for every $n\in\N^+$.
Then  $\mathcal{A}$ is r.e.~relative to $\beta_1,\dots,\beta_\ell$ and
\begin{equation}\label{BmP1timesP2omAn<2-n}
  \Bm{P_1\times P_2}{\osg{\mathcal{A}_n}}< 2^{-n}
\end{equation}
for every $n\in\N^+$.

Now, assume that
$\alpha_1$ is not Martin-L\"of $P_1$-random relative to $\alpha_2,\beta_1,\dots,\beta_\ell$.
Then,
there exists a Martin-L\"{o}f $P_1$-test $\mathcal{E}$ relative to $\alpha_2,\beta_1,\dots,\beta_\ell$
such that $\alpha_1\in\osg{\mathcal{E}_n}$ for every $n\in\N^+$.
It follows from \eqref{bcDnssbcCn-UMLP1} that
there exists
$\mathcal{C}$
such that
\begin{enumerate}
  \item
    $\mathcal{C}
    =\{x\in\N^+\times \Omega_1^*\mid
    \text{$\mathcal{M}$ accepts $x$ relative to $\alpha_2,\beta_1,\dots,\beta_\ell$}\}$, and
  \item
    \[
      \alpha_1\in\bigcap_{n=1}^\infty \osg{\mathcal{C}_n}.
    \]
\end{enumerate}
Let $n\in\N^+$.
Then there exists $m\in\N^+$ such that $\rest{\alpha_1}{m}\in\mathcal{C}_{n+1}$.
Then, there exists $k\ge m$ such that
$\mathcal{M}$ accepts $(n+1,\rest{\alpha_1}{m})$ relative to $\alpha_2,\beta_1,\dots,\beta_\ell$
with oracle access only to the prefix of $\alpha_2$ of length $k$ in the first infinite sequence $\alpha_2$.
It follows that $\rest{\alpha_1}{m}\in\mathcal{U}^{\reste{\alpha_2}{k}}_{n+1}$.
Thus, $\rest{\alpha_1}{k}\times\rest{\alpha_2}{k}\in G_{n+1}(k)$,
and therefore $\alpha_1\times\alpha_2\in\osg{G_{n+1}(k)}\subset\osg{G_{n+1}}=\osg{\mathcal{A}_n}$.
Hence,
it follows from
\eqref{BmP1timesP2omAn<2-n} and
Theorem~\ref{ML_P-randomness_eliminated-prefix-freeness-relative-to-infinite-sequences} that
$\alpha_1\times\alpha_2$ is not Martin-L\"{o}f $P_1\times P_2$-random relative to
$\beta_1,\dots,\beta_\ell$.
This completes the proof.
\end{proof}

\subsection{The proof of the ``if'' part of Theorem~\ref{gvL}}
\label{section-if-part}

Next, we prove the following theorem, from which the ``if'' part of Theorem~\ref{gvL} follows.

\begin{theorem}\label{if-part}
Let $\Omega_1$ and $\Omega_2$ be alphabets, and let $P_1\in\PS(\Omega_1)$ and $P_2\in\PS(\Omega_2)$.
Let $\alpha_1\in\Omega_1^\infty$ and $\alpha_2\in\Omega_2^\infty$, and
let $\beta_1,\dots,\beta_\ell$ be infinite sequences over an alphabet. 
Suppose that $P_1$ is left-computable.
If $\alpha_1$ is Martin-L\"of $P_1$-random relative to $\alpha_2,\beta_1,\dots,\beta_\ell$
and $\alpha_2$ is Martin-L\"{o}f $P_2$-random relative to $\beta_1,\dots,\beta_\ell$,
then
$\alpha_1\times\alpha_2$ is Martin-L\"{o}f $P_1\times P_2$-random relative to
$\beta_1,\dots,\beta_\ell$. 
\end{theorem}

\begin{proof}
Suppose that $\alpha_1\times\alpha_2$ is not Martin-L\"{o}f $P_1\times P_2$-random relative to
$\beta_1,\dots,\beta_\ell$. 
Then there exists a Martin-L\"of $P_1\times P_2$-test $\mathcal{V}$ relative to
$\beta_1,\dots,\beta_\ell$ such that
\begin{enumerate}
  \item $\mathcal{V}_d$ is prefix-free
    for every $d\in\N^+$,
  \item
    \begin{equation}\label{BmP1tP2omVd<2-2d}
      \Bm{P_1\times P_2}{\osg{\mathcal{V}_d}}<2^{-2d}
    \end{equation}
    for every $d\in\N^+$, and
  \item $\alpha_1\times\alpha_2\in\osg{\mathcal{V}_d}$
    for every $d\in\N^+$.
\end{enumerate}

On the one hand, for each $x\in\Omega_2^*$, we use $[\emptyset\times x]$ to denote the set
\[
  \{\gamma_1\times\gamma_2\mid
  \text{$\gamma_1\in\Omega_1^\infty$, $\gamma_2\in\Omega_2^\infty$,
  and $x$ is a prefix of $\gamma_2$}\}.
\]
On the other hand, 
for each $\sigma\in(\Omega_1\times\Omega_2)^*$,
we use $\Projl{\sigma}$ to denote
a finite string $\sigma_1\in\Omega_1^*$ such that
there exists a finite string $\sigma_2\in\Omega_2^*$ with the property that
$\sigma=\sigma_1\times\sigma_2$.
For each $x\in\Omega_2^*$ and $W\subset(\Omega_1\times\Omega_2)^{\le\abs{x}}$,
we use $F(W,x)$ to denote the set of all
$\sigma\in W$ such that
there exist $\sigma_1\in\Omega_1^*$ and $\sigma_2\in\Omega_2^*$ for which
(i) $\sigma=\sigma_1\times\sigma_2$ and
(ii) $\sigma_2$ is a prefix of $x$.
Finally, for each $x\in\Omega_2^*$ and $W\subset(\Omega_1\times\Omega_2)^{\le\abs{x}}$,
we use $P_1(W,x)$ to denote the sum
\[
  \sum_{\sigma\in F(W,x)}P_1(\Projl{\sigma}).
\]
Then, it is easy to see that
\begin{equation}\label{FWx}
  P_1(W,x)P_2(x)=\Bm{P_1\times P_2}{\osg{W}\cap[\emptyset\times x]}
\end{equation}
for every $x\in\Omega_2^*$ and every prefix-free subset $W$ of
$(\Omega_1\times\Omega_2)^{\le\abs{x}}$.

For each $d\in\N^+$, let
\[
  S_d=\bigl\{x\in\Omega_2^*\bigm|
  2^{-d}<P_1(\mathcal{V}_d\cap (\Omega_1\times\Omega_2)^{\le\abs{x}},x)\bigr\}.
\]
Since $P_1$ is left-computable, $S_d$ is r.e.~relative to $\beta_1,\dots,\beta_\ell$ uniformly in $d$.

Let $d\in\N^+$.
Let $\{x_i\}$ be a listing of the minimal strings in $S_d$.
It follows
from \eqref{FWx}
that
\begin{align*}
  2^{-d}\Bm{P_2}{\osg{x_i}}&=2^{-d}P_2(x_i)
  \le P_1(\mathcal{V}_d\cap (\Omega_1\times\Omega_2)^{\le\abs{x_i}},x_i)P_2(x_i) \\
  &=\Bm{P_1\times P_2}{\osg{\mathcal{V}_d\cap (\Omega_1\times\Omega_2)^{\le\abs{x_i}}}\cap[\emptyset\times x_i]} \\
  &\le\Bm{P_1\times P_2}{\osg{\mathcal{V}_d}\cap[\emptyset\times x_i]}.
\end{align*}
Since the sets
$\{\osg{\mathcal{V}_d}\cap[\emptyset\times x_i]\}_i$
are pairwise disjoint, we have
that
\[
  \sum_{i}2^{-d}\Bm{P_2}{\osg{x_i}}
  \le\sum_{i}\Bm{P_1\times P_2}{\osg{\mathcal{V}_d}\cap[\emptyset\times x_i]}
  \le\Bm{P_1\times P_2}{\osg{\mathcal{V}_d}}<2^{-2d},
\]
where the last inequality follows from \eqref{BmP1tP2omVd<2-2d}.
Thus, since $\Bm{P_2}{\osg{S_d}}=\sum_{i}\Bm{P_2}{\osg{x_i}}$ and
$d$ is an arbitrary positive integer,
we have that
\[
  \Bm{P_2}{\osg{S_d}}<2^{-d}
\]
for every $d\in\N^+$.
For each $d\in\N^+$, let $T_d=\bigcup_{c=d}^\infty S_{c+1}$.
It follows that
\begin{equation}\label{BmP2oTdlesBmP2oSc+1<2^-d}
  \Bm{P_2}{\osg{T_d}}\le\sum_{c=d}^\infty\Bm{P_2}{\osg{S_{c+1}}}<2^{-d}
\end{equation}
for each $d\in\N^+$, and
$T_d$
is r.e.~relative to $\beta_1,\dots,\beta_\ell$
uniformly in $d$.

Now, let us assume that
$\alpha_2$ is Martin-L\"of $P_2$-random relative to $\beta_1,\dots,\beta_\ell$.
We will then show that
$\alpha_1$ is not Martin-L\"of $P_1$-random relative to $\alpha_2,\beta_1,\dots,\beta_\ell$,
in what follows.
If $\alpha_2\in\osg{S_d}$ for infinitely many
$d\in\N^+$,
then we have that
$\alpha_2\in\osg{T_d}$
for every
$d\in\N^+$,
and therefore
using
\eqref{BmP2oTdlesBmP2oSc+1<2^-d} and
Theorem~\ref{ML_P-randomness_eliminated-prefix-freeness-relative-to-infinite-sequences}
we have that
$\alpha_2$ is not Martin-L\"of $P_2$-random relative to $\beta_1,\dots,\beta_\ell$.
This contradicts the assumption.
Thus, there must exists $d_0\in\N^+$ such that
\begin{equation}\label{alph2notinoSd_if-part}
  \alpha_2\notin\osg{S_d}
\end{equation}
for every $d>d_0$.

For each $d,n\in\N^+$, let
\[
  H_d(n)=\{w\in\Omega_1^n\mid
  \osg{w\times\rest{\alpha_2}{n}}\subset\osg{\mathcal{V}_d\cap(\Omega_1\times\Omega_2)^{\le n}}\}.
\]
Let $d,n\in\N^+$, and let ${w_1,\dots,w_m}$ be a listing of all elements of $H_d(n)$.
Since
\[
  \osg{w_i\times\rest{\alpha_2}{n}}\subset
  \osg{\mathcal{V}_d\cap(\Omega_1\times\Omega_2)^{\le n}}\cap[\emptyset\times\rest{\alpha_2}{n}]
\]
for every $i=1,\dots,m$,
and the sets
$\{\osg{w_i\times\rest{\alpha_2}{n}}\}_i$
are pairwise disjoint,
we see that
\begin{equation}\label{postFWx}
\begin{split}
  \Bm{P_1}{\osg{H_d(n)}}\Bm{P_2}{\osg{\rest{\alpha_2}{n}}}
  &=\left(\sum_{i=1}^m\Bm{P_1}{\osg{w_i}}\right)\Bm{P_2}{\osg{\rest{\alpha_2}{n}}} \\
  &=\sum_{i=1}^m\Bm{P_1}{\osg{w_i}}\Bm{P_2}{\osg{\rest{\alpha_2}{n}}} \\
  &=\sum_{i=1}^m\Bm{P_1\times P_2}{\osg{w_i\times\rest{\alpha_2}{n}}} \\
  &=\Bm{P_1\times P_2}{\bigcup_{i=1}^m\osg{w_i\times\rest{\alpha_2}{n}}} \\
  &\le\Bm{P_1\times P_2}{\osg{\mathcal{V}_d\cap(\Omega_1\times\Omega_2)^{\le n}}\cap[\emptyset\times\rest{\alpha_2}{n}]}.
\end{split}
\end{equation}
Assume that $d>d_0$.
Then, using \eqref{alph2notinoSd_if-part} we have $\rest{\alpha_2}{n}\notin S_d$.
Therefore,
it follows from \eqref{FWx} that
\begin{align*}
  \Bm{P_1\times P_2}{\osg{\mathcal{V}_d\cap(\Omega_1\times\Omega_2)^{\le n}}\cap[\emptyset\times\rest{\alpha_2}{n}]}
  &=P_1(\mathcal{V}_d\cap(\Omega_1\times\Omega_2)^{\le n},\rest{\alpha_2}{n})P_2(\rest{\alpha_2}{n}) \\
  &\le 2^{-d}\Bm{P_2}{\osg{\rest{\alpha_2}{n}}}.
\end{align*}
Thus,
using \eqref{postFWx} we have
\[
  \Bm{P_1}{\osg{H_d(n)}}\Bm{P_2}{\osg{\rest{\alpha_2}{n}}}\le 2^{-d}\Bm{P_2}{\osg{\rest{\alpha_2}{n}}}.
\]
Since $\alpha_2$ is Martin-L\"of $P_2$-random relative to $\beta_1,\dots,\beta_\ell$,
we can show that $\Bm{P_2}{\osg{\rest{\alpha_2}{n}}}>0$,
in a similar manner to the proof of (i) of Theorem~\ref{zero_probability}.
Hence, we see that
\begin{equation}\label{PHdn}
  \Bm{P_1}{\osg{H_d(n)}}\le 2^{-d}
\end{equation}
for every $d>d_0$ and
$n\in\N^+$.

On the other hand, we see that $\osg{H_d(n)}\subset\osg{H_d(n+1)}$
for every $d\in\N^+$ and $n\in\N^+$.
For each $d\in\N^+$, let $H_d=\bigcup_{n=1}^\infty H_{d+d_0}(n)$.
It follows from \eqref{PHdn} that
\begin{equation}\label{BmP1oHd<2-d_if-part}
  \Bm{P_1}{\osg{H_d}}<2^{-d}
\end{equation}
for every $d\in\N^+$.
It is easy to show that
$H_d(n)=\{w\in\Omega_1^n\mid\text{Some prefix of $w\times\rest{\alpha_2}{n}$ is in $\mathcal{V}_d$}\}$
for every $d,n\in\N^+$.
It follows that $H_d$ is r.e.~relative to $\alpha_2,\beta_1,\dots,\beta_\ell$ uniformly in $d$.

Let $d\in\N^+$.
Then,
since $\alpha_1\times\alpha_2\in\osg{\mathcal{V}_{d+d_0}}$, there exists $n\in\N^+$ such that
$\rest{(\alpha_1\times\alpha_2)}{n}\in\mathcal{V}_{d+d_0}$.
It follows that
$\rest{\alpha_1}{n}\times\rest{\alpha_2}{n}\in\mathcal{V}_{d+d_0}\cap(\Omega_1\times\Omega_2)^{\le n}$,
and therefore $\rest{\alpha_1}{n}\in H_{d+d_0}(n)$.
It follows that $\alpha_1\in\osg{H_{d+d_0}(n)}\subset\osg{H_d}$.
Therefore, $\alpha_1\in\osg{H_d}$ for every $d\in\N^+$.
Hence,
using
\eqref{BmP1oHd<2-d_if-part} and
Theorem~\ref{ML_P-randomness_eliminated-prefix-freeness-relative-to-infinite-sequences}
we have that
$\alpha_1$ is not Martin-L\"of $P_1$-random relative to $\alpha_2,\beta_1,\dots,\beta_\ell$,
as desired.
This completes the proof.
\end{proof}

\subsection{Equivalence between the three independence notions on computable finite probability spaces}

Theorem~\ref{mgvL} below
gives an equivalent characterization of the notion of the independence of
ensembles in terms of Martin-L\"of $P$-randomness relative to an oracle.

\begin{theorem}[Generalization of van Lambalgen's Theorem III]\label{mgvL}
Let $n\ge 2$.
Let $\Omega_1,\dots,\Omega_n$ be alphabets, and let $P_1\in\PS(\Omega_1),\dots,P_n\in\PS(\Omega_n)$.
Let $\alpha_1,\dots,\alpha_n$ be ensembles for $P_1,\dots,P_n$, respectively.
Suppose that $P_1,\dots,P_{n-1}$ are computable.
Then the ensembles $\alpha_1,\dots,\alpha_n$ are independent if and only if
for every $k=1,\dots,n-1$ it holds that
$\alpha_k$ is Martin-L\"of $P_k$-random relative to $\alpha_{k+1},\dots,\alpha_n$.
\end{theorem}

\begin{proof}
Theorem~\ref{mgvL} follows immediately from Theorem~\ref{cor_gvL}.
\end{proof}

Note that the computability of $P_n$ is not required in Theorem~\ref{mgvL}.
Combining Theorem~\ref{independence-independence} with Theorem~\ref{mgvL}
we obtain the following theorem.

\begin{theorem}\label{ind-vL}
Let $\Omega$ be an alphabet, and let $P\in\PS(\Omega)$. Let
$X_1\colon\Omega\to\Omega_1,\dots,X_n\colon\Omega\to\Omega_n$ be random variables on $\Omega$.
Suppose that $X_1(P),\dots,X_{n-1}(P)$ are computable.
Then the following conditions are equivalent to one another.
\begin{enumerate}
  \item The random variables $X_1,\dots,X_n$ are independent on $P$.
  \item For every ensemble $\alpha$ for $P$ and every $k=1,\dots,n-1$ it holds that
    $X_k(\alpha)$ is Martin-L\"of $X_k(P)$-random relative to $X_{k+1}(\alpha),\dots,X_n(\alpha)$.
  \item There exists an ensemble $\alpha$ for $P$ such that for every $k=1,\dots,n-1$ it holds that
    $X_k(\alpha)$ is Martin-L\"of $X_k(P)$-random relative to $X_{k+1}(\alpha),\dots,X_n(\alpha)$.
\end{enumerate}
\end{theorem}

\begin{proof}
Let $\alpha$ be an arbitrary ensemble for $P$.
Then it follows from Theorem~\ref{X(p)-X(P)} that
$X_1(\alpha),\dots,X_n(\alpha)$ are ensembles for $X_1(P),\dots,X_n(P)$, respectively.
Therefore, in the case where $X_1(P),\dots,X_{n-1}(P)$ are computable,
using Theorem~\ref{mgvL} we have that
the ensembles $X_1(\alpha),\dots,X_n(\alpha)$ are independent if and only if
for every $k=1,\dots,n-1$ it holds that
$X_k(\alpha)$ is Martin-L\"of $X_k(P)$-random relative to $X_{k+1}(\alpha),\dots,X_n(\alpha)$.
Thus, Theorem~\ref{ind-vL} follows from Theorem~\ref{independence-independence}.
\end{proof}

We remark
that the computability of the finite probability space $X_n(P)$ is not required in Theorem~\ref{ind-vL}.
This fact
plays a crucial role in providing
new equivalent characterizations of the notion of perfect secrecy
for an arbitrary encryption scheme
by
Theorem~\ref{PS-AR} below,
when we apply
Theorem~\ref{ind-vL}
to cryptography in Section~\ref{subsec:AC}.

Theorem~\ref{independence-independence} and Theorem~\ref{ind-vL} together show that
\emph{the three independence notions we have considered so far:
the independence of random variables, the independence of ensembles,
and the independence in the sense of van Lambalgen's Theorem,
are equivalent to one another} on an arbitrary \emph{computable} finite probability space.

Theorem~\ref{ind-vL} results in Theorem~\ref{ind-events-vL} below.
Theorem~\ref{independence-independence-events} and Theorem~\ref{ind-events-vL} together
show that
\emph{the three independence notions are equivalent
to one another also
for arbitrary events}, instead of random variables,
on an arbitrary \emph{computable} finite probability space.

\begin{theorem}\label{ind-events-vL}
Let $\Omega$ be an alphabet, and let $P\in\PS(\Omega)$.
Let $A_1,\dots,A_n$ be events on the finite probability space $P$.
Suppose that the finite probability space $\charaps{P}{A_k}$ is computable for every $k=1,\dots,n-1$.
Then the following conditions are equivalent to one another.
\begin{enumerate}
  \item The events $A_1,\dots,A_n$ are independent on $P$.
  \item For every ensemble $\alpha$ for $P$ and every $k=1,\dots,n-1$ it holds that
    $\chara{A_k}{\alpha}$ is Martin-L\"of $\charaps{P}{A_k}$-random relative to $\chara{A_{k+1}}{\alpha},\dots,\chara{A_n}{\alpha}$.
  \item There exists an ensemble $\alpha$ for $P$ such that for every $k=1,\dots,n-1$ it holds that
    $\chara{A_k}{\alpha}$ is Martin-L\"of $\charaps{P}{A_k}$-random relative to $\chara{A_{k+1}}{\alpha},\dots,\chara{A_n}{\alpha}$.
\end{enumerate}
\end{theorem}

\begin{proof}
The result is obtained by applying Theorem~\ref{ind-vL}
to the random variables
$\chi_{A_1},\dots,\chi_{A_n}$
as $X_1,\dots,X_n$, respectively,
and then using Proposition~\ref{independence-random-variables-events}.
\end{proof}

\section{Effectivization of the law of large numbers for computable finite probability spaces}
\label{Section:EffectiveLLN}

In Section~\ref{Subsec:The law of large numbers},
we
proved Theorem~\ref{LLN}, which states that the law of large numbers holds for an \emph{arbitrary} ensemble.
In that theorem,
the underlying finite probability space is quite arbitrary, and therefore
is not required to be computable at all, in particular.
In this regard, the following question may arise naturally.
\begin{quote}
\textbf{Question}:
What role does
the computability of the underlying finite probability space
play in the law of large numbers for an ensemble?
\end{quote}
In this section, we see that
the computability of the underlying finite probability space
leads to an \emph{effectivization} of the law of large numbers.

Let $\Omega$ be an alphabet, and let $P\in\PS(\Omega)$.
Let $\alpha$ be an ensemble for $P$, and let $a\in\Omega$.
Then, Theorem~\ref{LLN} states
that for every real $\varepsilon>0$ there exists
$n_0\in\N^+$
such that
for every $n\in\N^+$ if
$n\ge n_0$
then 
\begin{equation}\label{LLN-classical}
  \abs{\frac{N_a(\rest{\alpha}{n})}{n}-P(a)}<\varepsilon.
\end{equation}
In what follows, in the form of Theorem~\ref{EffectiveLLN} below we show that
there exists an \emph{effective} procedure which computes the positive integer $n_0$
in the above statement~\eqref{LLN-classical} for any given rational $\varepsilon>0$,
provided that the underlying finite probability space $P$ is \emph{computable}.

Before showing Theorem~\ref{EffectiveLLN}, we
first
show the following theorem.

\begin{theorem}
\label{EffectiveLLN0}
Let $\Omega$ be an alphabet, and let $P\in\PS(\Omega)$.
Suppose that $P$ is computable.
Let $\alpha\in\Omega^\infty$. Suppose that $\alpha$ is an ensemble for $P$.
Let $\varepsilon$ be an arbitrary positive real.
Then
there exists $M\in\N^+$ such that for every $n\ge M$ and every $k\in\N^+$ if $k\ge n^{2+\varepsilon}$ then
for every $a\in\Omega$ it holds that
\begin{equation}\label{EffectiveLLN0-eq}
  \abs{\frac{N_a(\rest{\alpha}{k})}{k}-P(a)}<\frac{1}{n},
\end{equation}
where $N_a(\sigma)$ denotes the number of the occurrences of $a$ in $\sigma$
for every $a\in\Omega$ and $\sigma\in\Omega^*$.
\qed
\end{theorem}

In order to prove Theorem~\ref{EffectiveLLN0}, we need the following two lemmas.

\begin{lemma}\label{EffectiveLLN-lemma-trf}
Let $\epsilon$ be a positive rational.
Consider a function $f\colon\N^+\to\N^+$ defined by
$f(n)=\left\lceil n^\epsilon \right\rceil$.
Then the function $f\colon\N^+\to\N^+$ is a total recursive function.
\end{lemma}

\begin{proof}
Since $\epsilon$ is a positive rational, there are positive integers $p$ and $q$ such that $\epsilon=p/q$.
Then, for each $n\in\N^+$,
we see that
the function value $f(n)$ is the least positive integer $m$ such that $n^p\le m^q$,
since it is  the least positive integer $m$ such that $n^\epsilon\le m$.
Thus, the function $f\colon\N^+\to\N^+$ is a total recursive function.
\end{proof}

\begin{lemma}\label{EffectiveLLN-lemma}
Let $\epsilon$ be a positive rational, and let $L\in\N^+$.
Then there exists a total recursive function $g\colon\N^+\to\N^+$ such that
for every $m\in\N^+$ it holds that $g(m)\ge L$ and
\begin{equation}\label{EffectiveLLN-lemma-Result}
  \sum_{n=g(m)}^\infty\sum_{k=f(n)}^\infty \exp(-k/n^2)<2^{-m-1},
\end{equation}
where $f$ denotes a function  $f\colon\N^+\to\N^+$ defined by
$f(n)=\left\lceil n^{2+\epsilon} \right\rceil$.
\end{lemma}

\begin{proof}
First,
for each $n\in\N^+$,
using the mean value theorem, we have that
\begin{equation}\label{EffectiveLLN-lemma-MVT1}
  1-\exp(-1/n^2)>\exp(-1/n^2)\frac{1}{n^2}\ge \exp(-1)\frac{1}{n^2}.
\end{equation}
Thus, for each $n\in\N^+$, we see that
\begin{equation}\label{EffectiveLLN-lemma-ineq01}
  \sum_{k=f(n)}^\infty \exp(-k/n^2)
  \le \sum_{l=0}^\infty\exp(-n^{\epsilon}) \exp(-l/n^2)
  =\frac{\exp(-n^{\epsilon})}{1-\exp(-1/n^2)}
  <e n^2\exp(-n^{\epsilon}),
\end{equation}
where the last inequality follows from the inequality~\eqref{EffectiveLLN-lemma-MVT1}.

In what follows, we use the incomplete gamma function $\Gamma(a,x)$ defined by
\begin{equation*}%
  \Gamma(a,x):=\int_{x}^\infty t^{a-1} e^{-t}dt,
\end{equation*}
where $a$ and $x$ are arbitrary reals
satisfying that
$a>0$ and $x\ge 0$.
It is easy to see that
the improper integral in the definition above
exists
certainly
for such reals $a$ and $x$.
Then,
using
the method of integration by parts, we have that
\begin{equation*}%
  \Gamma(a+1,x)=a\Gamma(a,x)+x^a e^{-x}
\end{equation*}
for every reals $a>0$ and $x>0$.
Thus, it follows that
\begin{equation}\label{Gmx=m-1sm-afxke-xk}
  \Gamma(n,x)=(n-1)!\sum_{k=0}^{n-1}\frac{x^k e^{-x}}{k!}
\end{equation}
for every $n\in\N^+$ and
every real
$x>0$.
See for instance Jameson~\cite{Jam16}
for the detail of the properties of the incomplete gamma function~$\Gamma(a,x)$.

We define a function $h\colon\N^+\to\N^+$ by $h(n):=\left\lceil n^{\epsilon} \right\rceil$.
Then it follows from Lemma~\ref{EffectiveLLN-lemma-trf} that
the function $h\colon\N^+\to\N^+$ is a total recursive function,
and obviously
\begin{equation}\label{EffectiveLLN-lemma-hn-1<eo2lehn}
  h(n)-1< n^\epsilon\le h(n)
\end{equation}
holds for every $n\in\N^+$.

Now, let $N\in\N^+$ with
$N^{\epsilon}\ge 2/\epsilon$.
Applying the method of integration of substitution to
the improper integral $\Gamma(3/\epsilon,N^{\epsilon})$ with $t=u^{\epsilon}$, we see that
the improper integral
\[
  \int_N^\infty u^2 \exp(-u^{\epsilon}) du
\]
exists and equals
$(1/\epsilon)\Gamma(3/\epsilon,N^{\epsilon})$.
Thus, since $u^2\exp(-u^{\epsilon})$ is a strictly decreasing function of $u$ for all
positive real
$u$ with $u^{\epsilon}\ge 2/\epsilon$,
we have that
\begin{equation}\label{Lemma-ML-Gamma-eq01}
  \sum_{n=N+1}^\infty n^2\exp(-n^{\epsilon})\le\int_N^\infty u^2 \exp(-u^{\epsilon}) du
  =\frac{1}{\epsilon}\Gamma(3/\epsilon,N^{\epsilon}).
\end{equation}
We
hereafter
denote $\left\lceil 3/\epsilon\right\rceil$ by $D$.
Note that $\Gamma(a,N^{\epsilon})$ is a non-decreasing function of $a>0$, since $N^{\epsilon}\ge 1$.
Thus,
we see that
\begin{equation}\label{Lemma-ML-Gamma-eq02}
\begin{split}
  \Gamma(3/\epsilon,N^{\epsilon})&\le\Gamma(D,N^{\epsilon})
  =(D-1)!\sum_{l=0}^{D-1}\frac{(N^{\epsilon})^l}{l!}\exp(-N^{\epsilon}) \\
  &\le (D-1)!\sum_{l=0}^{D-1}(N^{\epsilon})^{D-1}\exp(-N^{\epsilon})
  < D! \frac{(N^{\epsilon})^{D-1}}{2^{N^{\epsilon}}} \\
  &< 2 D! \frac{h(N)^{D}}{2^{h(N)}},
\end{split}
\end{equation}
where
the first equality follows from \eqref{Gmx=m-1sm-afxke-xk},
the third inequality follows from the fact that $e>2$,
and the last inequality follows from \eqref{EffectiveLLN-lemma-hn-1<eo2lehn}.
Thus, it follows from \eqref{EffectiveLLN-lemma-ineq01}, \eqref{Lemma-ML-Gamma-eq01},
and \eqref{Lemma-ML-Gamma-eq02} that
\begin{equation*}%
  \sum_{n=N+1}^\infty\sum_{k=f(n)}^\infty \exp(-k/n^2)
  <\frac{2e}{\epsilon} D! \frac{h(N)^{D}}{2^{h(N)}}.
\end{equation*}
Hence, since $N$ is an arbitrary positive integer such that $N^{\epsilon}\ge 2/\epsilon$,
we have that
\begin{equation}\label{Lemma-ML-Gamma-eq04}
  \sum_{n=N+1}^\infty\sum_{k=f(n)}^\infty \exp(-k/n^2)<C\frac{h(N)^{D}}{2^{h(N)}}
\end{equation}
for every $N\in\N^+$ with $N^{\epsilon}\ge 2/\epsilon$,
where $C$ denotes the positive integer $\left\lceil 2 e D!/\epsilon\right\rceil$.

Then, based on the above considerations, we can show that
there exists a total recursive function $g\colon\N^+\to\N^+$ such that
both $g(m)\ge L$ and the inequality~\eqref{EffectiveLLN-lemma-Result} hold for all $m\in\N^+$.
Actually, this $g$ can be computed by the following procedure:

Given $m\in\N^+$, one finds $N\in\N^+$ such that $N\ge L$, $N^{\epsilon}\ge 2/\epsilon$,
and
\begin{equation}\label{Lemma-ML-Gamma-eq05}
  \frac{h(N)^{D}}{2^{h(N)}}\le\frac{2^{-m-1}}{C}.
\end{equation}
This is possible, since
the function $h\colon\N^+\to\N^+$ is a total recursive function such that
$\lim_{N\to\infty}h(N)=\infty$, and moreover
$\lim_{x\to\infty}x^{D}/2^x=0$
holds.
One then outputs $N+1$.
Note
here
that
\[
  \sum_{n=N+1}^\infty\sum_{k=f(n)}^\infty \exp(-k/n^2)<2^{-m-1}
\]
certainly
holds for this $N+1$,
due to the inequalities~\eqref{Lemma-ML-Gamma-eq04} and \eqref{Lemma-ML-Gamma-eq05}.
\end{proof}

Theorem~\ref{EffectiveLLN0} is then proved as follows.

\begin{proof}[Proof of Theorem~\ref{EffectiveLLN0}]
Let $\Omega$ be an alphabet, and let $P\in\PS(\Omega)$.
Suppose that $P$ is computable. Let $\alpha\in\Omega^\infty$. Suppose that $\alpha$ is an ensemble for $P$.
Let $\varepsilon$ be an arbitrary positive real.
Let $a$ be an arbitrary element of $\Omega$.
We first show the following statement:
There exists $M\in\N^+$ such that for every $n\ge M$ and every $k\in\N^+$ if $k\ge n^{2+\varepsilon}$ then
\begin{equation}\label{EffectiveLLN-eq01}
  \abs{\frac{N_a(\rest{\alpha}{k})}{k}-P(a)}<\frac{1}{n}.
\end{equation}

In the case of $P(a)=0$, the statement~\eqref{EffectiveLLN-eq01} follows immediately from
the result~(i)
of Theorem~\ref{zero_probability}.
In the case of $P(a)=1$,
the statement~\eqref{EffectiveLLN-eq01} follows immediately from Theorem~\ref{one_probability}.
Thus we assume that $0<P(a)<1$, in what follows.

We define
$Q\in\PS(\{0,1\})$ by the condition that
$Q(1):=P(a)$ and $Q(0):=1-P(a)$.
Then
$0<Q(1)<1$.
Let $\beta$ be the infinite binary sequence obtained from $\alpha$
by replacing all $a$ by $1$ and
all other elements of $\Omega$
by $0$ in $\alpha$.
Then, by using Theorem~\ref{contraction}
repeatedly, it is easy to show that
$\beta$ is Martin-L\"of $Q$-random and
$N_1(\rest{\beta}{k})=N_a(\rest{\alpha}{k})$ for every $k\in\N^+$.
Note also that $Q(1)$ is a computable real, since $P$ is a computable finite probability space. 

We choose any specific $n_0\in\N^+$ such that
\[
  \frac{2}{n_0}\le \min\{Q(0),Q(1)\}.
\]
This is possible since $0<Q(1)<1$.
Then, it follows from Theorem~\ref{Chernoff_bound} that
\begin{equation}\label{EffectiveLLN-rCb}
  \Bm{Q}{\osg{\{\sigma\in\{0,1\}^k\mid\abs{N_1(\sigma)/k-Q(1)}>2/n\}}}<2\exp(-k/n^2)
\end{equation}
for every $n\ge n_0$ and every $k\in\N^+$.
We choose any specific positive rational $\delta$ such that $\varepsilon\ge 2\delta$, and
define a function $f\colon\N^+\to\N^+$ by $f(n):=\left\lceil n^{2+\delta} \right\rceil$.
Then it follows from Lemma~\ref{EffectiveLLN-lemma} that
there exists a total recursive function $g\colon\N^+\to\N^+$ such that
for every $m\in\N^+$ it holds that $g(m)\ge n_0$ and
\begin{equation}\label{EffectiveLLN0-lemma-Result-referred}
  \sum_{n=g(m)}^\infty\sum_{k=f(n)}^\infty \exp(-k/n^2)<2^{-m-1}.
\end{equation}
For each $m\in\N^+$,
we define a subset $S(m)$ of $\X$ by
\begin{equation}\label{EffectiveLLN0-def-S(N)}
  S(m):=\bigcup_{n=g(m)}^\infty\bigcup_{k=f(n)}^\infty\{\sigma\in\{0,1\}^k\mid\abs{N_1(\sigma)/k-Q(1)}>2/n\}.
\end{equation}
Then, for each $m\in\N^+$, we see that
\begin{equation}\label{EffectiveLLN-ineq-lQ<ss2xk2n2<2N6}
\begin{split}
  \Bm{Q}{\osg{S(m)}}
  &\le\sum_{n=g(m)}^\infty\sum_{k=f(n)}^\infty
         \Bm{Q}{\osg{\{\sigma\in\{0,1\}^k\mid\abs{N_1(\sigma)/k-Q(1)}>2/n\}}} \\
  &<\sum_{n=g(m)}^\infty\sum_{k=f(n)}^\infty 2\exp(-k/n^2) \\
  &<2^{-m},
\end{split}
\end{equation}
where the second inequality follows from the inequality~\eqref{EffectiveLLN-rCb} and the fact that $g(m)\ge n_0$,
and the last inequality follows from the inequality~\eqref{EffectiveLLN0-lemma-Result-referred}.

Now, we denote
the set
$\{(m,\sigma)\in\N^+\times\X\mid\sigma\in S(m)\}$
by $\mathcal{T}$.
On the one hand,
note
from Lemma~\ref{EffectiveLLN-lemma-trf} that
the function $f\colon\N^+\to\N^+$ is a total recursive function.
Moreover, note that
\[
  S(m)=\bigcup_{n=g(m)}^\infty\bigcup_{k=f(n)}^\infty
  \{\sigma\in\{0,1\}^k\mid N_1(\sigma)/k+2/n<Q(1)\;\text{ or }\;Q(1)<N_1(\sigma)/k-2/n\}
\]
for every
$m\in\N^+$.
Thus, since $Q(1)$ is both left-computable and right-computable, it is easy to see that $\mathcal{T}$ is an r.e.~set.
On the other hand,
\eqref{EffectiveLLN-ineq-lQ<ss2xk2n2<2N6} implies that
$\Bm{Q}{\osg{\mathcal{T}_m}}<2^{-m}$ for every $m\in\N^+$,
where $\mathcal{T}_m$ denotes the set
$\left\{\,
  \sigma\mid (m,\sigma)\in\mathcal{T}
\,\right\}$.
Hence,
since $\beta$ is Martin-L\"of $Q$-random,
it follows from Theorem~\ref{ML_P-randomness_eliminated-prefix-freeness} that
there exists $L\in\N^+$ such that $\beta\notin\osg{\mathcal{T}_L}$.
Thus,
since $\mathcal{T}_L=S(L)$,
from the definition of $S(L)$
we have the following:
For every $n\ge g(L)$ and every $k\ge f(n)$ it holds that
\begin{equation}\label{EffectiveLLN0-eq-absNbok-Q1le2on}
  \abs{\frac{N_1(\rest{\beta}{k})}{k}-Q(1)}\le\frac{2}{n}.
\end{equation}
Then we choose any specific $M\in\N^+$ such that $M^\delta\ge 4^{2+\delta}$ and $4M\ge g(L)$.
Note that for every $n\in\N^+$ if $n\ge M$ then
$\left\lceil n^{2+\varepsilon}\right\rceil\ge f(4n)$.
It follows from \eqref{EffectiveLLN0-eq-absNbok-Q1le2on} that
for every $n\ge M$ and every $k\ge n^{2+\varepsilon}$ it holds that
\[
  \abs{\frac{N_1(\rest{\beta}{k})}{k}-Q(1)}<\frac{1}{n}.
\]
Thus, since $Q(1)=P(a)$ and $N_1(\rest{\beta}{k})=N_a(\rest{\alpha}{k})$ for every $k\in\N^+$,
the statement~\eqref{EffectiveLLN-eq01} holds in this case of $0<P(a)<1$, as desired.
Hence, since $a$ is an arbitrary element of $\Omega$,
the statement~\eqref{EffectiveLLN-eq01} holds for every $a\in\Omega$.

Now, for each $a\in\Omega$,
let $M_a$ be the positive integer $M$ whose existence is guaranteed in the statement~\eqref{EffectiveLLN-eq01}.
Note that $\max\{M_a\mid a\in\Omega\}$ exists, since $\Omega$ is a
non-empty
finite set. 
We denote
by $\overline{M}$
this $\max\{M_a\mid a\in\Omega\}$.
It is then easy to check that
for every $n\ge \overline{M}$ and every $k\in\N^+$ if $k\ge n^{2+\varepsilon}$ then for every $a\in\Omega$
the inequality~\eqref{EffectiveLLN0-eq} holds.
This completes the proof.
\end{proof}

Theorem~\ref{EffectiveLLN0}
leads to
the following corollary,
in particular, for the notion of Martin-L\"of randomness.

\begin{corollary}
Let $\alpha$ be a Martin-L\"of random infinite binary sequence.
Let $\varepsilon$ be an arbitrary positive real.
Then there exists $M\in\N^+$ such that for every $n\ge M$ and every $k\in\N^+$ if $k\ge n^{2+\varepsilon}$ then
for every $a\in\{0,1\}$ it holds that
\begin{equation*}
  \abs{\frac{N_a(\rest{\alpha}{k})}{k}-\frac{1}{2}}<\frac{1}{n},
\end{equation*}
where $N_a(\sigma)$ denotes the number of the occurrences of $a$ in $\sigma$
for every $a\in\{0,1\}$ and $\sigma\in\X$.
\end{corollary}

\begin{proof}
Let
$U$
be a finite probability space on $\{0,1\}$ such that $U(0)=U(1)=1/2$.
Then the infinite binary sequence $\alpha$ is an ensemble for $U$. 
Thus, since $1/2$ is a computable real, the result follows from Theorem~\ref{EffectiveLLN0}.
\end{proof}

\begin{theorem}
[Effectivization of the law of large numbers]
\label{EffectiveLLN}
Let $\Omega$ be an alphabet, and let $P\in\PS(\Omega)$.
Suppose that $P$ is computable.
Let $\alpha\in\Omega^\infty$. Suppose that $\alpha$ is an ensemble for $P$.
Let $\epsilon$ be an arbitrary positive rational. 
Then there exists a
total recursive
function $f\colon\N^+\to\N^+$ which satisfies the following conditions~(i) and (ii):
\begin{enumerate}
  \item For every $n\in\N^+$ and every $k\in\N^+$
    if $k\ge f(n)$ then
    for every $a\in\Omega$ it holds that
    \begin{equation*}%
      \abs{\frac{N_a(\rest{\alpha}{k})}{k}-P(a)}<\frac{1}{n},
    \end{equation*}
    where $N_a(\sigma)$ denotes the number of the occurrences of $a$ in $\sigma$
    for every $a\in\Omega$ and $\sigma\in\Omega^*$.
\item The
total recursive
function $f$ has the
following
form:%
\footnote{This total recursive function $f\colon\N^+\to\N^+$ is
the restriction to $\N^+$ of some primitive recursive function on $\N$.
Actually, every function $g\colon\N\to\N$
which is an extension of $f$ to $\N$
is a primitive recursive function.}
\[f(n)=\left\lceil n^{2+\epsilon} \right\rceil\] for all but finitely many $n\in\N^+$.
\end{enumerate}
\end{theorem}

\begin{proof}
We define a function $g\colon\N^+\to\N^+$ by $g(n):=\left\lceil n^{2+\epsilon} \right\rceil$.
Then, on the one hand, it follows from Lemma~\ref{EffectiveLLN-lemma-trf} that
the function $g\colon\N^+\to\N^+$ is a total recursive function.
On the other hand, it follows from Theorem~\ref{EffectiveLLN0} that
there exists $M\in\N^+$ such that for every $n\ge M$ and every $k\in\N^+$ if $k\ge g(n)$ then
for every $a\in\Omega$ it holds that
\begin{equation}\label{EffectiveLLNpr-eq}
  \abs{\frac{N_a(\rest{\alpha}{k})}{k}-P(a)}<\frac{1}{n}.
\end{equation}
Then we
define a function $f\colon\N^+\to\N^+$ by the condition that
$f(n):=g(n)$ if $n\ge M$ and $f(n):=g(M)$ otherwise.
Since $g\colon\N^+\to\N^+$ is a total recursive function,
it follows that
$f\colon\N^+\to\N^+$ is also a
total recursive
function.
Then, based on \eqref{EffectiveLLNpr-eq}, it is
easy to check that
the conditions~(i) and (ii) of Theorem~\ref{EffectiveLLN} hold for this $f$.
This completes the proof.
\end{proof}

\section{Applications}
\label{Applications}

In this section
we make applications of our framework to the general areas of science and technology
in order to demonstrate the wide applicability of our framework to them.
We here choose \emph{information theory}, \emph{cryptography}, and
\emph{the simulation of a biased coin using fair coins}
as examples of the fields for the applications.
Furthermore,
we mention an application of our framework to quantum mechanics, which
was
developed in a series of
works~\cite{T14CCR,T15Kokyuroku,T15WiNF-Tadaki_rule,T16QIT35,T18arXiv}.

\subsection{Application to information theory}

In this subsection, we make
an
application of our framework to information theory~\cite{Sha48}.
\emph{Instantaneous codes} play
a basic
role in the \emph{noiseless source coding problem} in information theory, as described
in what follows
in terms of our framework.
See for instance Ash~\cite{A90} or Cover and Thomas~\cite{CT06}
for the detail of the noiseless source coding
based on
instantaneous codes
\emph{in the terminology of the conventional probability theory}.

First we
describe the basic notion and definitions in the noiseless source coding problem
based on instantaneous codes, in terms of our framework.
Let $\Omega$ be an alphabet, as in the preceding sections.
An \emph{instantaneous code} $C$ for $\Omega$ is an injective mapping from $\Omega$ to $\{0,1\}^+$
such that $C(\Omega):=\{C(a)\mid a\in\Omega\}$ is a prefix-free set.
Let $P\in\PS(\Omega)$ be a finite probability space.
In
the context of
the source coding problem,
an element of the sample space $\Omega$
of $P$
is called a \emph{source symbol}, and
the finite probability space $P$
itself
is called an \emph{information source}
which emits a source symbol in $\Omega$.
The \emph{average codeword length} $L_P(C)$ of
an instantaneous code $C$ for $\Omega$
in
an information source
$P\in\PS(\Omega)$
is defined by
\begin{equation*}
  L_P(C):=\sum_{a\in\Omega}P(a)\abs{C(a)}.
\end{equation*}
Then the objective of the noiseless source coding problem based on instantaneous codes is
to design an instantaneous code $C$ for $\Omega$ which minimizes 
the average codeword length $L_P(C)$ of $C$
for a given information source $P\in\PS(\Omega)$.

In what follows,
we
provide
\emph{in terms of ensembles}
a scenario which
motivates
this objective of the noiseless source coding problem
and which clarifies the
operational
meaning of the average codeword length.
For that purpose, we introduce
some notation:
Let $C$ be an instantaneous code for $\Omega$.
For any infinite sequence $\alpha$ over $\Omega$, we use $C(\alpha)$ to denote an infinite binary sequence
\[
  C(\alpha(1))C(\alpha(2))C(\alpha(3))\dotsc\dotsc.
\]
Similarly, for any $\tau\in\Omega^+$, we use $C(\tau)$ to denote
the concatenation $C(\tau(1))C(\tau(2))\dots C(\tau(\abs{\tau}))$.
Thus,
$C(\rest{\alpha}{n})$ denotes $C(\alpha(1))C(\alpha(2))\dots C(\alpha(n))$
for every $\alpha\in\Omega^\infty$ and $n\in\N^+$,
in particular.
Since $C$ is an instantaneous code for $\Omega$,
it is easy to
see
that for every $\alpha,\beta\in\Omega^\infty$ if $C(\alpha)=C(\beta)$ then $\alpha=\beta$.

Now, let $\Omega$ be an alphabet, and let $P\in\PS(\Omega)$.
Let us
consider an infinite sequence $\alpha$ of source symbols in $\Omega$
which is being generated by infinitely repeated
emission
from the information source $P$.
Namely, we consider an infinite sequence $\alpha\in\Omega^\infty$ of outcomes
which is being generated by infinitely repeated trials described
by the finite probability space $P$ on $\Omega$.
Then, according to
Thesis~\ref{thesis}, $\alpha$ is an ensemble for $P$.
Moreover, we consider a \emph{noiseless binary communication channel} which can reliably transmits
bits
from one end,
called the \emph{input} of the channel, to the other end, called the \emph{output} of the channel.
Let $C$ be an instantaneous code for $\Omega$.
We make use of
the instantaneous code $C$
in order
to transmit the ensemble $\alpha$ from the input to the output
of the channel
in the following manner:
Whenever a source symbol $a$ is emitted from the information source $P$,
an \emph{encoding equipment} at the input of the channel
converts
it into $C(a)$
and then transmits $C(a)$ through the channel from the input to the output.
In this way,
the channel is transmitting
each bit of the infinite binary sequence $C(\alpha)$ one by one in a sequential order from the top of $C(\alpha)$,
toward the output of the channel.
At the output of the channel, an \emph{decoding equipment}
performs
the following \emph{decoding procedure} forever:
The decoding equipment is continuously monitoring bits
which are being outputted from the channel and
which
are
thus
accumulating at the output.
Whenever the accumulation of the bits becomes a finite binary string $\sigma$ which is an element of $C(\Omega)$,
the decoding equipment converts $\sigma$
into a source symbol $a$
satisfying that $C(a)=\sigma$,
and then repeats this procedure. 
The
decoding procedure
can recover
the original
ensemble
$\alpha$
at the output of the channel.
Namely,
through
the
decoding procedure, the
original
infinite sequence $\alpha$ is being generated at the output of the channel.
This
result
can be confirmed
by using the property of $C$ as an instantaneous code for $\Omega$.

From a practical point of view, 
it is important to consider the \emph{efficiency} of the transmission of the ensemble $\alpha$ through the channel
in the above setting.
As a natural measure of the efficiency
we consider
the number of bits per source symbol sent through the channel
during the transmission of the ensemble $\alpha$, and
we
try to minimize it.
Note that in the transmission of the ensemble $\alpha$
each of the
finite
binary strings $C(\alpha(1)),C(\alpha(2)),C(\alpha(3)),\dotsc\dotsc$ passes through the channel
one by one in a sequential order,
and moreover
the total number of bits used
to transmit the source symbols $\alpha(1),\alpha(2),\dots,\alpha(n)$
through the channel equals $\sum_{k=1}^n\abs{C(\alpha(k))}=\abs{C(\rest{\alpha}{n})}$.
Thus,
we adopt the quantity
\begin{equation}\label{AppIT:me}
  \lim_{n\to\infty}\frac{\abs{C(\rest{\alpha}{n})}}{n}
\end{equation}
as
the
measure of the efficiency of the transmission
which represents
the number of bits per source symbol sent through the channel
during the transmission of the ensemble $\alpha$.
Theorem~\ref{Operational-Meaning-of-ACL} below states that
this measure~\eqref{AppIT:me}
of the efficiency precisely equals
the average codeword length $L_P(C)$ of the instantaneous code $C$
for the information source $P$.
Thus, via the quantity~\eqref{AppIT:me},
the average codeword length $L_P(C)$ has an operational meaning as
the number of bits per source symbol sent through the channel
during the transmission of an ensemble $\alpha$
being
emitted from the information source $P$.
This scenario also explains the reason why
the objective of the noiseless source coding problem based on instantaneous codes is
to design an instantaneous code $C$ for $\Omega$ which minimizes 
the average codeword length $L_P(C)$.
The reason is that
the minimization of the number of bits per source symbol sent through the channel
is essential to
an efficient transmission of the ensemble $\alpha$ through the channel
in this scenario.

\begin{theorem}\label{Operational-Meaning-of-ACL}
Let $\Omega$ be an alphabet, and let $P\in\PS(\Omega)$.
Let $C$ be an instantaneous code $C$ for $\Omega$.
For every $\alpha\in\Omega^\infty$, if $\alpha$ is an ensemble for $P$ then
\begin{equation*}
  \lim_{n\to\infty}\frac{\abs{C(\rest{\alpha}{n})}}{n}=L_P(C).
\end{equation*}
\end{theorem}

\begin{proof}
For each $n\in\N^+$, we see that
\begin{equation}\label{OM-ACL:eq1-0}
\frac{\abs{C(\rest{\alpha}{n})}}{n}=\frac{1}{n}\sum_{k=1}^n\abs{C(\alpha(k))}
=\frac{1}{n}\sum_{a\in\Omega}N_a(\rest{\alpha}{n})\abs{C(a)}
=\sum_{a\in\Omega}\frac{N_a(\rest{\alpha}{n})}{n}\abs{C(a)},
\end{equation}
where $N_a(\rest{\alpha}{n})$ denotes the number of the occurrences of $a$ in $\rest{\alpha}{n}$
for every $a\in\Omega$,
as in Theorem~\ref{LLN}.
It follows from Theorem~\ref{LLN} that,
on letting $n\to\infty$,
the limit value of the most right-hand side of \eqref{OM-ACL:eq1-0} equals
\begin{equation*}
  \sum_{a\in\Omega}P(a)\abs{C(a)}=L_P(C),
\end{equation*}
as desired.
This completes the proof.
\end{proof}

We can
show that $L_P(C)\ge H(P)$ for every instantaneous code $C$ for $\Omega$
and every finite probability space $P\in\PS(\Omega)$,
where $H(P)$ is the Shannon entropy of $P$ defined by \eqref{def:Shannon-entropy}.
This is a basic result of information theory.
Hence,
\emph{the Shannon entropy gives the limit of the efficiency of the transmission of the ensemble $\alpha$, i.e,
a lower bound for
the number of bits per source symbol sent through the channel
during the transmission of the ensemble $\alpha$},
in the scenario above for the noiseless source coding problem based on instantaneous codes.
Note that
in information theory this situation
is phrased as
``the Shannon entropy gives the data compression limit
for the noiseless source coding problem based on instantaneous codes.''
For these reasons, 
it is important to consider the notion of absolutely optimality of an instantaneous code,
where we say that
an instantaneous code $C$ for $\Omega$ is \emph{absolutely optimal}
for a finite probability space $P\in\PS(\Omega)$ if $L_P(C)=H(P)$.
Then we can show
Theorem~\ref{OSDCL} below.
Recall from Theorem~\ref{equivMLR} that
Martin-L\"of random sequences are precisely the infinite binary sequences which
\emph{cannot be
compressible
any more}.
Thus, Theorem~\ref{OSDCL} rephrases in a sharp manner in terms of our framework
the following basic result of the noiseless source coding problem,
stated
in the terminology of information theory:
``The Shannon entropy gives the \emph{data compression limit}.''

\begin{theorem}\label{OSDCL}
Let $\Omega$ be an alphabet, and let $P\in\PS(\Omega)$. Let
$C$ be an instantaneous code for $\Omega$.
Suppose that $\alpha$ is an ensemble for $P$.
Then the following conditions~(i) and (ii) are equivalent to each other:
\begin{enumerate}
 \item
  The instantaneous code $C$ is absolutely optimal for the finite probability space $P$.
 \item $C(\alpha)$ is Martin-L\"of random.%
\end{enumerate}
\end{theorem}

\begin{proof}%
We note that
the instantaneous code $C$ for $\Omega$ is absolutely optimal for the finite probability space $P$
if and only if $P(a)=2^{-\abs{C(a)}}$ for every $a\in\Omega$.
This is a basic result of information theory.

First, we show the implication (i) $\Rightarrow$ (ii).
For that purpose, suppose that $C$ is absolutely optimal for $P$.
Then
$P(a)=2^{-\abs{C(a)}}$
for every $a\in\Omega$.
It follows that
\begin{equation}\label{P=2Ctau}
  P(\tau)=2^{-\abs{C(\tau)}}
\end{equation}
for every $\tau\in\Omega^+$.

Now, let us
assume contrarily that
$C(\alpha)$
is not Martin-L\"of random.
Then there exists a Martin-L\"of test $\mathcal{S}\subset\N^+\times\X$ such that
\begin{equation}\label{CodedCalphainosgmathcalSn}
  C(\alpha)\in\osg{\mathcal{S}_n}
\end{equation}
for every $n\in\N^+$.
For each $\sigma\in\{0,1\}^+$, let $F(\sigma)$ be
the set of all minimal strings $\tau\in\Omega^+$ such that $\sigma$ is a prefix of $C(\tau)$, i.e.,
the set of all $\tau\in\Omega^+$ such that $\sigma$ is a prefix of $C(\tau)$ but
$\sigma$ is not a prefix of $C(\upsilon)$ for any proper prefix $\upsilon$ of $\tau$ with $\upsilon\neq\lambda$.
Then
we see that
\begin{equation}\label{ossupsetbctFsoCt}
  \bigcup_{\tau\in F(\sigma)}\osg{C(\tau)}\subset\osg{\sigma}
\end{equation}
for every $\sigma\in\{0,1\}^+$.
Moreover,
it is easy to see that $F(\sigma)$ is a finite prefix-free subset of $\Omega^*$
for every $\sigma\in\{0,1\}^+$.
Note here that,
given a finite binary string $\sigma\in\{0,1\}^+$, one can effectively calculate the finite set $F(\sigma)$.
Thus, since $C$ is an instantaneous code for $\Omega$,
for each $\sigma\in\{0,1\}^+$, we see that
\begin{equation}\label{App-IT:eq-oCtcapCtp=es}
  \osg{C(\tau)}\cap\osg{C(\tau')}=\emptyset
\end{equation}
for every $\tau,\tau'\in F(\sigma)$ with $\tau\neq\tau'$.
Hence,
for each $\sigma\in\{0,1\}^+$, we have that
\begin{align}\label{BmPoFs=mLebesgueos-IT}
  \Bm{P}{\osg{F(\sigma)}}
  &\le\sum_{\tau\in F(\sigma)}\Bm{P}{\osg{\tau}}
  =\sum_{\tau\in F(\sigma)}P(\tau)
  =\sum_{\tau\in F(\sigma)}2^{-\abs{C(\tau)}}
  =\sum_{\tau\in F(\sigma)}\mathcal{L}(\osg{C(\tau)})
  \le\mathcal{L}(\osg{\sigma}),
\end{align}
where
the second equality follows from \eqref{P=2Ctau},
and the last inequality follows from \eqref{App-IT:eq-oCtcapCtp=es} and \eqref{ossupsetbctFsoCt}.
Recall also that $\mathcal{L}$ denotes Lebesgue measure on $\XI$.
We then define $\mathcal{T}$ to be a subset of $\N^+\times \Omega^*$ such that
$\mathcal{T}_n=\bigcup_{\sigma\in\mathcal{S}_n} F(\sigma)$ for every $n\in\N^+$.
Note here that, for each $n\in\N^+$,
$\lambda\notin\mathcal{S}_n$ since $\mathcal{L}(\osg{\mathcal{S}_n})<2^{-n}<1$.
Then, since $\mathcal{S}_n$ is a prefix-free subset of $\X$ for every $n\in\N^+$
and $F(\sigma)$ is a prefix-free subset of $\Omega^*$ for every $\sigma\in\{0,1\}^+$,
it is easy to see that $\mathcal{T}_n$ is a prefix-free subset of $\Omega^*$ for every $n\in\N^+$.
For each $n\in\N^+$, we also see that
\[
  \Bm{P}{\osg{\mathcal{T}_n}}
  \le\sum_{\sigma\in\mathcal{S}_n}\Bm{P}{\osg{F(\sigma)}}
  \le\sum_{\sigma\in\mathcal{S}_n}\mathcal{L}(\osg{\sigma})
  =\mathcal{L}(\osg{\mathcal{S}_n})<2^{-n},
\]
where the second inequality follows from \eqref{BmPoFs=mLebesgueos-IT} and
the equality follows from the prefix-freeness of $\mathcal{S}_n$.
Moreover, since $\mathcal{S}$ is r.e.,
we see that $\mathcal{T}$ is also r.e.
Thus, $\mathcal{T}$ is a Martin-L\"of $P$-test.

On the other hand,
note that, for every $n\in\N^+$, if
$C(\alpha)\in\osg{\mathcal{S}_n}$
then $\alpha\in\osg{\mathcal{T}_n}$.
Thus, it follows from \eqref{CodedCalphainosgmathcalSn}
that $\alpha\in\osg{\mathcal{T}_n}$ for every $n\in\N^+$.
Hence,
$\alpha$ is not Martin-L\"of $P$-random.
However, this contradicts the assumption of the theorem.
Thus,
we have that
$C(\alpha)$
is Martin-L\"of random.

Next, we show the implication (ii) $\Rightarrow$ (i).
We choose any specific $b_0\notin\Omega$ and then denote the alphabet $\Omega\cup\{b_0\}$ by $\Phi$.
Since $C(\Omega)$ is
a prefix-free subset of $\X$,
the Kraft inequality
\[
  \sum_{a\in\Omega}2^{-\abs{C(a)}}\le 1
\]
holds.
Hence, we can define a finite probability space $Q\in\PS(\Phi)$ with the property that
$Q(x):=2^{-\abs{C(x)}}$
if $x\in\Omega$ and
\[
  Q(x):=1-\sum_{a\in\Omega}2^{-\abs{C(a)}}
\]
otherwise.
Note that since $\alpha$ is an infinite sequence over $\Omega$ and $\Omega\subset\Phi$,
it is also an infinite sequence over $\Phi$.

We show that if $\alpha$ is not Martin-L\"of $Q$-random then
$C(\alpha)$
is not Martin-L\"of random.
Thus, assume that $\alpha$ is not Martin-L\"of $Q$-random.
Then there exists a Martin-L\"of $Q$-test $\mathcal{S}\subset\N^+\times\Phi^*$ such that
\begin{equation}\label{alphainosgmathcalSn-IT}
  \alpha\in\osg{\mathcal{S}_n}
\end{equation}
for every $n\in\N^+$.
Thus, since $Q(a)=2^{-\abs{C(a)}}$ for every $a\in\Omega$, we have that
\begin{equation}\label{Qos=Qs=2-aCs=mLoCs-IT}
  \Bm{Q}{\osg{\sigma}}=Q(\sigma)=2^{-\abs{C(\sigma)}}=\mathcal{L}(\osg{C(\sigma)})
\end{equation}
for each $\sigma\in\Omega^+$.
We define $\mathcal{T}$ to be a subset of $\N^+\times \X$ such that
$\mathcal{T}_n=\{C(\sigma)\mid\sigma\in\mathcal{S}_n\cap\Omega^+\}$
for every $n\in\N^+$.
Then, since
$C$ is an instantaneous code for $\Omega$ and
$\mathcal{S}_n$ is a prefix-free subset of $\Phi^*$ for every $n\in\N^+$,
it is easy to see that $\mathcal{T}_n$ is a prefix-free subset of $\X$ for every $n\in\N^+$.
For each $n\in\N^+$, we also see that
\[
  \mathcal{L}(\osg{\mathcal{T}_n})
  \le\sum_{\sigma\in\mathcal{S}_n\cap\Omega^+}\mathcal{L}(\osg{C(\sigma)})
  =\sum_{\sigma\in\mathcal{S}_n\cap\Omega^+}\Bm{Q}{\osg{\sigma}}
  =\Bm{Q}{\osg{\mathcal{S}_n\cap\Omega^+}}
  \le\Bm{Q}{\osg{\mathcal{S}_n}}<2^{-n},
\]
where the first equality follows from \eqref{Qos=Qs=2-aCs=mLoCs-IT} and
the second equality follows from the prefix-freeness of
$\mathcal{S}_n$.
Moreover, since $\mathcal{S}$ is r.e., $\mathcal{T}$ is also r.e.
Thus, $\mathcal{T}$ is Martin-L\"of test.

On the other hand,
we see that, for every $n\in\N^+$,
if $\alpha\in\osg{\mathcal{S}_n}$ then
$C(\alpha)\in\osg{\mathcal{T}_n}$.
Note here that, for each $n\in\N^+$,
$\lambda\notin\mathcal{S}_n$ since $\Bm{Q}{\osg{\mathcal{S}_n}}<2^{-n}<1$.
Thus, it follows from \eqref{alphainosgmathcalSn-IT} that
$C(\alpha)\in\osg{\mathcal{T}_n}$
for every $n\in\N^+$.
Hence,
$C(\alpha)$
is not Martin-L\"of random, as desired.

Now,
recall that $\alpha$ is a Martin-L\"of $P$-random infinite sequence over $\Omega$
by the assumption of the theorem.
We define a finite probability space $P'\in\PS(\Phi)$ by the condition that
$P'(x):=P(x)$ if $x\in\Omega$ and $P'(x):=0$ otherwise.
It is then easy to see that
$\alpha$ is a Martin-L\"of $P'$-random infinite sequence over $\Phi$.
Suppose that the condition~(ii) holds.
Then $\alpha$ is a Martin-L\"of $Q$-random infinite sequence over $\Phi$,
due to the result above.
It follows from Corollary~\ref{uniquness} that $P'=Q$, and therefore $P(a)=2^{-\abs{C(a)}}$
for every $a\in\Omega$. Hence, the condition~(i) holds.
Thus, we have the implication (ii) $\Rightarrow$ (i).
This completes the proof.
\end{proof}

By definition, an absolutely optimal instantaneous code for an information source $P\in\PS(\Omega)$ has
the minimal average codeword length for $P$ among all instantaneous codes for $\Omega$,
and therefore
it follows from Theorem~\ref{Operational-Meaning-of-ACL} that
it provides the most efficient coding
in transmitting
the ensemble $\alpha$ through
the noiseless binary communication channel
in the scenario above for the noiseless source coding problem based on instantaneous codes.
On the other hand,
we again recall from Theorem~\ref{equivMLR} that
Martin-L\"of random sequences are precisely the infinite binary sequences which
\emph{cannot be compressible any more}.
Thus,
the implication (i) $\Rightarrow$ (ii) of Theorem~\ref{OSDCL}
further shows that
an
arbitrary
absolutely optimal instantaneous code
$C$
can already
achieve
the limit of the efficiency of the transmission
of the ensemble $\alpha$
\emph{among all conceivable noiseless source coding schemes},
including but not limited to the noiseless source coding
scheme
based on
instantaneous codes,
in a certain sense,
since $C(\alpha)$
is a Martin-L\"of random infinite binary sequence in this case and therefore it
cannot be compressed any more.

Finally,
we remark that
in both Theorem~\ref{Operational-Meaning-of-ACL} and Theorem~\ref{OSDCL}
the underlying finite probability space $P$ is \emph{quite arbitrary},
and therefore is not required to be computable at all, in particular.

\subsection{Application to cryptography}
\label{subsec:AC}

In this subsection, we make
an
application of our framework to cryptography.
The notion of probability plays a crucial role in modern cryptography.
The security of almost all cryptographic schemes, including
private-key encryption schemes, public-key encryption schemes, message authentication, and digital signatures,
is based on the notion of probability (see Goldreich~\cite{G01,G04}).
In this subsection we present new equivalent characterizations of the notion of \emph{perfect secrecy},
in terms of ensembles.

The notion of perfect secrecy was introduced by Shannon \cite{Sha49},
and plays a basic role in modern cryptography.
Perfect secrecy is the most stringent security notion
among the class of security notions
called \emph{information-theoretic security} or \emph{unconditional security}.
See Katz and Lindell~\cite[Section 2]{KL07} for a modern treatment of the notion of perfect secrecy.
To begin with, we review the notion of \emph{encryption schemes} to which the notion of perfect secrecy is applied,
in terms of our framework.

\begin{definition}[Encryption scheme]\label{ES_simple}
Let $\mathcal{M}$, $\mathcal{K}$, and $\mathcal{C}$ be alphabets.
An \emph{encryption scheme} over a message space $\mathcal{M}$, a key space $\mathcal{K}$,
and a ciphertext space $\mathcal{C}$ is a tuple $\Pi=(P_{\mathrm{key}},\mathsf{Enc},\mathsf{Dec})$ such that
\begin{enumerate}
  \item $P_{\mathrm{key}}\in\PS(\mathcal{K})$,
  \item $\mathsf{Enc}\colon\mathcal{M}\times\mathcal{K}\to\mathcal{C}$,
  \item $\mathsf{Dec}\colon\mathcal{C}\times\mathcal{K}\to\mathcal{M}$, and
  \item $\mathsf{Dec}(\mathsf{Enc}(m,k),k)=m$ for every $m\in\mathcal{M}$ and $k\in\mathcal{K}$.\qed
\end{enumerate}
\end{definition}

Let $\Pi=(P_{\mathrm{key}},\mathsf{Enc},\mathsf{Dec})$ be as in Definition~\ref{ES_simple},
and let $P_{\mathrm{msg}}\in\PS(\mathcal{M})$.
The finite probability space $P_{\mathrm{msg}}$ serves as a ``probability distribution''
over message space $\mathcal{M}$ for the encryption scheme $\Pi$.
We then define random variables $M_{\Pi}\colon\mathcal{M}\times\mathcal{K}\to\mathcal{M}$ and
$C_{\Pi}\colon\mathcal{M}\times\mathcal{K}\to\mathcal{C}$ on $\mathcal{M}\times\mathcal{K}$
by $M_{\Pi}(m,k):=m$ and $C_{\Pi}(m,k):=\mathsf{Enc}(m,k)$, respectively.
The notion of perfect secrecy is then defined as follows,
in terms of our framework.

\begin{definition}[Perfect secrecy, Shannon \cite{Sha49}]\label{PS_simple}
Let $\mathcal{M}$, $\mathcal{K}$, and $\mathcal{C}$ be alphabets.
Let $\Pi=(P_{\mathrm{key}},\mathsf{Enc},\mathsf{Dec})$ be an encryption scheme
over a message space $\mathcal{M}$, a key space $\mathcal{K}$, and a ciphertext space $\mathcal{C}$.
The encryption scheme $\Pi$ is \emph{perfectly secret} if
for every $P_{\mathrm{msg}}\in\PS(\mathcal{M})$
it holds that the random variables $M_{\Pi}$ and $C_{\Pi}$ are independent
on $P_{\mathrm{msg}}\times P_{\mathrm{key}}$.
\qed
\end{definition}

We use $U_{\mathcal{M}}$ to denote ``the uniform distribution over a message space $\mathcal{M}$,''
i.e.,
to denote
a finite probability space in $\PS(\mathcal{M})$ such that
\[
  U_{\mathcal{M}}(m)=\frac{1}{\#\mathcal{M}}
\]
for every $m\in\mathcal{M}$.
Note that $U_{\mathcal{M}}$ is a \emph{computable} finite probability space
since
every rational is computable.
Based on Theorems~\ref{independence-independence} and \ref{ind-vL}
we can show Theorems~\ref{PS-AR} and \ref{PS-AR2} below,
which characterize
the notion of perfect secrecy equivalently in terms of
the notions of the independence of ensembles and Martin-L\"of $P$-randomness relative to an oracle.

\begin{theorem}[New equivalent characterizations of perfect secrecy I]\label{PS-AR}
Let $\mathcal{M}$, $\mathcal{K}$, and $\mathcal{C}$ be alphabets.
Let $\Pi=(P_{\mathrm{key}},\mathsf{Enc},\mathsf{Dec})$ be an encryption scheme
over a message space $\mathcal{M}$, a key space $\mathcal{K}$, and a ciphertext space $\mathcal{C}$.
Then the following conditions are equivalent to one another.
\begin{enumerate}
  \item The encryption scheme $\Pi$ is perfectly secret.
  \item For every $P_{\mathrm{msg}}\in\PS(\mathcal{M})$ and every ensemble $\alpha$
    for $P_{\mathrm{msg}}\times P_{\mathrm{key}}$,
    the ensembles $M_{\Pi}(\alpha)$ and $C_{\Pi}(\alpha)$ are independent.
  \item For every $P_{\mathrm{msg}}\in\PS(\mathcal{M})$ there exists an ensemble $\alpha$
    for $P_{\mathrm{msg}}\times P_{\mathrm{key}}$
    such that the ensembles $M_{\Pi}(\alpha)$ and $C_{\Pi}(\alpha)$ are independent.
  \item For every computable $P_{\mathrm{msg}}\in\PS(\mathcal{M})$ and
    every ensemble $\alpha$ for $P_{\mathrm{msg}}\times P_{\mathrm{key}}$ it holds that
    $M_{\Pi}(\alpha)$ is Martin-L\"of $P_{\mathrm{msg}}$-random relative to $C_{\Pi}(\alpha)$.
  \item For every computable $P_{\mathrm{msg}}\in\PS(\mathcal{M})$
    there exists an ensemble $\alpha$ for $P_{\mathrm{msg}}\times P_{\mathrm{key}}$ such that
    $M_{\Pi}(\alpha)$ is Martin-L\"of $P_{\mathrm{msg}}$-random relative to $C_{\Pi}(\alpha)$.
  \item For every ensemble $\alpha$ for $U_{\mathcal{M}}\times P_{\mathrm{key}}$ it holds that
    $M_{\Pi}(\alpha)$ is Martin-L\"of $U_{\mathcal{M}}$-random relative to $C_{\Pi}(\alpha)$.
  \item There exists an ensemble $\alpha$ for $U_{\mathcal{M}}\times P_{\mathrm{key}}$ such that
    $M_{\Pi}(\alpha)$ is Martin-L\"of $U_{\mathcal{M}}$-random relative to $C_{\Pi}(\alpha)$.\qed
\end{enumerate}
\end{theorem}

\begin{theorem}[New equivalent characterizations of perfect secrecy II]\label{PS-AR2}
Let $\mathcal{M}$, $\mathcal{K}$, and $\mathcal{C}$ be alphabets.
Let $\Pi=(P_{\mathrm{key}},\mathsf{Enc},\mathsf{Dec})$ be an encryption scheme
over a message space $\mathcal{M}$, a key space $\mathcal{K}$, and a ciphertext space $\mathcal{C}$.
Suppose that $P_{\mathrm{key}}$ is computable.
Then the following conditions are equivalent to one another.
\begin{enumerate}
  \item The encryption scheme $\Pi$ is perfectly secret.
  \item For every computable $P_{\mathrm{msg}}\in\PS(\mathcal{M})$ and
    every ensemble $\alpha$ for $P_{\mathrm{msg}}\times P_{\mathrm{key}}$ it holds that
    $C_{\Pi}(\alpha)$ is
    Martin-L\"of $C_{\Pi}(P_{\mathrm{msg}}\times P_{\mathrm{key}})$-random relative to $M_{\Pi}(\alpha)$.
  \item For every computable $P_{\mathrm{msg}}\in\PS(\mathcal{M})$
    there exists an ensemble $\alpha$ for $P_{\mathrm{msg}}\times P_{\mathrm{key}}$ such that
    $C_{\Pi}(\alpha)$ is
    Martin-L\"of $C_{\Pi}(P_{\mathrm{msg}}\times P_{\mathrm{key}})$-random relative to $M_{\Pi}(\alpha)$.
  \item For every ensemble $\alpha$ for $U_{\mathcal{M}}\times P_{\mathrm{key}}$ it holds that
    $C_{\Pi}(\alpha)$ is Martin-L\"of $C_{\Pi}(U_{\mathcal{M}}\times P_{\mathrm{key}})$-random relative to $M_{\Pi}(\alpha)$.
  \item There exists an ensemble $\alpha$ for $U_{\mathcal{M}}\times P_{\mathrm{key}}$ such that
    $C_{\Pi}(\alpha)$ is Martin-L\"of $C_{\Pi}(U_{\mathcal{M}}\times P_{\mathrm{key}})$-random relative to $M_{\Pi}(\alpha)$.\qed
\end{enumerate}
\end{theorem}

Note that in Theorem~\ref{PS-AR}
the encryption scheme $\Pi=(P_{\mathrm{key}},\mathsf{Enc},\mathsf{Dec})$ is \emph{quite arbitrary}
and therefore the finite probability space $P_{\mathrm{key}}$ is not required to be computable at all, in particular.
In contrast, the computability of $P_{\mathrm{key}}$ is required in Theorem~\ref{PS-AR2}.
Note, however,
that the finite probability space $P_{\mathrm{key}}$,
which serves as a ``probability distribution'' over key space $\mathcal{K}$,
is normally \emph{computable} in modern cryptography.

In the cryptographic community, the notion of perfect secrecy for an encryption scheme
is said to imply that
even if the eavesdropper has \emph{infinite computing power},
she cannot obtain any information about a message
from the corresponding ciphertext.
We may interpret the conditions (iv)--(vii) of Theorem~\ref{PS-AR}
as implying
this situation in a certain sense
since they state that
the randomness of a message $M_{\Pi}(\alpha)$
cannot be
reduced
even by
Martin-L\"of tests of unlimited computing power
with a complete reference to the corresponding ciphertext $C_{\Pi}(\alpha)$ as side information.

Now, in order to prove Theorems~\ref{PS-AR} and \ref{PS-AR2}, we need the following lemma.

\begin{lemma}\label{PS_lemma}
Let $\mathcal{M}$, $\mathcal{K}$, and $\mathcal{C}$ be alphabets.
Let $\Pi=(P_{\mathrm{key}},\mathsf{Enc},\mathsf{Dec})$ be an encryption scheme
over a message space $\mathcal{M}$, a key space $\mathcal{K}$, and a ciphertext space $\mathcal{C}$.
Then the following conditions are equivalent
to one another.
\begin{enumerate}
  \item The encryption scheme $\Pi$ is perfectly secret.
  \item For every computable $P_{\mathrm{msg}}\in\PS(\mathcal{M})$
    it holds that the random variables $M_{\Pi}$ and $C_{\Pi}$ are independent
    on $P_{\mathrm{msg}}\times P_{\mathrm{key}}$.
  \item The random variables $M_{\Pi}$ and $C_{\Pi}$ are independent
    on $U_{\mathcal{M}}\times P_{\mathrm{key}}$.
\end{enumerate}
\end{lemma}

\begin{proof}
The implication (i) $\Rightarrow$ (ii) is obvious.
Since $U_{\mathcal{M}}$ is computable, the implication (ii) $\Rightarrow$ (iii) is also obvious.
Thus, we show the implication (iii) $\Rightarrow$ (i) in what follows.
For that purpose, we first note that
the following hold for every $P_{\mathrm{msg}}\in\PS(\mathcal{M})$, $m\in\mathcal{M}$, and $c\in\mathcal{C}$:
\begin{equation}\label{PSparts}
\begin{split}
  &(P_{\mathrm{msg}}\times P_{\mathrm{key}})(M_{\Pi}=m)=P_{\mathrm{msg}}(m),\\
  &(P_{\mathrm{msg}}\times P_{\mathrm{key}})(C_{\Pi}=c)=\sum_{m'\in\mathcal{M},\,k\in\mathcal{K}}P_{\mathrm{msg}}(m')P_{\mathrm{key}}(k)\jump{\mathsf{Enc}(m',k)=c},\\
  &(P_{\mathrm{msg}}\times P_{\mathrm{key}})(M_{\Pi}=m\;\&\;C_{\Pi}=c)
  =P_{\mathrm{msg}}(m)\sum_{k\in\mathcal{K}}P_{\mathrm{key}}(k)\jump{\mathsf{Enc}(m,k)=c},
\end{split}
\end{equation}
where $\jump{\mathsf{Enc}(m,k)=c}:=1$ if $\mathsf{Enc}(m,k)=c$ holds and $\jump{\mathsf{Enc}(m,k)=c}:=0$
otherwise.

Suppose that the random variables $M_{\Pi}$ and $C_{\Pi}$ are independent
on $U_{\mathcal{M}}\times P_{\mathrm{key}}$.
It follows from \eqref{PSparts} that
\begin{equation}\label{not-depend-on-m}
  \sum_{k\in\mathcal{K}}P_{\mathrm{key}}(k)\jump{\mathsf{Enc}(m,k)=c}
  =\frac{1}{\#\mathcal{M}}\sum_{m'\in\mathcal{M},\,k\in\mathcal{K}}P_{\mathrm{key}}(k)\jump{\mathsf{Enc}(m',k)=c}
\end{equation}
for every $m\in\mathcal{M}$ and $c\in\mathcal{C}$.
Note that the left-hand side of \eqref{not-depend-on-m} is independent of $m$.
Let $P_{\mathrm{msg}}$ be an arbitrary finite probability space in $\PS(\mathcal{M})$.
For each $m\in\mathcal{M}$ and $c\in\mathcal{C}$ we see that
\begin{align*}
  (P_{\mathrm{msg}}\times P_{\mathrm{key}})(C_{\Pi}=c)
  &=\sum_{m'\in\mathcal{M}}(P_{\mathrm{msg}}\times P_{\mathrm{key}})(M_{\Pi}=m'\;\&\;C_{\Pi}=c) \\
  &=\sum_{m'\in\mathcal{M}}P_{\mathrm{msg}}(m')\sum_{k\in\mathcal{K}}P_{\mathrm{key}}(k)\jump{\mathsf{Enc}(m',k)=c} \\
  &=\sum_{m'\in\mathcal{M}}P_{\mathrm{msg}}(m')\sum_{k\in\mathcal{K}}P_{\mathrm{key}}(k)\jump{\mathsf{Enc}(m,k)=c} \\
  &=\sum_{k\in\mathcal{K}}P_{\mathrm{key}}(k)\jump{\mathsf{Enc}(m,k)=c},
\end{align*}
where the second and third equalities follow from \eqref{PSparts} and \eqref{not-depend-on-m}, respectively.
It follows from \eqref{PSparts} that the random variables $M_{\Pi}$ and $C_{\Pi}$ are
independent
on
$P_{\mathrm{msg}}\times P_{\mathrm{key}}$.
Since $P_{\mathrm{msg}}$ is an arbitrary finite probability space in $\PS(\mathcal{M})$,
we have that the encryption scheme $\Pi$ is perfectly secret.
This completes the proof.
\end{proof}

Then, on the one hand, the proof of Theorem~\ref{PS-AR} is given as follows.

\begin{proof}[Proof of Theorem~\ref{PS-AR}]
Note that $M_{\Pi}(P_{\mathrm{msg}}\times P_{\mathrm{key}})=P_{\mathrm{msg}}$
for every $P_{\mathrm{msg}}\in\PS(\mathcal{M})$,
and $U_{\mathcal{M}}$ is computable.
Thus, the theorem follows from Theorems~\ref{independence-independence} and \ref{ind-vL}
using Lemma~\ref{PS_lemma}.
Note here that we do not
need
to require the computability of
the finite probability space $C_{\Pi}(P_{\mathrm{msg}}\times P_{\mathrm{key}})$
when we apply Theorem~\ref{ind-vL}
to the finite probability space $P_{\mathrm{msg}}\times P_{\mathrm{key}}$ on $\mathcal{M}\times\mathcal{K}$
and the random variables $M_{\Pi}$ and $C_{\Pi}$ on $\mathcal{M}\times\mathcal{K}$.
\end{proof}

On the other hand, the proof of Theorem~\ref{PS-AR2} is given as follows.

\begin{proof}[Proof of Theorem~\ref{PS-AR2}]
Since $P_{\mathrm{key}}$ is computable,
$C_{\Pi}(P_{\mathrm{msg}}\times P_{\mathrm{key}})$ is also computable
for every computable $P_{\mathrm{msg}}\in\PS(\mathcal{M})$.
Note also that $U_{\mathcal{M}}$ is computable.
Hence, the theorem follows from Theorem~\ref{ind-vL} using Lemma~\ref{PS_lemma}.
\end{proof}

\subsection{Application to repeated tossing of a biased coin}
\label{AFBC}

Arora and Barak~\cite[Lemma 7.12]{AB09} describes
a
method
for simulating one
tossing
of a biased coin
by repeated
tossing
of a fair coin,
in the case where
the ``probability'' of getting a head of the biased coin is a
\emph{computable}
real.
Their method is simple and clever.
But their description of
their
method is made
\emph{in the terminology of the conventional probability theory},
and
seems
vague and
incomplete from the measure-theoretic point of view
on which the conventional probability theory relies.
In this subsection, we adapt
their
method to our framework
for generating an ensemble for an arbitrary
\emph{computable}
finite probability space
on $\{0,1\}$,
given a Martin-L\"of random infinite binary sequence
as an \emph{oracle}.
Our method so obtained is simple and appealing to intuition, but rigorous.

To describe our method, we first introduce
some notation:
Let $\alpha\in\XI$.
For any $n\in\N^+$,
we denote by
$\alpha(n;\infty)$
the infinite binary sequence obtained from $\alpha$ by eliminating the first~$n-1$ elements
in
$\alpha$.
Thus,
for example,
we have that $\alpha=(\rest{\alpha}{n})\alpha(n+1;\infty)$
for every $n\in\N$.

Now,
let $P$ is
an arbitrary
computable finite probability space on $\{0,1\}$.
Let $\rho$ be an infinite binary sequence such that $0.\rho$ is a base-two expansion of
the real $P(0)\in[0,1]$.
Since $P(0)$ is a computable real,
note that
the infinite binary sequence $\rho$ is also computable.
Let $\alpha$ be an arbitrary Martin-L\"of random infinite binary sequence.
Given the infinite binary sequence $\alpha$ as an oracle,
by the following procedure,
we calculate each bit of an infinite binary sequence $\beta$,
which depends on both $P$ and $\alpha$,
one by one in a sequential order from the top
of $\beta$.

Initially, one sets $\omega:=\alpha$ and $i:=1$.
One then
repeats the following
three steps
\emph{forever}:
\begin{description}
\item[Step 1.] One finds the least $k\in\N^+$ such that $\omega(k)\neq\rho(k)$.
\item[Step 2.] One sets $\beta(i):=\omega(k)$.
\item[Step 3.] One updates $\omega$ and $i$ by $\omega:=\omega(k+1;\infty)$ and $i:=i+1$. 
\end{description}

Note that (i)~every Martin-L\"of random infinite binary sequence is not computable and
(ii)~$\gamma(n+1;\infty)$ is Martin-L\"of random
for every Martin-L\"of random infinite binary sequence $\gamma$ and $n\in\N$.
Thus, since $\rho$ is computable,
it follows that
$\alpha(n+1;\infty)\neq\rho$ for every $n\in\N$.
Therefore, in the above procedure,
one certainly finds the least $k\in\N^+$ such that $\omega(k)\neq\rho(k)$, whenever executing \textbf{Step~1}.
Hence,
by the above procedure,
one certainly calculates each bit of the infinite binary sequence $\beta$
one by one in a sequential order from the top
of $\beta$.
We use $\alpha_P$ to denote this infinite binary sequence $\beta$.
We remark that,
since the above procedure is effective,
the argument so far is \emph{formally} summarized as the statement that
for every computable finite probability space $P$ on $\{0,1\}$
there exists an oracle Turing machine $M$ such that
for every Martin-L\"of random infinite binary sequence $\alpha$
it holds that, on every input $n\in\N^+$,
$M$ eventually halts and outputs $\alpha_P(n)$ relative to $\alpha$.

We can then prove Theorem~\ref{biased-coins-from-ML-random} below,
which shows that the infinite binary sequence $\alpha_P$ obtained by the above
simple
procedure
is an ensemble for $P$.

\begin{theorem}[Repeated tossing of a biased coin]
\label{biased-coins-from-ML-random}
Let $P\in\PS(\{0,1\})$. Suppose that $P$ is computable.
For every Martin-L\"of random infinite binary sequence  $\alpha$,
it holds that $\alpha_P$ is an ensemble for $P$.
\end{theorem}

\begin{proof}
For any $\sigma\in\{0,1\}^+$, we use $\mathrm{Flip}(\sigma)$
to denote the finite binary sequence which is obtained from $\sigma$
by flipping the last bit of $\sigma$.
Let $\rho$ be an infinite binary sequence such that $0.\rho$ is a base-two expansion of the real $P(0)$.
We define $P_0$ and $P_1$
as the sets $\{\mathrm{Flip}(\rest{\rho}{n})\mid \rho(n)=1\text{ \& }n\in\N^+\}$
and $\{\mathrm{Flip}(\rest{\rho}{n})\mid \rho(n)=0\text{ \& }n\in\N^+\}$, respectively.
Then we see that
\begin{equation}\label{mLoP0=ssigP02-abssig=srn=12-n=0r=P0-biased-coin}
  \sum_{\sigma\in P_0}2^{-\abs{\sigma}}
  =\sum_{\rho(n)=1}2^{-n}=0.\rho=P(0),
\end{equation}
where the second sum is over all $n\in\N^+$ such that $\rho(n)=1$.
Similarly, we see that
\begin{equation}\label{mLoP1=srn=2-n=1-srn12-n=1-P0=P1-biased-coin}
  \sum_{\sigma\in P_1}2^{-\abs{\sigma}}
  =\sum_{\rho(n)=0}2^{-n}=1-\sum_{\rho(n)=1}2^{-n}=1-P(0)=P(1).
\end{equation}
Moreover, it is easy to see that $P_0\cup P_1$ is a prefix-free subset of $\X$.
Furthermore, note
that both $P_0$ and $P_1$ are
r.e.~subsets of $\X$
since $P(0)$ is a computable real.

Now, let $\alpha$ be an arbitrary Martin-L\"of random infinite binary sequence.
Let us
assume contrarily that $\alpha_P$ is not Martin-L\"of $P$-random.
Then there exists a Martin-L\"of $P$-test $\mathcal{S}\subset\N^+\times\X$
such that
\begin{equation}\label{betainosgmSn_biased-coin}
  \alpha_P\in\osg{\mathcal{S}_n}
\end{equation}
for every $n\in\N^+$.
For each $\sigma\in\{0,1\}^+$,
let $F(\sigma)$ be the set of all $\tau\in\X$ such that
$\tau$ is obtained
by replacing each occurrence of $0$ in $\sigma$, if exists, by some element of $P_0$ and
by replacing each occurrence of $1$ in $\sigma$, if exists, by some element of $P_1$.
Namely,
for each $\sigma\in\{0,1\}^+$,
we define $F(\sigma)$ as the set of all finite binary strings of the form $\tau_1 \tau_2 \dots \tau_L$
with $\tau_i\in\X$
such that for every $i=1,2,\dots,L$
it holds that $\tau_i\in P_0$ if $\sigma(i)=0$ and $\tau_i\in P_1$ otherwise,
where 
$L:=\abs{\sigma}$.
Then, for each $\sigma\in\{0,1\}^+$, we have that
\begin{equation}\label{mLoFslestFs2-at=Ps=BmPos-biased-coin}
  \mathcal{L}(\osg{F(\sigma)})
  =
  \sum_{\tau\in F(\sigma)}2^{-\abs{\tau}}=P(\sigma)=\Bm{P}{\osg{\sigma}},
\end{equation}
where the first equality follows from the prefix-freeness of $P_0\cup P_1$,
and the second equality follows from \eqref{mLoP0=ssigP02-abssig=srn=12-n=0r=P0-biased-coin} and
\eqref{mLoP1=srn=2-n=1-srn12-n=1-P0=P1-biased-coin}.
Also recall here that $\mathcal{L}$ denotes Lebesgue measure on $\XI$.
We then define $\mathcal{T}$ to be a subset of $\N^+\times \X$ such that
$\mathcal{T}_n=\bigcup_{\sigma\in\mathcal{S}_n} F(\sigma)$ for every $n\in\N^+$.
Note here that, for each $n\in\N^+$,
$\lambda\notin\mathcal{S}_n$ since $\Bm{P}{\osg{\mathcal{S}_n}}<2^{-n}<1$.
Then, since
$\mathcal{S}_n$ is a prefix-free subset of $\X$ for every $n\in\N^+$
and $P_0\cup P_1$ is a prefix-free subset of $\X$,
it is easy to see that $\mathcal{T}_n$ is a prefix-free subset of $\X$ for every $n\in\N^+$.
For each $n\in\N^+$, we
also
see that
\[
  \mathcal{L}(\osg{\mathcal{T}_n})
  \le\sum_{\sigma\in\mathcal{S}_n}\mathcal{L}(\osg{F(\sigma)})
  =\sum_{\sigma\in\mathcal{S}_n}\Bm{P}{\osg{\sigma}}
  =\Bm{P}{\osg{\mathcal{S}_n}}<2^{-n},
\]
where the first equality follows from \eqref{mLoFslestFs2-at=Ps=BmPos-biased-coin} and
the second equality follows from the prefix-freeness of $\mathcal{S}_n$.
Moreover, since $\mathcal{S}$, $P_0$, and $P_1$ are all r.e., we see that $\mathcal{T}$ is also r.e.
Thus, $\mathcal{T}$ is a Martin-L\"of test.

On the other hand,
from the definition of $\alpha_P$
it is easy to see
that, for every $n\in\N^+$, if $\alpha_P\in\osg{\mathcal{S}_n}$ then $\alpha\in\osg{\mathcal{T}_n}$.
Thus, it follows from \eqref{betainosgmSn_biased-coin}
that $\alpha\in\osg{\mathcal{T}_n}$ for every $n\in\N^+$.
Hence, $\alpha$ is not Martin-L\"of random.
Thus, we have a contradiction, and
therefore $\alpha_P$ is Martin-L\"of $P$-random.
This completes the proof.
\end{proof}

Let
$U$
be a finite probability space on $\{0,1\}$ such that $U(0)=U(1)=1/2$, and
let us consider an infinite sequence $\alpha\in\XI$ of outcomes
which is being generated by infinitely repeated trials described
by the finite probability space $U$ on $\{0,1\}$.
\emph{Intuitively},
this
infinite binary sequence $\alpha$
is considered as
an infinite sequence
generated
by infinitely repeated
tossing
of a fair coin.
Thesis~\ref{thesis} treats this intuition in a rigorous manner.
Actually,
according to Thesis~\ref{thesis}, we have that $\alpha$ is an ensemble for $U$,
that is,
$\alpha$ is a Martin-L\"of random infinite binary sequence.
Thus, via the
infinitely repeated
trials
above,
we obtain each element of the Martin-L\"of random infinite binary sequence $\alpha$
one by one in a sequential order from the top of
$\alpha$.
Now, suppose that we are given a computable finite probability space $P\in\PS(\{0,1\})$.
Then, applying the above
simple
procedure
to the elements of $\alpha$ one by one in a sequential order from the top of $\alpha$,
we can \emph{effectively} convert $\alpha$ into a Martin-L\"of $P$-random infinite
binary
sequence, i.e.,
the infinite
binary
sequence $\alpha_P$.

In this subsection we require that
the finite probability space $P$ on $\{0,1\}$ is computable when we are generating an ensemble for $P$.
This requirement can be considered to be \emph{natural} due to the following reason:
To perform the above procedure the computability of the infinite binary sequence $\rho$ is
\emph{indispensable},
and the computability of $\rho$ is equivalent to the computability of $P$.
Thus, the computability of $P$ is natural
and does not impose a restriction on the applicability of
the above procedure and Theorem~\ref{biased-coins-from-ML-random}.

\subsection{Application to quantum mechanics}

The notion of probability plays a crucial role in quantum mechanics.
It appears in quantum mechanics as the so-called \emph{Born rule}, i.e.,
the \emph{probability interpretation of the wave function} \cite{D58,vN55,LL77,NC00}.
In modern mathematics which describes quantum mechanics, however,
probability theory means nothing other than \emph{measure theory},
and therefore any \emph{operational characterization of the notion of probability} is still missing
in quantum mechanics.
In this sense, the current form of quantum mechanics is considered to be \emph{imperfect}
as a physical theory which must stand on operational means.

As a \emph{major application} of our framework developed in the present paper,
we developed a framework
of
an operational refinement of quantum mechanics with respect to
measurements
in a series of works~\cite{T14CCR,T15Kokyuroku,T15WiNF-Tadaki_rule,T16QIT35,T18arXiv}.
In these works
we first presented
a \emph{refinement} of the Born rule as an alternative rule to it,
based on the notion of ensemble,
for aiming at
making quantum mechanics
\emph{operationally perfect}.
Namely,
we used the notion of ensemble
to state
the refinement of the Born rule
for specifying the property of the results of quantum measurements \emph{in an operational way}.
We then presented an operational refinement of the Born rule for mixed states, as an alternative rule to it,
based on the notion of ensemble.
In particular, we gave a precise definition for the notion of \emph{mixed state},
using the notion of ensemble.
We then showed that all of
the refined rules of the Born rule for both pure states and mixed states
can be derived
from a \emph{single} postulate
on quantum measurements,
called the \emph{principle of typicality}, in a unified manner.
We did this from the point of view of the \emph{many-worlds interpretation of quantum mechanics}
\cite{E57}.
Furthermore,
we made an application of our framework to the BB84 quantum key distribution protocol
in order to demonstrate how properly our framework works in \emph{practical} problems
in quantum mechanics, based on the principle of typicality.
See the work~\cite{T18arXiv} for the detail
of the development
of the framework based on the principle of typicality.

In the work~\cite{T20Kokyuroku}, we made an application of our framework
based on the principle of typicality
to the argument of Bell's inequality versus quantum mechanics
to refine it.
Specifically, in the work~\cite{T20Kokyuroku}, we \emph{refined} and \emph{reformulated}
the argument of Bell's inequality versus quantum mechanics,
which is described
in
Section~2.6 ``EPR and the Bell inequality'' of Nielsen and Chuang~\cite{NC00}.
Thus,
on the one hand,
we refined and reformulated the \emph{assumptions of local realism}
to lead to Bell's inequality,
in terms of
our framework on the
operational characterization of the notion of probability
developed by
the present paper.
On the other hand,
we refined and reformulated
the corresponding argument of quantum mechanics to violate Bell's inequality,
based on the principle of typicality~\cite{T18arXiv}.
The results of the work~\cite{T20Kokyuroku} demonstrates further that
both the framework on the operational characterization of the notion of probability
developed by the present paper and
the framework on the refinement of quantum mechanics based on the principle of typicality~\cite{T18arXiv}
work properly in \emph{practical} problems.

\section{Concluding remarks}
\label{Concluding}

In this paper we have developed an operational characterization of the notion of probability.
As the first step of the research of this line,
we have considered only the case of finite probability space, where the sample space is finite, for simplicity.
As the next step of the research,
it is natural to consider the case of \emph{discrete probability space},
where the sample space is \emph{countably infinite}.
Actually,
in this case
we can develop a framework for the operational characterization of the notion of probability
in \emph{almost the same manner} as the case of finite probability space.
The detail is reported in the sequel to this paper,
Tadaki~\cite{T19arXiv}.

\begin{acks}
This work was partially supported
by JSPS KAKENHI Grant Numbers JP24540142, JP15K04981, JP18K03405.
This work was
partially done
while the author was visiting the Institute for Mathematical Sciences,
National University of Singapore in 2014.
\end{acks}

\end{document}